\documentclass[a4,12pt]{amsart}

\usepackage[utf8]{inputenc}
\usepackage{amsfonts}
\usepackage{amssymb}
\usepackage{amsmath}
\usepackage{comment}
\usepackage[dvipsnames]{xcolor}
\usepackage{hyperref} 

\hypersetup{
    colorlinks,
    linkcolor=CadetBlue,
    citecolor=CadetBlue,
    urlcolor=CadetBlue
}

\usepackage{graphicx}

\numberwithin{equation}{section}
\usepackage[leftcaption]{sidecap}
\sidecaptionvpos{figure}{m}
\usepackage{accents}

\definecolor{textcol}{rgb}{0.37,0,0.57}
\usepackage{tikz}

\usepackage[left=1.9cm, top=2.1cm,bottom=2.1cm,right=1.9cm]{geometry}

\makeatletter
\renewcommand{\boxed}[1]{\text{\fboxsep=.2em\fbox{\m@th$\displaystyle#1$}}}
\makeatother

\newtheorem{thm}{Theorem}[section]

\newtheorem{cor}[thm]{Corollary}
\newtheorem{lemma}[thm]{Lemma}
\newtheorem{df}[thm]{Definition}
\newtheorem{proposition}[thm]{Proposition}

\newtheorem{thmalpha}{Theorem}

\newtheorem{rem}{Remark}

\newtheorem{assu}{Assumption}

\newcommand{\MassP}{\mathcal{P}_{\rm m}}

\newcommand{\const}{{\rm C}} % generic constant

\title[The largest fragment in a fragmentation process]{The largest fragment in self-similar fragmentation processes of positive index}
\author[Dyszewski]{Piotr Dyszewski}
\author[Johnston]{Samuel G. G. Johnston}
\author[Palau]{Sandra Palau}
\author[Prochno]{Joscha Prochno}

\subjclass{Primary: 60J27, 60J80, 60G55.}
\keywords{Fragmentation process, Self-similar, Spines, Branching Random Walk}

\begin{document}

\maketitle

\begin{abstract}
	We study a self-similar fragmentation process with dislocation measure $\nu$ 
	and self-similarity index $\alpha > 0$. 
	Let $e^{-m_t}$ denote the size of the largest fragment at time $t \geq 0$. 
	For dislocation measures satisfying a regularity condition of the form 
	$\nu(1 - s_1 > \delta) = \delta^{-\theta} \ell(1/\delta)$ with 
	$\theta \in [0,1)$ and slowly varying $\ell$, 
	we prove almost sure convergence
	\[
		\lim_{t \to \infty} (m_t - g(t)) = 0,
	\]
	where $g(t) = (\log t - (1 - \theta) \log \log t + f(t))/\alpha$, and 
	$f(t) = o(\log \log t)$ is a lower order correction that can be described explicitly in terms of $\ell$ and $\theta$. 
	Our results sharpen substantially the best prior result on general self-similar 
	fragmentation processes, due to Bertoin, which states that 
	$m_t = (1+o(1)) \log (t)/\alpha$. 
\end{abstract}

%%%%%%%%%%%%%%%%%%%%%%%%
\section{Introduction}
%%%%%%%%%%%%%%%%%%%%%%%%

    Fragmentation processes are stochastic models describing how objects break 
    into smaller parts over time. Examples arise in settings ranging from the 
    disintegration of solids under mechanical impact or stress~\cite{kooij2021explosive, ghabache2016frozen}, the breakup of aggregates in turbulent flows~\cite{babler2008modelling} 
    to catastrophic disruption in the solar system~\cite{ramesh2015review}, 
    tectonic plate motion~\cite{mallard2016subduction}, and balloon 
    popping~\cite{moulinet2015popping}.

    Early studies of fragmentation can be traced back to investigations of the statistical 
    properties of what is now known as a Poisson point process~\cite{lienau1936random}. 
    From a practical perspective, researchers were primarily interested in the 
    size distribution resulting from a single 
    fragmentation event. The fragment masses of bombs and shells were observed 
	to follow a stretched exponential 
	distribution~\cite{Mott2006}. 
	In contrast, the crushing of fused silica plates results in a size distribution that can 
	be modeled as a mixture of exponential distributions~\cite{grady1985geometric}.

    Fragmentation phenomena have a long history in both analysis and probability~\cite{fu, BD1, BD2, evans1998stationary, aldous1998standard}. 
    On the probabilistic side, already Kolmogorov \cite{kolmogorov} proposed 
    a stochastic model leading to the log-normal distribution of particle sizes 
    in fragmentation processes. 
    On the analytic side, the study of coagulation–fragmentation equations goes 
    back to the classical work of Smoluchowski \cite{Smoluchowski1916}, and was 
    developed in a rigorous PDE framework in the 1980s and 1990s, see for instance 
    Stewart \cite{Stewart1989} and Escobedo, Mischler and Perthame 
    \cite{EscobedoMischlerPerthame2003}. 
    Structural aspects of random partitions are related to Kingman's theory of 
    exchangeable partitions \cite{Kingman1978} and to continuum random trees 
    \cite{Aldous1991}. 
    The works of Bertoin \cite{bertoin1, bertoin2, bertoin3} and Berestycki \cite{Berestycki} 
    in the 2000s provided a systematic probabilistic framework for self-similar 
    fragmentation processes. Over the past two decades, it has been demonstrated that fragmentation processes are related to various topics in the mathematical literature, such as phylogenetic trees~\cite{haas2008continuum}, random planar maps~\cite{bertoin2018random,bertoin2018martingales}, and the additive coalescent~\cite{miermont2001ordered}.
    
	Formally, a fragmentation process is a stochastic process $(X(t))_{t \geq 0}$ 
	taking values in the state space
	\begin{align*}
		\MassP:=\left\{ \mathbf{s}=(s_1,s_2,\ldots) \,:\, 
		s_1 \geq s_2 \geq \ldots \geq 0, \, \sum_{i \geq 1} s_i \leq 1\right\};
	\end{align*}
    	this set shall be endowed with the topology of pointwise convergence.
	We think of $X(t) = (X_1(t), X_2(t),\ldots)$ as a list (in decreasing order) 
	of the random sizes of the fragments in the system at time 
	$t \geq 0$. In other words, $X_1(t)$ is the size of the 
	largest fragment, $X_2(t)$ is the size of the second largest fragment, and so on; 
	note that a portion of the initial mass may be lost during fragmentation, 
	corresponding to $\sum_{i \geq 1} X_i(t) < 1$.

	In developing a mathematically tractable theory of fragmentation processes,  
	we typically assume the process is Markovian, meaning that, roughly speaking,  
	each fragment evolves independently of the others in the system.  
	This is formalized in the so-called \textbf{finite-activity case},  
	where fragmentation processes are described using a rate function  
	$r \colon (0,1] \to [0,\infty)$ and a probability measure $\tilde{\nu}$ on $\MassP$,  
	as follows.
	We begin at time zero with a single fragment of size $1$.  
	This fragment has an exponential lifetime with rate $r(1)$, and then splits into  
	a (possibly infinite) collection of fragments with sizes $(s_1, s_2, \ldots)$,  
	distributed according to $\tilde{\nu}$.
	The process then evolves according to the following rule:  
	independently of all other fragments, any fragment of size $u > 0$  
	has an exponential lifetime with rate $r(u)$,  
	and upon death it splits into fragments of sizes $(us_1, us_2, \ldots)$,  
	where $(s_1, s_2, \ldots)$ is sampled independently from $\tilde{\nu}$.

%%%%%%%%%%%%%%%%%%%%%%%%%%%%%%%%%%%%%%%%%%%
\subsection{Self-similar fragmentation processes and the infinite-activity case}
%%%%%%%%%%%%%%%%%%%%%%%%%%%%%%%%%%%%%%%%%%%

	In the present article, we focus on \textbf{self-similar fragmentation processes},  
	which are characterised by the property that a fragment of size $u \in (0,1]$  
	breaks at the specific rate $r(u) = \lambda u^\alpha$, where $\alpha \in \mathbb{R}$ is the index of  
	self-similarity, and $\lambda > 0$ is a fixed constant.
	Self-similarity is a natural assumption in fragmentation models,  
	as the splitting dynamics depend on relative sizes,  
	making the fragment structure statistically identical across scales,  
	a behaviour observed in many physical systems.
	As an intuitive example, consider the fragmentation of onion pieces on a (stochastic) 
	cook’s chopping board.  
	One may argue that this is a fragmentation process with index $\alpha = 1/3$:  
	the rate at which an onion of mass $m$ is chopped is proportional to its diameter,  
	which scales as $m^{1/3}$.
	(Of course, any good chef is likely to chop several pieces at once,  
	while our formal stochastic models exclude simultaneous fragmentations.)
	Different values of the exponent $\alpha$ can model different physical systems:  
	the case $\alpha = 1$ typically arises in one-dimensional settings,  
	and the case $\alpha = 2/3$ has even been proposed as a 
	model for polymer degradation~\cite{BEE}.

	\begin{figure}
    	\includegraphics[width=0.49\textwidth]{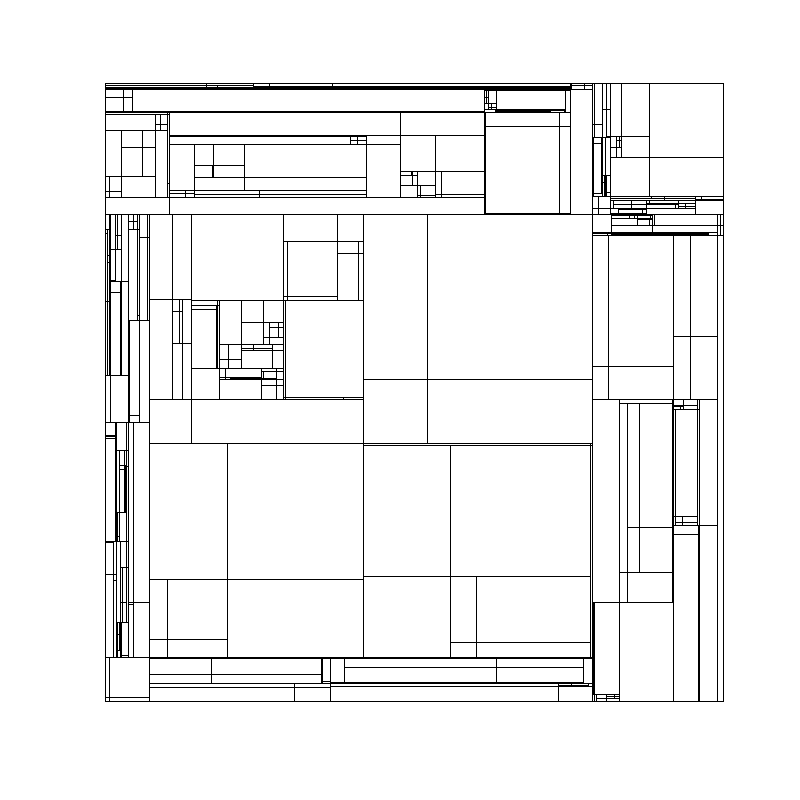}
    	\includegraphics[width=0.49\textwidth]{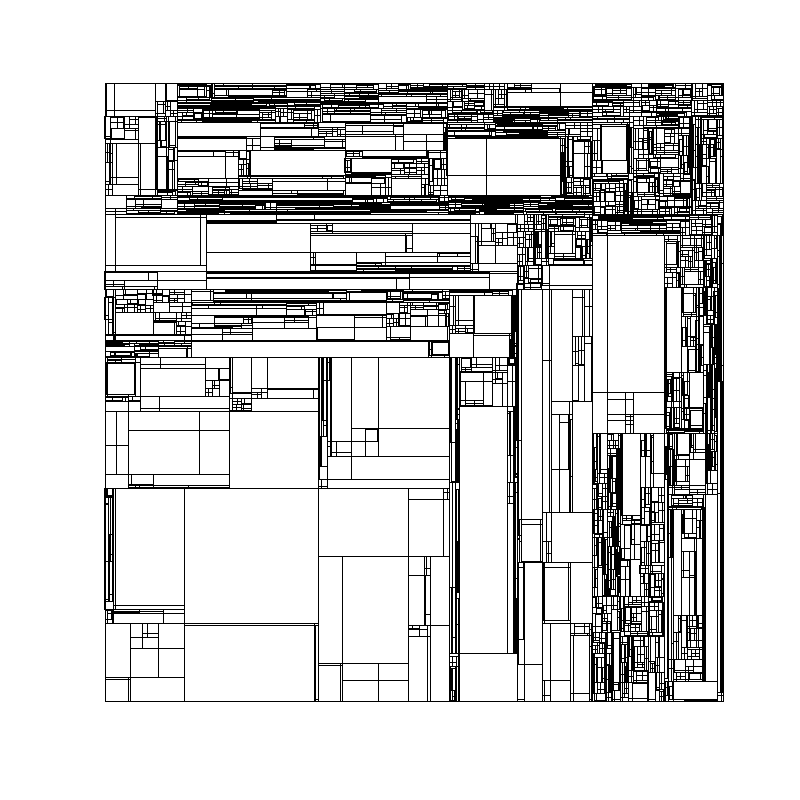}
    	\caption{Simulation of a fragmentation process of a unit square. 
    	Each rectangle in the system splits into four new rectangles in 
    	proportion $U$ of its width and height,
    	where $U$ is uniformly distributed on the unit interval.
    	The dislocation measure 
    	$\nu(\mathrm{d}\mathbf{s})$
    	is therefore given via 
    	$\int_{\MassP} f(\mathbf{s}) \nu(\mathrm{d}\mathbf{s}) = 
    	2\int^1_{1/2} f\left(x^2, x(1-x), (1-x)x, (1-x)^2 , 0, 0, \ldots\right) \mathrm{d}x
    	$ for $f \colon \MassP \to \mathbb{R}$. 
    	Both images show simulations of the process at time $t = 5$. 
    	On the left, $\alpha = 1/10$; on the right, $\alpha = -1/10$. 
    	For simulation purposes, we froze fragments whose area is smaller than $10^{-4}$.}
    	\label{fig:sym}
	\end{figure}
	
	When $\alpha > 0$, self-similarity implies that larger fragments 
	tend to break more quickly than smaller ones. 
	As a result, the size distribution of fragments becomes regularized over time 
	(see the first picture in Figure~\ref{fig:sym}). 
	This effect may be seen as serendipitous for anyone slicing onions and 
	wishing for uniformly small pieces: 
	larger onion fragments are more likely to be struck by the cook’s knife, 
	and thus more likely to break and 
	“catch up” in size with the smaller fragments.

    Of course, the $\alpha < 0$ case is characterised by the opposite behaviour: 
    smaller fragments break more quickly than larger ones 
    (see the second picture in Figure~\ref{fig:sym}). 
    Examples of such processes include ice floe breakup due to thermal stress 
    or droplet breakup in turbulent flow. 
    In this regime, the system exhibits loss of mass through the formation of dust, 
    a phenomenon already observed in early works on fragmentation 
    (see, e.g., \cite{Filippov1961,bertoin3}). 
    More precisely, after initially preserving a positive proportion of mass, 
    the entire system turns into dust in finite time. That is,
    \begin{equation*}
    \zeta = \inf\{ t \geq 0 : X_1(t) = 0 \}
    \end{equation*}
    is almost surely finite~\cite{bertoin3}, 
    where $X_1(t)$ denotes the size of the largest fragment at time $t$. 
    The asymptotic behaviour of $\zeta$~\cite{haas2023tail}, 
    as well as the structure of the fragmentation process near the extinction time~\cite{GH1,GH2}, 
    are now well understood.
    
	In the finite-activity case, we can package the rate $\lambda$ and the probability measure $\tilde{\nu}$ into a finite 
	measure  
    by setting $\nu := \lambda \tilde{\nu}$.
	The measure $\nu$ is called the \textit{dislocation measure}; 
	roughly speaking, the dislocation measure specifies 
	the rate at which blocks split. The dynamics are then 
	summarised by the following definition.

	\begin{df}
		In a self-similar fragmentation process with index $\alpha \in \mathbb{R}$ 
		and dislocation measure $\nu$, fragments behave independently of one another, 
		and at rate $u^\alpha \nu (\mathrm{d}\mathbf{s})$, each fragment of size $u$ 
		breaks into fragments of sizes $u\mathbf{s}:=(us_1,us_2,\ldots)$. 
	\end{df}

	One of the major technical achievements in the theory of fragmentation processes is
	the development of the \textbf{infinite-activity case}, 
	which allows for the \emph{continuous crumbling}
	of fragments over time. This is modeled by permitting the dislocation measure
	$\nu$ to be an infinite measure, although concentrated mainly on configurations
	$(s_1, s_2, \ldots)$ in $\MassP$ where $1 - s_1$ is small. As a result, events may occur
	at an arbitrarily high rate, but most involve only a tiny amount of mass being
	chipped off each fragment to form smaller ones.

	A key technical condition required to ensure the well-posedness of such processes is
	\begin{align}
		\label{eq:intcond}
		\int_{\MassP} (1 - s_1)\, \nu(\mathrm{d}\mathbf{s}) < \infty.
	\end{align}
	In other words, the expected instantaneous loss of mass from a given fragment must be finite,
	where we naturally associate the parent fragment with its largest child fragment after 
	a fragmentation event.
	For a detailed construction and further introduction to the topic, 
	we refer the reader to~\cite{bertoinbook}.

	\begin{figure}
    	\includegraphics[width=0.49\textwidth]{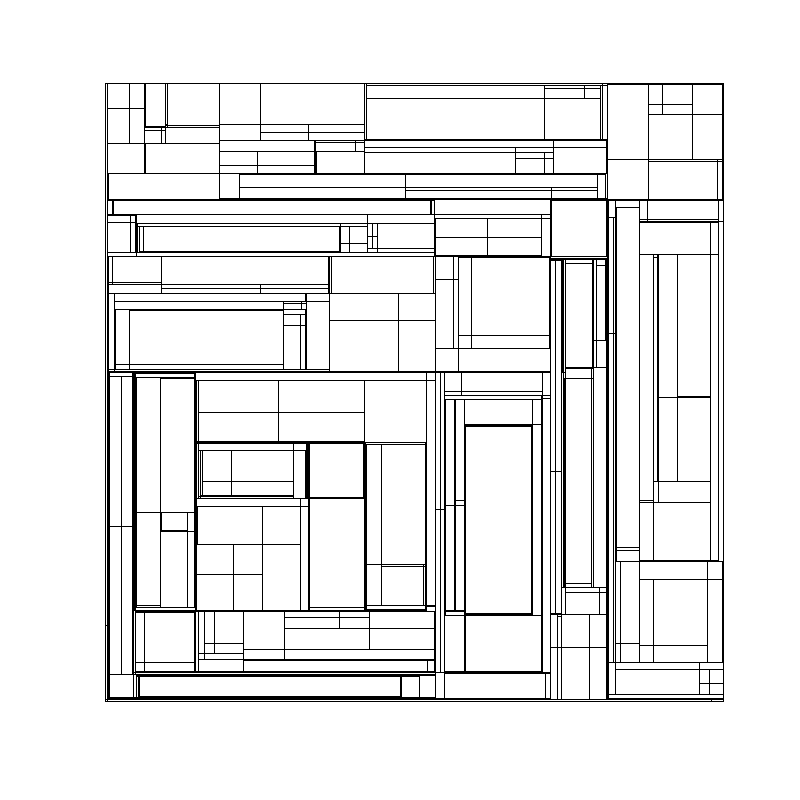}
    	\includegraphics[width=0.49\textwidth]{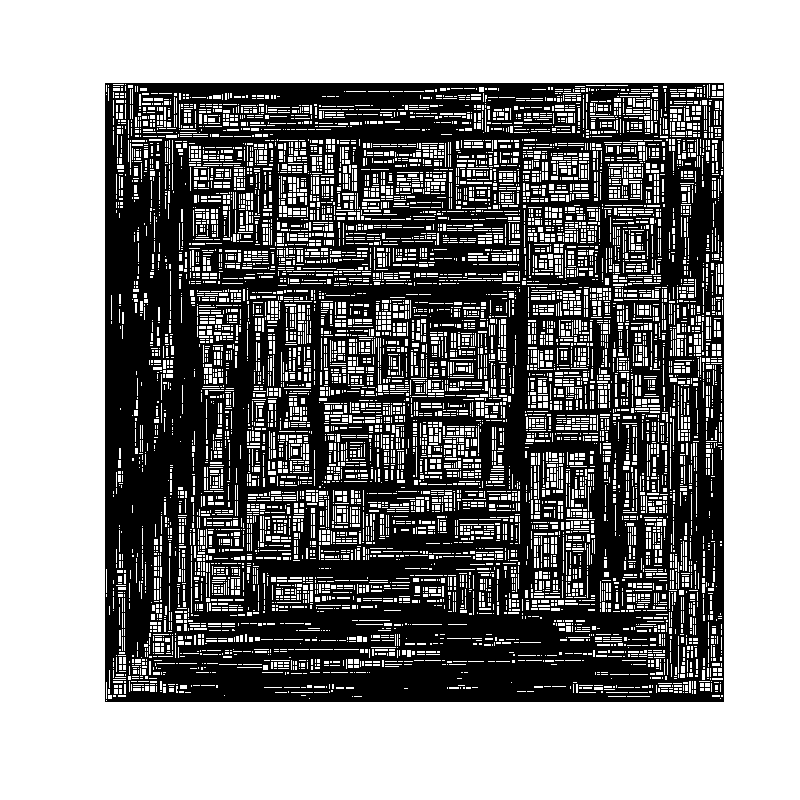}
    	\caption{Simulation of a fragmentation process of a unit square with 
    	infinite dislocation measure given via
    	$ \int_{\MassP} f(\mathbf{s}) \nu(\mathrm{d}\mathbf{s}) = \int_{1/2}^1 2/(x(1-x)) 
    	f\left(x^2,x(1-x),x(1-x),(1-x)^2, 0, \ldots \right) \mathrm{d}x$
    	at time $t=20$.
    	The approximation is made by truncating $\nu$ to the set $\{s_1<1-\delta\}$ 
    	and taking $\delta=10^{-15}$.
    	On the left $\alpha=1$ and on the right $\alpha=-1$.
    	Since the dislocation measure is infinite, the fragments lose very small bits of the fragments. 
    	These small bits are visible on 
    	the simulation as thick segments between larger fragments.
    	} 
    	\label{fig:sym2}
	\end{figure}

	Finally, through the article, in order to simplify the discussion, we will assume that the 
	process is \textbf{conservative}. In other words, the measure $\nu$ is supported on the 
	subset of $\MassP$ consisting of sequences $(s_1,s_2,\ldots)$ such that 
	$\sum_{i \geq 1}s_i  =1$. In the case where the index $\alpha $ of self-similarity is 
	positive, this property is sufficient to ensure that the sum of masses of all fragments 
	in the system is equal to $1$ at all times. We emphasise that this assumption is in force 
	for the remainder of the article. 
	To avoid trivialities, we also assume throughout that 
	$\nu$ is a nonzero measure that has no mass on $\{(1,0,\ldots)\}$. 

%%%%%%%%%%%%%%%%%%%%%%%%%%%%%%%%%%%%%%%%%%%%%%%%%%%%%
\subsection{The size of the largest fragment in self-similar fragmentation processes of positive index}
%%%%%%%%%%%%%%%%%%%%%%%%%%%%%%%%%%%%%%%%%%%%%%%%%%%%%

	Given a self-similar fragmentation process with positive index of self-similarity, 
	it is natural to enquire about the size of the largest fragment in the 
	system at large times $t\geq 0$, as it describes the speed at which the object that 
	undergoes the fragmentation breaks down. 

	The most general result on this front is due to Bertoin~\cite{bertoin3}, who
	studies self-similar fragmentation processes whose dislocation measures 
	satisfy the integrability condition
	\begin{align} \label{eq:intcond2}
    	\int_{\MassP} \sum_{i \geq 1} s_i \, |\log(s_i)| \, \nu(\mathrm{d}\mathbf{s}) < \infty,
    \end{align}
    in addition to \eqref{eq:intcond}.
    Bertoin shows that under \eqref{eq:intcond} and \eqref{eq:intcond2},  
    if $X_1(t) = e^{ - m_t }$ is the size of the largest fragment 
	in the system at time $t \geq 0$, then
    \begin{align} \label{eq:bertoin}
        \lim_{t \to \infty} \frac{m_t}{\log t} = \frac{1}{\alpha} \quad \text{almost surely}.
	\end{align}
	In the recent work \cite{DGJPS} (whose authors overlap those of the present article), 
	it is shown that it is possible to obtain a significant sharpening of \eqref{eq:intcond} in 
	a certain special case. Namely, the authors of \cite{DGJPS} consider the simplest 
	possible fragmentation mechanism, in which each fragment of size $m \in (0,1]$ 
	splits into $k$ equally sized pieces (of size $m/k$) at rate $m^\alpha$.
	In other words, this corresponds to the case where
	\begin{align} \label{eq:kfrag}
    		\nu := \delta_{(1/k, \ldots, 1/k, 0, \ldots)} \quad \text{(with $k$ copies of $1/k$)},
	\end{align}
	where $\delta_x$ refers to the Dirac mass at $x$.

	Note that, starting with a single initial fragment of size $1$, 
	every fragment in the system at time $t$ has size $k^{-n}$ for some nonnegative integer $n$.
	It is shown in~\cite{DGJPS} that if $k^{-\tilde{m}_t}$ is the size of the largest fragment 
	in the system at time $t \geq 0$, then for ‘most’ large times $t$, we have
	\begin{align} \label{eq:conc}
    		\tilde{m}_t = \left\lceil 
		\kappa \left( \log t - \log \log t + \log(\alpha) \right) \right\rceil
    		\quad \text{with } \kappa = 1/(\alpha \log k),
	\end{align}
	with high probability as $t \to \infty$.
	By ‘most’ large times $t$, we mean those $t$ for which the quantity inside the ceiling 
	function is not too close to an integer; see~\cite[Theorem A]{DGJPS} for a precise statement.
	The authors also prove a statement similar to~\eqref{eq:conc} for the behaviour of 
	the exponent $\tilde{M}_t$ 
	governing the size $k^{-\tilde{M}_t}$ of the \emph{smallest} fragment at large times.

	The main tool used in \cite{DGJPS} is to relate the process to an expanding 
	branching random walk 
	by considering the genealogy of the fragment process.
	Namely, for each $n$, one may consider the $k^{n}$ different times at which a fragment of size 
	$k^{-n}$ fragmented into $k$ fragments of size 
	$k^{-(n+1)}$. 
	One can think of each of these $k^n$ times as the position of a particle in 
	generation $n$, and thereby relate the fragmentation process to a time-inhomogeneous 
	discrete-time branching random walk, where the jumps become exponentially 
	larger as the generations progress. 
	The exponential growth of the jumps accounts for the concentration of \eqref{eq:conc}. 
	Analogous ideas are used in the present paper, though the techniques used here are
	substantially more involved, 
	with infinite-activity L\'evy processes playing the role of the random walk.

%%%%%%%%%%%%%%%%%%%%%%%%%%%%%%%%%%%%%%%%%%%%%%%%%%%%%
\subsection{Main result}
 %%% (before it said main results (plural))
 %%%%%%%%%%%%%%%%%%%%%%%%%%%%%%%%%%%%%%%%%%%%%%%%%%%%%

	In the article at hand, we prove fine estimates for the size $X_1(t)=e^{-m_t}$ of the 
	largest fragment in the system at large times $t$ in a broad class of self-similar 
	fragmentation processes with conservative dislocation measures. These estimates take the form
	\begin{align*}
		m_t = \frac{1}{\alpha}\left[\log t - (1 - \theta) \log \log t \right] + O(\log \log t), 
		\qquad \text{almost surely for large $t$},
	\end{align*}
	where the $O(\log \log t)$ terms are explicit. Here $\theta \geq 0$ is a parameter 
	that will describe how `erosive' the dislocation measure is. 
	More precisely, we consider the broad class of dislocation measures 
	satisfying the mild regularity condition that
	the dislocation measure 
	takes the form
	\begin{equation}\label{eq:1:crumblingpre}
		\nu( \{\mathbf{s}\in \MassP: 1- s_1 > \delta \} ) = \delta^{ - \theta} \ell(1/\delta),
	\end{equation}
	for some $\theta \geq 0$ and $\ell \colon [1,\infty) \to [0, +\infty)$ 
	slowly varying at infinity;
	recall that the latter means that for any positive $\lambda$, 
	$\ell(\lambda x)/\ell(x) \to 1$ as $x \to \infty$. 
	We call $\theta$ in \eqref{eq:1:crumblingpre} the \textbf{crumbling index}.
	Note that for $\nu$ satisfying \eqref{eq:1:crumblingpre}, $\theta$ is the infimum over
	$\phi \geq 0$ such that the strong version
	\begin{align*}
    		\int_{\MassP} (1 - s_1)^\phi \, \nu(\mathrm{d}\mathbf{s}) < \infty
   	\end{align*}
    	of \eqref{eq:intcond} holds. Consequently, in order for \eqref{eq:1:crumblingpre} 
	to hold, we must have $\theta \in [0,1]$.
	Note that~\eqref{eq:1:crumblingpre} is a natural condition controlling 
    	the strength of the singularity of 
	$\nu$ at $\mathbf{s} = (1, 0, \ldots)$ (should it have one).
	It appeared previously in the literature, for example in~\cite{haas2008continuum}, 
	where it was used to identify fragmentation trees as limits of discrete
 	fragmentation trees associated with a homogeneous fragmentation process,
	or in~\cite{haas2025ancestral} to describe the limiting 
	behaviour of the ancestor-counting process.

	This framework encompasses all finite-activity fragmentation processes:
	\begin{align*}
		\lambda = \nu(\MassP) < \infty \iff 
		\text{\eqref{eq:1:crumblingpre} holds with $\theta = 0$ and $\ell(x) \uparrow \lambda<\infty$ as $x \uparrow \infty$}.
	\end{align*}
	We emphasise that  
	$\theta = 0$ does not imply that the process has finite activity. 
	For instance, we could have $\theta =0$ and $\ell(x) = \log(1+x)$. 
    It is also possible to have $\theta = 1$, though we rule this edge case out.

    After assuming the mild regularity condition in \eqref{eq:1:crumblingpre}, we make two further assumptions: one is a technical condition that is only relevant in the $\theta = 0$ case with infinite activity, and the other is a non-lattice condition.
    Wrapping our assumptions together, for the remainder of the article we will operate under the following conditions:
    
   \begin{assu} \label{assu:1}
    We assume that the dislocation measure $\nu$ satisfies the following three conditions:
    \begin{enumerate}
        \item The singularity of the dislocation measure $\nu$ satisfies
            \begin{equation}\label{eq:1:crumbling}
        		\nu( \{\mathbf{s}\in \MassP: 1- s_1 > \delta \} ) = \delta^{ - \theta} \ell(1/\delta),
        	\end{equation}
            for some $\theta \in [0,1)$ and $\ell$ slowly varying at infinity. 
        \item If \eqref{eq:1:crumbling} holds with $\theta = 0$ and the process has infinite activity 
            (which implies $\ell(y) \to \infty$ as $y \to \infty$), we make the additional assumption 
            that there exists a second slowly varying function $\ell_0$ such that 
            \begin{equation}\label{eq:1:fortheta0}
              \ell(x) \sim \int_1^x \frac{\ell_0(y)}{y} \, \mathrm{d}y \qquad x \to \infty.
        	\end{equation}
        \item The process is non-lattice, in that:
        	\begin{equation} \label{eq:nonlattice}
            		\begin{array}{c}\text{There does not exist } a > 0 \text{ such that } \\ \nu \text{ is concentrated on } 
        		\left\{ \mathbf{s} \in \MassP : -\log(s_i) \in a\mathbb{Z} \text{ for all } i \geq 1 \right\}.\end{array}
        	\end{equation}
    \end{enumerate}
    \end{assu}

    The non-lattice condition in \eqref{eq:nonlattice} is automatically satisfied whenever $\nu$ is an infinite measure. 
    The example in~\eqref{eq:kfrag} is, of course, lattice. With that said, we are now ready to state our main result.

    \begin{thmalpha} \label{thm:ia}
    	Let $(X(t))_{t\geq 0}$ be a self-similar fragmentation process with index $\alpha >0$ and with 
    	dislocation measure $\nu$ satisfying Assumption \ref{assu:1} with corresponding $\theta\in[0,1)$ and slowly varying $\ell$. 
    	Let $e^{ -m_t}$ be the size of the largest fragment in the system at time $t$. Then there is a deterministic function of the form
    	\begin{align} \label{eq:ia}
    	   g(t) := \frac{1}{\alpha} 
    		\big[  \log t - (1-\theta) \log \log t 
    		+ h(t) \big],
            % + o(1),
    	\end{align}
        with an explicit correct $h(t)$ of the order $ o(\log \log t)$,
        such that
    	\begin{align*}
    		\lim_{t \to \infty}(m_t-g(t)) = 0
    	\end{align*}
    	almost surely.
	\end{thmalpha}

    In fact, we can characterise the correction $h(t)$ explicitly:
    \begin{enumerate}
        \item In either the finite-activity case or the infinite-activity case with $\theta \in (0,1)$, $h(t)$ is given by 
        \begin{align*}
            h(t) = \log \ell( \log t) + \log \alpha + \log \Gamma(1-\theta) + (1-\theta)\log(1-\theta),    
        \end{align*}
        where $\ell$ and $\theta$ are as in \eqref{eq:1:crumbling}. 

        \item In the case where $\theta = 0$ but we have infinite activity (i.e., in which Assumption \ref{assu:1} part 2 is relevant), the form of $h$ is slightly more complicated, and can be described implicitly in terms of $\ell$ (as in \eqref{eq:1:crumbling}) and $\ell_0$ (as in \eqref{eq:1:fortheta0}):
        \begin{align*}
            h(t) = \log \alpha - \log G(\log t),
        \end{align*}
        where $G$ is defined sequentially through other slowly varying functions $L$ and $J$ by 
        \begin{align} \label{eq:complicated0}
            1 \sim G(h)L(hG(h)), \text{ where } L(h) := \ell(hJ(h)), \text{ and where } J(h)^{-1}\ell_0(hJ(h)) \sim 1.
        \end{align}
    \end{enumerate}

	We emphasise that our main result holds under the crumbling 
	condition \eqref{eq:1:crumbling} and the non-lattice 
	condition \eqref{eq:nonlattice}, regardless of whether \eqref{eq:intcond2} is or is not satisfied. 

	Let us pick out the special case of our main result corresponding to finite activity with rate $\lambda$.
	Setting $\theta = 0$ and $\lim_{x \to \infty} \ell(x) = \lambda$, 
	the deterministic function $g$ occurring in \eqref{eq:ia} simply reads
	\begin{align} \label{eq:fa}
		g(t) =  \frac{1}{\alpha} \big[ \log t - \log \log t + \log \alpha + \log \lambda + o(1) \big], 
		\qquad t\geq 0.
	\end{align}
	Setting $\lambda = 1$, we see that the previous function coincides, up to a ceiling function, 
	with \eqref{eq:conc} (which is a lattice case, and therefore does not satisfy the non-lattice condition \eqref{eq:nonlattice}).
	We believe it would not be difficult to adapt our methods to handle the lattice case.

% %%%%%%%%%%%%%%%%%%%%%%%%%%%%%%%
% \subsection{Notation?}
% %%%%%%%%%%%%%%%%%%%%%%%%%%%%%%%
 
% Given functions $F_1,F_2:(0,\infty) \to \mathbb{R}$ we will use the notation 
% \begin{align*}
% F_1(q) \sim F_2(q)
% \end{align*}
% to denote that $\lim_{ q \to \infty} F_2(q)/F_1(q) = 1$.

%%%%%%%%%%%%%%%%%%%%%%%%%%%%%%%
\subsection{Overview}
%%%%%%%%%%%%%%%%%%%%%%%%%%%%%%%

	The remainder of the article is structured as follows.

    In Section~\ref{sec:further}, we discuss further the technical framework we work under (comparing in particular the conditions  \eqref{eq:intcond2} and \eqref{eq:1:crumbling}), introduce the notion of spines for fragmentation processes, and give an overview of our proof.
    
	In Section~\ref{sec:CMJ}, we review the basics of Crump–Mode–Jagers branching processes and their limit theory. 
	We then discuss a natural projection of finite-activity fragmentation processes onto Crump–Mode–Jagers branching processes, 
	and thereafter show that even in the infinite-activity case, one can obtain a Crump–Mode–Jagers branching process 
	by viewing an infinite-activity self-similar fragmentation process at certain discrete stochastic times.

	In Section~\ref{sec:levy}, we review the spine construction of a tagged fragment in a self-similar fragmentation 
	process due to Bertoin~\cite{bertoin1}. 
	Thereafter, we use tools from Jain and Pruitt~\cite{JP} to study the rare events where a 
	spine particle does not lose too much size over a large time interval.

	In Section~\ref{sec:upper}, we use a many-to-one type formula to prove an upper bound for 
	the expected number of fragments of size greater than $e^{-h}$ at time $t$, 
	and, for a given $h$, study the asymptotics of the times $t(h)$ at which this expectation becomes $o(1)$.

	In Section~\ref{sec:lower}, we use our comparison with a Crump–Mode–Jagers branching process, 
	together with a delicate spine argument, 
	to show that with high probability there exist fragments of reasonably large size at certain times.

	In the final Section~\ref{sec:proof}, we tie our work together to prove our main result, Theorem~\ref{thm:ia}.

%%%%%%%%%%%%%%%%%%%%%%%%%%%%%%%%%%%%%%%%%%%%%%%%
\section{Further discussion and proof ideas} \label{sec:further}
%%%%%%%%%%%%%%%%%%%%%%%%%%%%%%%%%%%%%%%%%%%%%%%%

\subsection{Discussion and comparison of technical conditions}

\subsubsection{The technical condition in the $\theta=0$ case}

    In the second part of Assumption \ref{assu:1}, we assumed in the $\theta =0$ case that the slowly varying function $\ell(x)$ takes the form $\ell(x) \sim \int_1^x \ell_0(y)/y \mathrm{d}y$ for large $x$. Briefly, this is because in our proofs for certain tail bounds for L\'evy processes developed in Section \ref{sec:levy}, we need to show that the derivative of the function $f(x) = x^\theta \ell(x)$, with $\ell(\cdot)$ slowly varying (cf.\ \eqref{eq:1:crumbling}), takes the form $f'(x) = x^{\theta-1} \tilde{\ell}(x)$ for some slowly varying function $\tilde{\ell}(x)$. If $f$ is monotone, this is automatically the case when $\theta \in (0,1)$ (and in fact in this case $\tilde{\ell}(x)$ is asymptotically equivalent to $\theta\ell(x)$), but need not hold in the case $\theta = 0$. Indeed, consider that the function  	
    \[
        \ell(x) = \exp\left\{ \int_e^x \frac{\sin(y)}{y \log(y)} \, \mathrm{d}y \right\}.
	\]
    is slowly varying, but its derivative does not take the form $x^{-1}\tilde{\ell}(x)$ for a slowly varying $\tilde{\ell}(x)$. To circumvent these possible pathologies, in the special case where $\theta = 0$ but the dislocation measure is infinite, 
	we will impose the additional condition that $\ell(x) \sim \int_1^x \ell_0(y)/y ~\mathrm{d}y$ for a second slowly varying function $\ell_0$, i.e.\ the second part of Assumption \ref{assu:1}. This assumption is not particularly restrictive. Indeed, we note that \color{black} by the smooth variation theorem~\cite[Theorem 1.3.3]{bingham1989regular}, any slowly varying function is asymptotically equivalent to a smooth function.   

\subsubsection{The edge case $\theta = 1$}

    Let us touch on an edge case that we do not consider in the present article. Namely, it is possible for 
	\eqref{eq:1:crumblingpre} to hold with $\theta = 1$, 
	provided $\ell$ satisfies the integrability condition $\int_0^1 \ell(1/t)/t~\mathrm{d}t < \infty$. (One such example might be $\ell(x) = 1/(1 - \log(x))^2$.) 
	Since this is an edge case whose treatment would demand substantial adaptations of our tools, 
	for the sake of concision we rule this special case out of our consideration.

\subsubsection{The relationship between \eqref{eq:intcond2} and \eqref{eq:1:crumbling}}

    As noted above, \eqref{eq:1:crumbling} implies \eqref{eq:intcond}. 
	We now turn to discussing the relationship between \eqref{eq:intcond2} and \eqref{eq:1:crumbling}.
	Since~\eqref{eq:1:crumbling} is a local condition (i.e.\ depends only on the distribution
	of the largest fragment $s_1$ of $\mathbf{s}$ sampled according to $\nu$) and
	\eqref{eq:intcond2} is a global one (it depends on the whole sequence $\mathbf{s}$), it
	transpires that for any $\theta \in (0,1)$, there exist examples of dislocation measures with 
	crumbling index $\theta$ for which \eqref{eq:intcond2} does hold, and examples for which \eqref{eq:intcond2} does not hold:

    \begin{rem}[Under \eqref{eq:1:crumbling}, \eqref{eq:intcond2} may or may not hold]
    	First let $\theta \in (0,1)$. Let $f_{\theta}(t) = t^{-(1+\theta)} \mathrm{1}_{\{t \leq 1/2\}}$, 	
    	and consider a dislocation measure associated with the following dynamics: at rate $f_{\theta}(t)\mathrm{d}t$, 
    	we break into two pieces of sizes $(1-t,t)$, with $1-t$ being the bigger piece. 
    	Then, the crumbling exponent of this dislocation measure is clearly $\theta$,  since
    	\begin{align*}
        		\nu( \{ \mathbf{s}\in \MassP: 1 - s_1 > \delta \} ) = 
    		\int_{0}^{1}  t^{-(1+\theta)}\mathrm{1}_{\{\delta<t \leq 1/2\}} \mathrm{d}t = 
    		\theta^{-1} (\delta^{ - \theta} - 2^{ \theta}) = \delta^{-\theta}\ell(1/\delta).
       	\end{align*}
    	However, \eqref{eq:intcond2} clearly holds since 
    	$\int_0^{1/2} \left( |\log(1-t)|(1-t) + |\log(t)| t \right) t^{-(1+\theta)}\mathrm{d}t < \infty$. 
    
    	Conversely, let $(c_i)_{i \geq 2}$ be a (deterministic) collection of nonincreasing nonnegative real 
    	numbers satisfying $\sum_{i \geq 2} c_i = 1$ but $\sum_{i \geq 2} |\log(c_i)| c_i = + \infty$ 
    	(e.g., take $c_i = C/(i |\log(i)|^{3/2})$ for some constant $C>0$). 
    	Now consider a dislocation measure  associated with the following dynamics: 
    	at rate $f_{\theta}(t)\mathrm{d}t$, we break into infinitely many pieces, 
    	the largest of which has size $S_1 = 1-t$, and for $i \geq 2$, 
    	the $i^{\text{th}}$ largest of which has size $s_i = t c_i$. 
    	Similarly as before, the crumbling exponent is $\theta$. 
    	However \eqref{eq:intcond2} fails since 
    	\[
        \int_{\MassP} \sum_{i \geq 1 } |\log(s_i)| s_i \nu(\mathrm{d}\mathbf{s}) 
    	= \int_0^{1/2} \hspace{-.2cm}\big(|\log(1-t)|(1-t) + \sum_{i \geq 2} (c_i t) |\log(c_it) |\big) t^{-(1+\theta)} \mathrm{d}t = +\infty.
        \]
    
        In the case $\theta = 0$, set $f_0(t) = ((1 - \log t)^2 t)^{-1}$ and repeat the above calculation in the two cases. 
        
    	In summary, for any $\theta \in [0,1)$ condition \eqref{eq:intcond2} may or may not hold for a 
    	dislocation measure $\nu$ with crumbling index $\theta$. 
    \end{rem}

%%%%%%%%%%%%%%%%%%%%%%%%%%%%%%%%%%%%%%%
\subsection{Spine fragments}
%%%%%%%%%%%%%%%%%%%%%%%%%%%%%%%%%%%%%%%

    Over the course of proving our results, we rely on spine methods, 
    which allow us to make sense of ‘following’ a distinguished particle 
    through the jump evolution of the system. 
    More explicitly, we appeal to the size-biased (many-to-one) construction 
    introduced in~\cite{bertoin1} (see also the monograph~\cite{bertoinbook}), 
    which consists in following a size-biased fragment forward in time. 
    Such spine or many-to-one techniques are by now standard in the theory 
    of branching processes.
    
    Briefly, we follow a single fragment, 
	and when it undergoes a fragmentation, we choose a size-biased fragment to continue 
	following after the split.
	This turns out to be a sound notion even in the infinite-activity case.

	The pertinent fact here is that the size of this size-biased fragment process
    if $(Z_t)_{t \geq 0}$ is a self-similar Markov process
	(see Section~\ref{sec:one} for a detailed explanation) and thus has the Lamperti representation
	\begin{align} \label{eq:distid}
	    (Z_t)_{t \geq 0} \overset{d}{=} \left( e^{- \xi_{\rho(t)}} \right)_{t \geq 0},
	\end{align}
	where $(\xi_t)_{t \geq 0}$ is a nondecreasing L\'evy process and $(\rho(t))_{t \geq 0}$ is a time change.
    The Laplace exponent of this L\'evy process is given in terms of the dislocation measure $\nu$ of the underlying fragmentation via the formula
        \begin{align*}
        \Phi(q) := \int_{\MassP} \left( 1 - \sum_{i \geq 1} s_i^{q+1} \right) \nu (\mathrm{d}\mathbf{s}).
    \end{align*}
    The time change associated with the process is given by
   	\begin{align} \label{eq:sclock}
		t = \int_0^{\rho(t)} e^{\alpha \xi_s} \, \mathrm{d}s, \quad t \geq 0.
	\end{align} 
    The L\'evy process $(\xi_t)_{t \geq 0}$ has finite expectation if and only if $\Phi'(0) <\infty$, or equivalently, if \eqref{eq:intcond2} holds. Under the condition \eqref{eq:intcond2}, Bertoin uses the L\'evy process to gain a rough idea of the behaviour of the largest particles in the system at time $t > 0$ as follows. 
	When $(\xi_t)_{t \geq 0}$ has finite expectation, its paths are easily controlled at a macroscopic scale,
	as there are no macroscopically significant jumps.
	Using this property, it is possible to show that the random variable
	\begin{align*} 
    		W_t := -\frac{1}{\alpha} \frac{\log Z_t}{\log t} 
    		\quad \text{satisfies} \quad \lim_{t \to \infty} W_t = 1 \quad \text{almost surely}.
	\end{align*}
	Since the size $X_1(t)$ of the largest fragment in the process at time $t$ is at least as large as $Z_t$,
	the size of the size-biased fragment, one can deduce that
	\begin{equation}\label{eq:x1limsup}
    		\limsup_{t \uparrow \infty} \frac{-\log X_1(t)}{\alpha \log t} \leq 1 \quad \text{almost surely}.
	\end{equation}
	In fact, it is possible to show that as $t \to \infty$, 
	\[
    		-\log X_1(t) \sim -\log Z_t \quad \text{almost surely},
	\]
	and deduce Bertoin’s result~\eqref{eq:bertoin} from this relation.

	When the integrability condition~\eqref{eq:intcond2} fails, however, 
	the subordinator $(\xi_t)_{t \geq 0}$ instead satisfies the property that 
	$\mathbb{E}[\xi_t] = +\infty$ for all $t > 0$.
	In this case, it is a consequence of the results of Caballero and Rivero~\cite{CR} 
	that $W_t$ converges in distribution to a nondegenerate random variable that is $\geq 1$ almost surely.
	In light of our main result, Theorem~\ref{thm:ia}, which says in particular that 
	\eqref{eq:x1limsup} holds, it follows that when \eqref{eq:intcond2} fails, 
	the size-biased fragment is no longer a good proxy for the size of the largest fragment in the process, in the sense that
	\[
		-\log X_1(t) \nsim -\log Z_t.
	\]

	Fortunately, it is possible to adapt the size-biased fragment idea to obtain a better proxy 
	for the size-biased fragment process. Namely, instead of choosing a size-biased child fragment 
	at each fragmentation event (interpreted, if necessary, continuously in time), 
	one can instead choose the \emph{largest} child fragment at each fragmentation event.
	We call this process the \textbf{greedy fragment process}.
	While we will not use the greedy fragment process in our proof of Theorem~\ref{thm:ia} 
	(preferring to rely on the more flexible size-biased fragment process, which can be used 
	in the context of a many-to-one formula), it is possible to use the greedy fragment process 
	to conclude that, provided \eqref{eq:intcond} holds, \eqref{eq:x1limsup} is valid
	regardless of whether or not \eqref{eq:intcond2} holds.

	Sketching the details of this calculation here, let $(G_t)_{t \geq 0}$ denote the size-process
	of the greedy fragment process. In analogy with~\eqref{eq:distid}, we have
	\begin{align*} 
    		(G_t)_{t \geq 0} \overset{d}{=} \left( e^{- \zeta_{\rho_\zeta(t)}} \right)_{t \geq 0},
	\end{align*}
	where $(\zeta_t)_{t \geq 0}$ is a nondecreasing L\'evy process with a different 
	Laplace mechanism from that of $(\xi_t)_{t \geq 0}$ (which can nonetheless be expressed in terms of $\nu$), 
	and the stochastic clock $(\rho_\zeta(t))_{t \geq 0}$ is defined as in~\eqref{eq:sclock}, 
	but with $(\zeta_s)_{s \geq 0}$ replacing $(\xi_s)_{s \geq 0}$ in that equation.
    In this case however, the greedy fragment L\'evy process $(\zeta_t)_{t \geq 0}$ has Laplace exponent
    \begin{align*}
        \Psi(q) = \int_{\MassP} (1 - s_1^q)\nu(\mathrm{d}\mathbf{s}). 
    \end{align*}
	While the L\'evy process $(\xi_t)_{t \geq 0}$ may or may not have finite expectation depending on 
	whether \eqref{eq:intcond2} holds, it is always the case, by virtue of the milder condition~\eqref{eq:intcond},
	that $(\zeta_t)_{t \geq 0}$ has finite expectation. Indeed,
    \begin{align*}
    \mathbb{E}[ \zeta_1 ] := \Psi'(0) = \int_{\MassP} |\log(s_1)|  \nu(\mathrm{d}\mathbf{s}) < \infty,
    \end{align*}
    where the latter inequality follows from \eqref{eq:intcond}. 

    In short, using the greedy fragment process instead of the size-biased fragment process, one may amend the sketch proof above to conclude that the log-size of the largest fragment is of the order $-\frac{1}{\alpha}\log t$, regardless of whether \eqref{eq:intcond2} does or does not hold.

%%%%%%%%%%%%%%%%%%%%%%%%%%%%%%%%%%%%%%%
\subsection{Proof outline} \label{sec:outline} % instead of `Idea of Proof'. 
%%%%%%%%%%%%%%%%%%%%%%%%%%%%%%%%%%%%%%%
    
    In this informal section we outline some of the key ideas of the proof. In order to keep the discussion here reasonably fluent, we will use the notation $e^{A(t,h)} \approx e^{B(t,h)}$ if $A(t,h)/B(t,h) \to 1$ as $t$ and/or $h$ are sent to infinity. 
    
    Before developing our fine results for L\'evy processes in Section \ref{sec:levy}, we begin in Section \ref{sec:CMJ} by studying the number of particles existing at any moment in time who have a size of the order $O(e^{-h})$, and show that the number of such particles is at least $e^{ (1+o(1))h}$. 
    For this task, we rely on the theory of Crump–Mode–Jagers branching processes.
	Of course, there are certain technical challenges associated with making sense of a ‘particle’ 
	in the infinite-activity case, as fragmentation events occur continuously in time, but we do so by viewing the particles within the process at certain genealogy-dependent times.
    
	Our work in this case is related to that of Bertoin and Martinez~\cite{BM}, who associate 
	the process of fragmentation with an energy cost, and estimate the amount of energy required 
	to break up material into smaller pieces.

    This guarantees that the process has $e^{(1+o(1))h}$ particles of size $O(e^{-h})$ will be used in a lower bound that will complement an upper bound that uses a first moment formula. In other words, in the remainder of this proof overview we will appeal to arguments in expectation, with the understanding that our bounds from Section \ref{sec:CMJ} will assure that expectation estimates correspond well with almost sure estimates.
        
    Recall that the process $(Z_t = e^{ - \xi_{\rho(t)}})_{t \geq 0}$ measures the size of a size-biased chosen fragment forwards in time. Accordingly, for suitable functions $f:(0,\infty) \to \mathbb{R}$ on the size of a fragment, we have the relation
    \begin{equation}\label{eq:2:Many2one0}
		\mathbb{E} \left[ \sum_{i \geq 1} f(X_i(t)) \right] 
		= \mathbb{E} \left[ \frac{f(Z_t)}{Z_t} \right],
	\end{equation}
    where the sum inside the expectation on the left is over the sizes $(X_1(t),X_2(t),\ldots)$ of the existing fragments at time $t$. Setting $f(x) := \mathrm{1} \{ x \geq e^{-h} \}$ to be the indicator function that the size of the fragment exceeds $e^{-h}$ at time $t$, we obtain the relation
	\begin{equation}\label{eq:2:Many2one00}
		\mathbb{E} \left[ \# \{ \text{Particles of size $\geq e^{-h}$ at time $t$} \}  \right] 
		= \mathbb{E}[ e^{ \xi_{\rho(t)} } \mathrm{1} \{ \xi_{\rho(t)} \leq h \} ] \approx  e^h \mathbb{P}[ \xi_{\rho(t)} \leq h ],
	\end{equation}
    where the final relation above follows from the heuristic that the expectation gets most of its contribution from events where $\xi_{\rho(t)}$ is close to but does not exceed $h$.

    Given a fixed $h$, we are interested in choosing the largest possible $t$ such that the quantity $e^h \mathbb{P}[ \xi_{\rho(t)} \leq h]$ is not too small. We begin in Section \ref{sec:levy} by studying the extreme lower tails of the (un-time-changed) L\'evy process $(\xi_t)_{t \geq 0}$. Using changes of measure and delicate analysis of slowly varying functions, we prove the main result of this section, which states that
    \begin{align*}
        \mathbb{P}[ \xi_t \leq 1 ] \approx  \exp \left\{ - t^{ \frac{1}{1-\theta} } L(t) \right\} \qquad \text{as $t \to \infty$},
    \end{align*}
    where $L(\cdot)$ is another slowly varying function that can be constructed explicitly from $\ell$ (see \eqref{eq:Lpuredef}).

Consider now the tail events of the time-changed process $\mathbb{P}[ \xi_{\rho(t)} \leq h ]$. By studying the delicate interplay between the tail events of the subordinator and the time change, we are able to prove that for large $h$ and very large $t$ (i.e., with $te^{- \alpha h}$ also large), we have the relation
\begin{align} \label{eq:relat}
\mathbb{P} \left[ \xi_{\rho(t)} \leq h \right] \approx  \exp \left\{ - C_2 (t e^{ - \alpha h} )^{\frac{1}{1-\theta}} L(t e^{ - \alpha h} ) \right\}.
\end{align} 
The constant $C_2 = C_2(\alpha,\theta)$ in \eqref{eq:relat} is given by $(\alpha/\theta)^{\theta/(1-\theta)}$, and for $\theta \in (0,1)$ is the solution of a variational problem
\begin{align*}
C_2 = C_2(\alpha,\theta) := \inf \left\{ \int_0^\infty \left(f(s) e^{ \alpha s} \right)^{\frac{1}{1-\theta}} \mathrm{d}s : f \text{ satisfies} \int_0^\infty f(s) \mathrm{d}s = 1 \right\},
\end{align*}
where the infimum is taken over all functions $f:[0,\infty) \to [0,\infty)$ with integral $1$. Implicit in our proof of \eqref{eq:relat} is the idea that if $f_*:[0,\infty) \to [0,\infty)$ is the minimiser of the variational problem, when $te^{-\alpha h}$ is very large, the tail event $\mathbb{P}[ \xi_{\rho(t)} \leq h]$ has most of its contribution due to the stochastic process $\xi_{\rho(t)}$ spending time of the order $t f_*(s) \mathrm{d}s$ in the height interval $[h-s,h-s+\mathrm{d}s]$. When $\theta = 0$, $C_2(\alpha,0) = 1$, and while the variational problem is not well posed, it is possible to interpret the solution to the variational problem as having its solution $f_*(s)\mathrm{d}s$ be the Dirac mass at zero. Accordingly, when $\theta = 0$, the tail event $\mathbb{P}[ \xi_{\rho(t)} \leq h]$ gets its leading order contribution from spending $(1-o(1))t$ time in a small height window $[h-\varepsilon,h]$, and time $o(1)t$ elsewhere. 
The proofs of the rigorous forms of the relation \eqref{eq:relat} with absolute bounds are highly technical, and occupy large amounts of Sections \ref{sec:upper} and Section \ref{sec:lower}. 

In any case, combining \eqref{eq:2:Many2one00} with \eqref{eq:relat} we see that 
\begin{align*}
	\mathbb{E} \left[ \# \{ \text{Particles of size $\geq e^{-h}$ at time $t$} \}  \right] \approx  e^{H(t,h)}
\end{align*}
where
\begin{align*}
H(t,h) := h - C_2(te^{-\alpha h})^{\frac{1}{1-\theta}} L(t e^{ - \alpha h} ). 
\end{align*} 
As mentioned above, by combining our expectation bounds with our estimates for the underlying Crump-Mode-Jagers process, we will see that our process is well captured by its first moment behaviour, in the sense that for large $t$ and $h$,
\begin{align*}
&H(t,h) \gg 0 \implies \text{there are particles of size $e^{-h}$ at time $t$ with high probability,}\\
&H(t,h) \ll 0 \implies \text{there are no particles of size $e^{-h}$ at time $t$ with high probability}.
\end{align*}
Thus, we are interested in finding a relationship between $t$ and $h$ such that $e^{H(t,h)}$ becomes large or small. 
Specifically, we find that if we choose a slowly varying function $G(\cdot)$ correctly and set 
\begin{align*}
F_\theta(h) := h^{1-\theta}e^{\alpha h} G(h),
\end{align*}
then with the right choice of $G$ we have the sharp transition
% where $C_3 := C_1^{-1}C_2^{\theta-1} = \alpha^{-\theta}(1-\theta)^{-(1-\theta)}\Gamma(1-\theta)^{-1}$, then by plugging $t = (1+\delta)F_\theta(h)$ into $H(t,h)$ and appealing to \eqref{eq:dbc01}, one can verify immediately that we the sharp behaviour
\begin{align*}
H( (1+\delta)F_\theta(h),h) \to 
\begin{cases}
- \infty \qquad :\text{$\delta > 0$},\\
+ \infty \qquad :\text{$\delta < 0$}.
\end{cases}
\end{align*}
In short, the last time at which particles of size $e^{-h}$ exist in the process takes the form $t = (1+o(1))F_\theta(h)$ almost surely. Our main result, Theorem \ref{thm:ia}, now follows from inverting the function $F_\theta$. Indeed, a brief calculation tells us that the inverse function of $F_\theta$ takes the form
\begin{align*}
F_\theta^{-1}(t) = \frac{1}{\alpha} \left\{ \log t - (1-\theta) \log \log t + (1-\theta)\log \alpha - \log G( \log t ) \right\} + O(\log \log t/\log t).
\end{align*}
This coincides with the statement of Theorem \ref{thm:ia}. The descriptions of the correction term $h(t)$ following the statement of Theorem \ref{thm:ia} follow from further reductions that can be made (relating the slowly varying function $G$ to $\ell$ via $L$) along the way.

% It transpires that in the finite-activity case or the infinite-activity case with $\theta = 0$, the function $G(t)$ reduces further, and we obtain the statement, which up to $o(1)$ terms coincides with the function $g(t)$ appearing in Theorem \ref{thm:ia}.

That completes our discussion of our proof outline. In the next section we begin working towards our proofs in earnest by relating our fragmentation process to a Crump-Mode-Jagers branching process.
	\begin{figure}
	\includegraphics[width=0.49\textwidth]{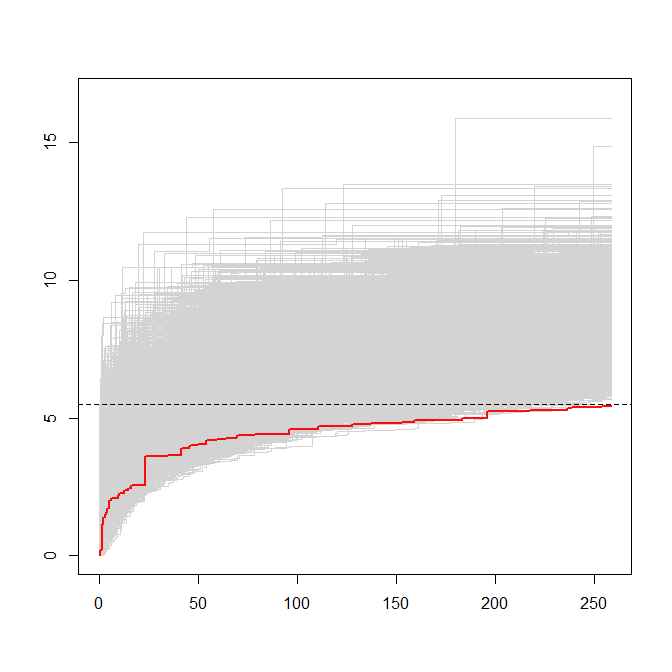}
	\includegraphics[width=0.49\textwidth]{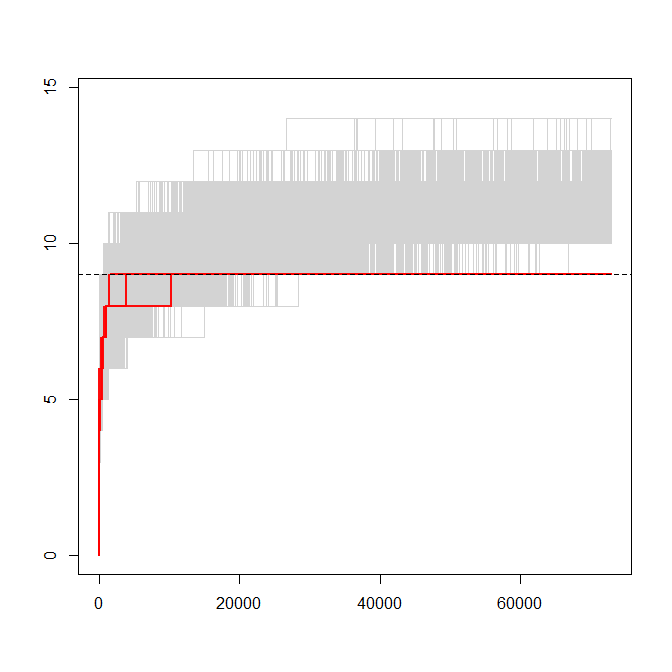}
	\caption{Simulation of a L\'evy subordinator under
        the time change \eqref{eq:sclock}.
           There were simulated 
        $10^4$ trajectories.
		On the left there is a gamma subordinator under 
        the time change,
        up to time $t=258.4618$, $\alpha=0.7$
		and $h=5.5$. On the right there is a homogeneous Poisson process under 
        the time change, with intensity one 
		up to time $t=72927.76$, $\alpha=1$ and $h=9$. 
		In red we indicated trajectories for which $\xi_{\rho(t)} \leq h$.
		Note that in both cases initially the path exhibits a typical behaviour and slows down near $t$.
	} 
	\label{fig:sym4}
	\end{figure}

\section{Correspondences with Crump--Mode--Jagers branching processes} \label{sec:CMJ}
%%%%%%%%%%%%%%%%%%%%%%%%%%%%%%%%%%%%%%%%%%%%%%%%

%%%%%%%%%%%%%%%%%%%%%%%%%%%%%%%
\subsection{Crump--Mode--Jagers branching processes}
%%%%%%%%%%%%%%%%%%%%%%%%%%%%%%%

	In this section, we introduce Crump--Mode--Jagers branching processes 
	(CMJ processes for short). Let 
	$D$ be a random point process on $[0, \infty)$. 
	%For the sake of simplicity, we will assume throughout that $D$ has at least one point in $[0, \infty)$ almost surely.

 	In the Crump--Mode--Jagers 
	branching process, we start at height zero with a single particle. This particle 
	reproduces offspring particles at heights 
	$0 \leq \mathcal{H}(1) \leq \mathcal{H}(2) \leq \ldots$ according to an independent copy of the 
	point process 
	$D$. If the reproductive point process has only $k$ points on 
	$[0,\infty)$, we use the convention that $\mathcal{H}(j) = \infty$ for all $j > k\geq 0$. 
	Thereafter, the $j^{\text{th}}$ child of the initial particle reproduces 
	children at heights
	\begin{align*}
		(\mathcal{H}(j1),\mathcal{H}(j2),\ldots ) := 
		(\mathcal{H}(j)+ d_{j1},\mathcal{H}(j) + d_{j2},\ldots),
	\end{align*}
	where $0 \leq d_{j1} \leq d_{j2} \leq \ldots$ are the points (listed in increasing order) 
	of an independent copy of the point process $D$ (if the process has only $k$ points, 
	set $d_{j\ell } :=+ \infty$ for all $\ell > k$). More generally, using the Ulam--Harris 
	labelling convention, each particle in the process is associated with a word $w$ in the set
	\begin{align*}
		\mathbb{T} := \bigcup_{n \geq 0} \mathbb{N}^n,
	\end{align*}
	with $\mathbb{N}^0=\{\varnothing\}$.
	If a particle associated with a word $w$ has height $\mathcal{H}(w)$, then the children 
	of this particle occur at heights
	\begin{align*}
		\left( \mathcal{H}(w1),\mathcal{H}(w2),\ldots \right) 
		= \left( \mathcal{H}(w) + d_{w1} , \mathcal{H}(w) + d_{w2},\ldots \right),
	\end{align*}
	where $\mathcal{D}_w = (d_{w1},d_{w2},\ldots )$ are the points (labelled in nondecreasing order) 
	of an independent copy $\mathcal{D}_w$ of $D$. In essence, the Crump--Mode--Jagers process is a random function
	\begin{align*}
		\mathcal{H}: \mathbb{T} \to [0,\infty]
	\end{align*}
	that may also be regarded as a branching random walk.
	We refer the reader to~\cite{Iksanov:2024:asymptotic} for an extensive literature review on the topic of CMJ processes.

%%%%%%%%%%%%%%%%%%%%%%%%%%%%%%%%%%%%%%%
\subsection{Limit theory for supercritical CMJ processes} \label{sec:CMJlimit}
%%%%%%%%%%%%%%%%%%%%%%%%%%%%%%%%%%%%%%%

	We now develop some central properties of supercritical CMJ  processes, 
	following Nerman \cite{nerman}. Firstly, we define the reproduction intensity $\mu$ 
	of the reproductive point process $D$ to be the measure $\mu$ on $[0,\infty)$ given 
	by
	\begin{align*}
		\mu(A) := \mathbb{E} [ D(A) ] := \mathbb{E} [ \# \{ \text{points of $D$ in $A$} \} ]
	\end{align*}
	for any Borel subset $A\subseteq [0,\infty)$. 

    Consider the following list of assumptions that may be satisfied by the law of the reproductive point process $D$:
	% Through the remainder of this section we will operate under the following  assumptions:
    \begin{quote} %%% SJ: I've used begin quote to get the lists to standout a bit more vs the normal text. 
    \normalsize
	\begin{itemize}
		\item[(i)] $\mu$ is not concentrated on any lattice $\{0,h,2h,\ldots\}$, $h>0$.
		\item[(ii)] There exists a Malthusian parameter, that is $\kappa > 0$ with the property that	
			\begin{align*}
				1 = \int_{[0,\infty)} e^{ - \kappa h } \mu( \mathrm{d}h).
			\end{align*}
 		\item[(iii)] With $\kappa$ as in (ii), the first moment of the exponentially tilted measure $e^{ - \kappa h } \mu( \mathrm{d}h)$ is finite, i.e.,
			\begin{align} \label{eq:firstmoment}
				\int_{[0,\infty)}he^{ - \kappa h } \mu( \mathrm{d}h) < \infty.
			\end{align}
        \item[(iv)] With $\kappa$ as in (ii), we define a random variable $\Omega_\kappa$ by 
	\begin{align*}
		\Omega_\kappa := \int_{[0,\infty) } e^{ - \kappa h } D(\mathrm{d}h).
	\end{align*}
    The variable $\Omega_\kappa$ satisfies the \emph{Kesten--Stigum condition}, 
	\begin{align} \label{eq:KS}
		\mathbb{E} \left[ \Omega_\kappa \log_+ \Omega_\kappa \right] < \infty.
	\end{align}
	\end{itemize}
    \end{quote}

\vspace{3mm}
	Given $a > 0$, define the random variable
	\begin{align*}
		D_h^a := \# \{ w \in \mathbb{T} : \mathcal{H}(w) \in (h-a,h] \},
	\end{align*}
	and for short write $D_h := D_h^\infty=\# \{ w \in \mathbb{T} : \mathcal{H}(w) \leq h \}$. 

Then according to \cite{nerman, doney1, doney2} under the conditions (i)--(iv) above, there exists a nontrivial and nonnegative random variable $W$ such that we have the almost-sure convergence
\begin{align} \label{eq:limit}
D_h e^{ - \kappa h} \to W \qquad \text{as $h \to \infty$},
\end{align}
and moreover, the event $\{W > 0\}$ a.s. coincides with the event that the process has infinitely many particles.
In addition, under these same conditions we have 
	\begin{align*} 
D_h^a e^{ - \kappa h} \to (1-e^{-\kappa a} ) W \qquad \text{as $h \to \infty$},
	\end{align*}
    where for every $a > 0$, $W$ is the same limit defined in \eqref{eq:limit}. 

%%%%%%%%%%%%%%%%%%%%%%%%%%%%%%%%%%%%%%%%%%%%%
\subsection{Crump--Mode--Jagers from self-similar fragmentations} \label{sec:proj}
%%%%%%%%%%%%%%%%%%%%%%%%%%%%%%%%%%%%%%%%%%%%%

	There is a natural way of generating a CMJ branching process from a self-similar 
	fragmentation process with finite activity, namely one simply associates each fragmentation event 
	with an increase in generation. In this section, however, we will seek to find a more robust 
	connection between self-similar fragmentation processes and CMJ processes that is 
	also suitable for the infinite-activity case. Our trick here is to observe the self-similar 
	fragmentation process at certain discrete and particle-dependent times.

	We may construct a CMJ branching process $\{ \mathcal{H}(w) \,:\, w \in \mathbb{T} \}$ 
	from a self-similar fragmentation process with possible infinite activity as follows. Each 
	$\mathcal{H}(w) < \infty$ will represent a fragment of size $e^{ - \mathcal{H}(w)}$ living at some 
	time $\mathcal{T}(w)$ in the process. Thus, we will construct a collection 
	$\{ \mathcal{T}(w) \,:\, w \in \mathbb{T} \}$ of times $\mathcal{T}(w)$ associated to words
	$w \in \mathbb{T}$.

	The construction is as follows. Set $\mathcal{T}(\varnothing) := 0$ and $\mathcal{H}(\varnothing) := 0$. 
	The word $\varnothing$ is associated with the initial fragment of size $1 = e^{- \mathcal{H}(\varnothing)}$ 
	existing at time $\mathcal{T}(\varnothing) = 0$. Consider now viewing the descendant fragments of this 
	initial fragment at time 
	$1 = e^{\alpha \mathcal{H}(\varnothing)}$. These fragments have 
	sizes $(X_1(1),X_2(1),\ldots)$ listed in decreasing order. We define the  variables 
	$\{ \mathcal{H}(i) \,:\, i \geq 1 \}$ by setting $e^{ - \mathcal{H}(i)} := X_{i}(1)$. We also define the 
	first generation times by setting $\mathcal{T}(i) := 1$ for all $i \geq 1$. Thereafter, consider a 
	word $w \in \mathbb{T}$ associated with a fragment existing at a moment $\mathcal{T}(w)$ in time 
	and of size $e^{- \mathcal{H}(w) }$. Consider viewing the descendants of this fragment at time 
	$\mathcal{T}(w) + e^{ \alpha \mathcal{H}(w) }$, and write $(e^{ - \mathcal{H}(w1)},e^{-\mathcal{H}(w2)},\ldots)$ 
	for the sizes of these descendants. Set $\mathcal{T}(wi) :=\mathcal{T}(w) + e^{ \alpha \mathcal{H}(w) }$ 
	for all $i \geq 1$. 

	By the self-similarity property, $\{ \mathcal{H}(w) : w \in \mathbb{T} \}$ is a Crump–Mode–Jagers branching process. 
	Let $\mu$ be the reproduction intensity and $D$ the associated reproductive point process.
    Since the process is conservative, the sum of the sizes of descendants of a given fragment at a later time 
	is equal to the size of that fragment. That is, for every $w \in \mathbb{T}$,
	\begin{align*} 
		e^{- \mathcal{H}(w)} = \sum_{i \geq 1} e^{- \mathcal{H}(wi)}.
	\end{align*}
	In particular, by taking $w=\emptyset$, we have 
	\begin{align} \label{eq:cad}
		1 = \int_{ [0,\infty) } e^{-h} D(\mathrm{d}h) \qquad \text{almost surely}.
	\end{align}   
    Taking expectations through \eqref{eq:cad}, we see that the Malthusian parameter $\kappa$ satisfies $\kappa = 1$. Note that in our case $\Omega_1$ is in fact deterministic, satisfying $\Omega_1 = 1$ almost surely, and accordingly, the Kesten--Stigum condition \eqref{eq:KS} holds trivially.                         
\vspace{3mm}

	We now seek a lower bound on the number of particles existing in the 
	CMJ process associated with our fragmentation processes at large heights. 
	In this direction, we would like to use the limit theory of CMJ processes to derive such a lower bound. 
	We now examine whether the assumptions in (i)-(iv) from Section \ref{sec:CMJlimit} hold. 
	In fact, we will find that while (i), (ii) and (iv) are satisfied, (iii) is not:
\begin{quote}
\normalsize
	\begin{itemize}
		\item[(i)] 	The non-lattice assumption \eqref{eq:nonlattice} guarantees that the associated CMJ process is non-lattice.
        \item[(ii)] The Malthusian parameter for the CMJ process associated with our fragmentation is $\kappa=1$. This follows from the conservativeness of the fragmentation. Indeed, taking expectations through \eqref{eq:cad} we have
        	\begin{align} \label{eq:cad2}
		1 = \int_{ [0,\infty) } e^{-h} \mu(\mathrm{d}h).
	\end{align}
    \item[(iv)] Equation \eqref{eq:cad} entails that $\Omega_1 = 1$ almost surely. It follows that the Kesten--Stigum condition $\mathbf{E}[ \Omega_1 \log_+ \Omega_1 ] < \infty$ holds trivially.
	\end{itemize}
    \end{quote}
However, it is not generally the case that condition (iii) holds for our CMJ process associated with a self-similar fragmentation. It is entirely possible that $\int_{[0,\infty)} he^{ - h} \mu(\mathrm{d}h) = \infty$; in fact, a straightforward argument using the spine process of Section \ref{sec:one} may be used to argue that $\int_{[0,\infty)} he^{-h} \mu(\mathrm{d}h) < \infty$ implies \eqref{eq:intcond2} and the authors believe that the converse implication is also likely true.

	In any case, we circumvent any concern about whether or not $\int_{[0,\infty)}he^{-h} \mu(\mathrm{d}h) < \infty$ 
	(or \eqref{eq:intcond2}) holds by appealing to an inelegant but simple strategy: 
	we simply truncate the CMJ process by ignoring any particle whose height exceeds that of their parent by more than $M$.  
	This procedure induces a second, truncated CMJ process 
	$\mathcal{H}_M := \{\mathcal{H}_M(w) : w \in \mathbb{T} \}$ defined inductively by 
	\begin{align*} 
	\mathcal{H}_M(wi) :=
	\begin{cases}
	\mathcal{H}(wi) \quad &: \text{$\mathcal{H}(wi) - \mathcal{H}(w) \leq M$ and $\mathcal{H}_M(w) < \infty$,}\\
	+ \infty \quad &: \text{otherwise}.% \text{ $\mathcal{H}(wi) - \mathcal{H}(w) > M$ or $\mathcal{H}_M(w)  = \infty$ (or both)}.
	\end{cases}
	\end{align*} 
    	The reproduction intensity measure $\mu_M$ of the truncated CMJ process $\mathcal{H}_M$ 
	is simply the restriction of $\mu$ to $[0,M]$. Likewise, the random point process $D_M$ has the law of $D$ 
    	restricted to $[0,M]$ (i.e., $D$ with every point outside of $[0,M]$ deleted). 

	We note that this definition couples, for any $ M > 0$, the truncated process with the untruncated process.

	We now verify that all four of the conditions (i)--(iv) of Section \ref{sec:CMJlimit} are satisfied by the truncated CMJ process, provided $M$ is sufficiently large:
\begin{quote}
\normalsize
    \begin{itemize}
	\item[(i)] 	If $M$ is sufficiently large, $\mathcal{H}_M$ is still not concentrated on any lattice so that (i) holds.
        \item[(ii)] Equation \eqref{eq:cad2} implies that there exists $\kappa_M < 1$ such that 
      \begin{align} \label{eq:malthusian}
		1 = \int_{[0,M]} e^{ - \kappa_M h} \mu(\mathrm{d}h) = \int_{[0,\infty)} e^{- \kappa_M h } \mu_M(\mathrm{d}h).
	\end{align}
    \item[(iii)] Plainly, 
    \begin{align*}
    \int_{[0,\infty)} h e^{- \kappa_M h } \mu_M(\mathrm{d}h) \leq M \int_{[0,M]} e^{ - \kappa_M h} \mu(\mathrm{d}h) = M <\infty.
    \end{align*}
\item[(iv)] Similarly,
  \begin{align*}
	  \Omega_{\kappa_M}^{(M)} = \int_{[0,\infty)} e^{ - \kappa_M h} D_M(\mathrm{d}h) \leq e^{ (1-\kappa_M)M} \int_{[0,M]} e^{ - h} D (\mathrm{d}h) \leq e^{ (1-\kappa_M)M},
    \end{align*}
    almost surely.
    Thus, the Kesten--Stigum condition $$\mathbf{E}\left[ \Omega_{\kappa_M}^{(M)} \log_+\left( \Omega_{\kappa_M}^{(M)} \right) \right] < \infty$$ 
    holds. 
	\end{itemize}
    \end{quote}

Let us also note that the Malthusian parameter $\kappa_M$ satisfies $\kappa_M \uparrow 1$ as $M \to \infty$.

In summary, the level $M$-truncated CMJ process clearly satisfies (i)--(iv), and accordingly if we set 
\begin{align*}
D_{h,M}^a := \# \{ w \in \mathbb{T} \,:\,  \mathcal{H}_M(w) \in (h-a,h]  \}
\end{align*}
and $D_{h,M} := D_{h,M}^\infty$, then we have the almost sure convergence
	\begin{align} \label{eq:conve}
		D_{h,M}^a e^{ - \kappa_M h} \to \left(1 - e^{ - \kappa_M a }\right)W_M, \quad\text{as }  h \to \infty,
	\end{align}
for a nontrivial nonnegative limit random variable $W_M$ depending on the truncation level $M$.

\begin{lemma} \label{lem:mart}
	The probability $p_M := \mathbb{P}[ W_M > 0 ]$ increases to $1$ as $M \to\infty$.
\end{lemma}
\begin{proof}
	As noted above, the event $\{W_M > 0 \}$ coincides with the event that the level $M$-truncated CMJ process has 
	infinitely many particles. To bound the probability of this event from below, we consider a coupling with a 
	discrete-time Galton--Watson tree. 
	In this direction, let $L(\varnothing) \in \{0,1,\ldots,+\infty\}$ denote the number of children of finite height birthed 
	by the initial particle $\emptyset \in \mathbb{T}$ in the CMJ process associated with a self-similar fragmentation process, i.e.,
	\begin{align*}
		L(\varnothing) := \# \{ i \geq 1\,:\, \mathcal{H}(i) < \infty \}.
	\end{align*}
	Equivalently, $L(\varnothing)$ is the number of fragments of positive size in the fragmentation process at time one.
	Likewise let
	\begin{align*}
		L_M(\varnothing) := \# \{ i \geq 1\,:\,  \mathcal{H}(i) \leq M \}
	\end{align*}
	be the associated number in the truncated process. 
	This corresponds to the number of fragments of size $\geq e^{-M}$ in the fragmentation process.
	Then,  the function $M\mapsto L_{M}(\varnothing)$ is nondecreasing, 
	$L_M(\varnothing) < \infty$ for all $M < \infty$, and 
	$L_M(\varnothing) \uparrow L(\varnothing)$ as $M \to \infty$. 

	Set $r = \mathbb{P}[ L(\varnothing) = 1]$, $b=\mathbb{P}[L(\varnothing)\geq 2]$
	and note that since for $\alpha>0$ the fragmentation survives with probability one (i.e. does not hit $(0, 0, \ldots)$), 
	$\mathbb{P}[L(\varnothing)=0] =0$ (see the tagged fragment process described in Section~\ref{sec:levy}).
	% We observe that in the infinite-activity case, we have $L(w)=+ \infty$ almost surely, so 
	% that $r = 0$, while in the finite-activity case, we will have $L(w) < \infty$ almost surely, and 
	% the event $\{L(\varnothing) = 1\}$ is equivalent to the event where 
	% the fragmentation process does not fragment up in the interval $[0,1]$. Then,  
	% $r = e^{ - \lambda}$, where $\lambda = \nu(\MassP)$ is the total event rate of a single fragment. 
	Now observe that 
	\begin{align*}
		\mathbb{P}[ L_M(\varnothing) = 1 ] = r_M \to r,\qquad \text{as $M\to \infty$}
	\end{align*}
	and
	\begin{align*}
		b_M = \mathbb{P}[ L_M(\varnothing) \geq 2 ] \to 1 - r, \qquad \text{ as $M \to \infty$}.
	\end{align*}
	In particular, the number of particles in the process is bounded below by a Galton-Watson 
	tree in which a particle has two children with probability $b_M$, has one child with probability $r_M$, 
	and otherwise has zero children with probability $1-r_M-b_M$. We note in particular that the probability of 
	having no children, $1 - r_M- b_M$, converges to zero as $M \to \infty$. 
	The extinction probability $q_M$ of this Galton-Watson tree converges to zero as $M \to \infty$. 
	In fact, $q_M$ is the smallest root of $s = ( 1 - r_M - b_M ) + r_Ms + b_M s^2$. So,
	\[
		q_M = (2b_M)^{-1}(1 - r_M  - \sqrt{ (1-r_M)^2 - 4 b_M ( 1 - r_M - b_M) })\rightarrow 0, \text{ as $M \to \infty$}.
	\]
Since the probability that the CMJ process survives forever is at least $1 - q_M$, it follows that $p_M \to 1$ as $M \to \infty$. 
\end{proof}

% \textcolor{red}{[Is this really necessary?]}
%     \color{gray}
% 	We now use the truncated CMJ process to show that $W$ is almost surely positive.

% 	\begin{lemma}\label{lem:2:hotfix2}
% 		$\mathbb{P}[W>0]=1$.
% 	\end{lemma}
% 	\begin{proof}
% \color{OliveGreen}
% We now establish that $p_M := \mathbb{P}(W_M = 0) \downarrow 0$ as $M \to \infty$. 
% To see this, let $Q = - \log X_2(t)$ be the size of the second largest particle at time $1$ in the untruncated process. 
% Then $Q$ is a $(0,\infty)$-valued random variable. 
% Let $q_M = \mathbb{P}( Q \leq M )$. Plainly $q_M \uparrow 1$ as $M \to \infty$ {\color{Red} NO!}. 
% Then the probability the $M$-truncated process survives forever is at least as large the survival probability 
% $s_M$ of a Galton-Watson tree in which, starting with a single initial particle, 
% each particle has a probability $q_M$ of having two children and probability 
% $1- q_M$ of having zero children in the next generation. 
% We have $s_M \uparrow 1$ as $M \to \infty$, and accordingly, $p_M \uparrow \infty$ also. 
% 	\end{proof}
\color{black}

We will extract from these results the following rough lower bound on the number of particles in our CMJ branching process existing at any point in time and with heights $\mathcal{H}(w)$ taking values in $(h-a,h]$.

\begin{proposition} \label{prop:heightlower}
Let $a > 0$ and $\varepsilon > 0$. Then, there exists a random number $h(a,\varepsilon)$ such that for all $h \geq h(a,\varepsilon)$ we have
\begin{align*}
D_h^a \geq e^{ (1 - \varepsilon) h}.
\end{align*}
\end{proposition}

\begin{proof}
	Take a copy of the untruncated CMJ process associated with our fragmentation process. 
	Simultaneously, for each $M \in (0,\infty)$, we can couple a copy of the level 
	$M$-truncated process with a copy of this untruncated CMJ process. 
	Moreover, for each $M \in (0,\infty)$, we have the convergence in \eqref{eq:conve}. 

	By Lemma \ref{lem:mart}, the survival probability $p_M := \mathbb{P}[ W_M > 0]$ increases to $1$ as $M \to \infty$. 
	By the coupling, it follows that there exists a random $M_0$ such that $W_M > 0$ for all $M \geq M_0$. 

	Moreover, with $\varepsilon > 0$ as in the statement of the proposition, and with the Malthusian parameter $\kappa_M$ as in \eqref{eq:malthusian}, there exists a deterministic $M_1(\varepsilon) > 0$ such that for all $M \geq M_1(\varepsilon)$ we have $\kappa_M \geq 1- \tfrac{\varepsilon}{2}$.

In particular, there exists a random $M_2 := M_2(\varepsilon) = \mathrm{max} \{ M_0, M_1(\varepsilon) \}$ such that for all $M \geq M_2$ we have both $W_{M_2} > 0$ and $\kappa_{M_2} \geq 1- \tfrac{\varepsilon}{2}$. By virtue of \eqref{eq:conve}, we have
\begin{align*}
D_{h,M_2}^a  \geq \frac{1}{2} W_{M_2}   (1 - e^{ - \kappa_{M_2} a }) e^{  \kappa_{M_2} h } \geq \frac{1}{2} W_{M_2}  (1 - e^{ - \kappa_{M_2} a }) e^{  (1-\varepsilon/2) h } 
\end{align*}
for all $h \geq h_0(M_2,a)$ exceeding some random height $h_0(M_2,a)$.

Since $W_{M_2} > 0$, there exists a potentially larger random height $h_1(M_2,a) \geq h_0(M_2,a)$ such that whenever $h \geq h_1(M_2,a)$ we have $\frac{1}{2} W_{M_2}  (1 - e^{ - \kappa_{M_2} a }) \geq e^{- \varepsilon h/2}$, which in turn implies that 
\begin{align*}
D_{h,M_2}^a  \geq  e^{  (1-\varepsilon) h }.
\end{align*}

Finally, $D_h^a \geq D_{h,M_2}^a$, since the truncated process consists of a subset of the particles occurring in the original process. It follows that $D_h^a \geq e^{(1-\varepsilon)h}$ for all $h$ exceeding some random height $h(a,\varepsilon) = h_1(M_2(\varepsilon),a)$ depending on $a$ and $\varepsilon$, completing the proof.
\end{proof}

%%%%%%%%%%%%%%%%%%%%%%%%%%%%%%%%%%%%%%%%%%%%%%%%%%%%%%%%
\subsection{Independent sets in Crump--Mode--Jagers branching processes} \label{sec:independent}
%%%%%%%%%%%%%%%%%%%%%%%%%%%%%%%%%%%%%%%%%%%%%%%%%%%%%%%%

	In the sequel, we will use a variant of our previous lower bounds for CMJ processes 
	to show that there exist particles exhibiting certain properties with high probability.
	In this direction, we will apply a variant of Proposition~\ref{prop:heightlower} 
	to show that there exist a number of fragments of size between $e^{-h}$ and $e^{-(h - s)}$ 
	at various times in the process.
	However, we wish to rule out the possibility that one of these fragments is the parent, grandparent, 
	or ancestor of another, so that the corresponding events can be considered independent. 
	This motivates the following definition.

	\begin{df}
		A subset $\mathbb{S} \subseteq \mathbb{T}$ is called an \textbf{antichain} if and only if for every distinct pair of elements $w, w' \in \mathbb{S}$, 
neither $w$ nor $w'$ is an ancestor of the other. 
	\end{df}

	The following result guarantees the existence of large antichains of fragments at large heights.

	\begin{thm} \label{thm:antichains}
		Let $\delta > 0 $ and  $ \log(2) > a > 0$. 
		Then there exists a random number $h_0(a,\delta)\in[0,\infty)$ such that for all $h \geq h_0(a,\delta)$, 
		there exists a subset $\mathbb{S}$ of $\mathbb{T}$ so that the following hold:
		\begin{enumerate}
			\item $\mathbb{S}$ is an antichain,
			\item $\mathcal{H}(w) \in (h-a,h]$ for all $w \in \mathbb{S}$,
			\item The cardinality of $\mathbb{S}$ satisfies $\# \mathbb{S} \geq e^{ (1 - \delta)h}$.
		\end{enumerate}
	\end{thm}

	\begin{proof}
		Define
		\begin{align*}
			\mathcal{D}(h-a,h] := \{ w \in \mathbb{T} : \mathcal{H}(w) \in (h-a,h] \}.
		\end{align*}
        
		By Proposition \ref{prop:heightlower} there exists a random variable 
		$h_0(a,\delta/2)$ such that for all $h \geq h_0(a,\delta/2)$ we have
		\begin{align} \label{eq:bc1}
			D_h^a = \# \mathcal{D}(h-a,h] \geq e^{ (1 - \delta/2) h}. 
		\end{align}
		We now show that with high probability, $\mathcal{D}(h-a,h]$ contains 
		an antichain of size at least $e^{(1-\delta)h}$.
		To this end, let $\tilde{\mathcal{D}}(h-a,h]$ denote the set of first ancestors in $(h-a,h]$. 
		This is the subset of $\mathcal{D}(h-a,h]$ consisting of those elements with no 
	ancestor also in $\mathcal{D}(h-a,h]$. In other words,
		\begin{align*}
			\tilde{\mathcal{D}}(h-a,h] := 
			\{w \in \mathcal{D}(h-a,h] : \mathcal{H}(v) \leq h-a ~\text{for every strict ancestor $v$ of $w$} \}.
		\end{align*}
		Plainly, $\tilde{\mathcal{D}}(h-a,h]$ is an antichain. 
		We now seek to bound below the cardinality $\tilde{D}_h^a$ of $\tilde{\mathcal{D}}(h-a,h]$. 
		In this direction, let $w \in \tilde{\mathcal{D}}(h-a,h]$ and define the random variable
		\begin{align*}
			N_w := \# \{ v \in \mathcal{D}(h-a,h] : v \text{ descendant of $w$} \}.
		\end{align*}
		We control the expectation of $N_w$. Recall that from the conservativeness of the underlying fragmentation process for every $w \in \mathbb{T}$ we have
		\begin{align*} 
			1 = \sum_{i \geq 1}e^{ - ( \mathcal{H}(wi) - \mathcal{H}(w) ) }.
		\end{align*}
		Note that since $a < \log(2)$, the previous equation tells us that each $w$ in $\mathbb{T}$ 
		may have at most one child $wi$ with $\mathcal{H}(wi) - \mathcal{H}(w) < a$. 
		Let $p(a) \in (0,1)$ denote the probability of this occurring, i.e.,
		\begin{align*}
			p(a) := \mathbb{P}\left[ \exists i : \mathcal{H}(i)  < a \right].
		\end{align*}
		It follows that for each $w$ in $\tilde{\mathcal{D}}(h-a,h]$, $N_w$ is stochastically dominated by a 
		geometric random variable with success probability $p(a)$. In particular, there is some constant $c(a) < \infty$ such that 
		\begin{align*}
			\mathbb{E}\big[N_w+1 | w \in \mathcal{D}(h-a,h]\big] \leq c(a).
		\end{align*}
		In particular, 
		\begin{align*}
			\mathbb{E}\left[ D_h^a \left| \tilde{D}_h^a\right. \right] \leq c(a) \tilde{D}_h^a.
		\end{align*}
		By a conditional Markov's inequality, it follows that
		\begin{align} \label{eq:bc2}
			\mathbb{P}\left[ \tilde{D}_h^a \leq e^{ -\delta h/2} D_h^a \right] \leq c(a) e^{ - \delta h/2}.
		\end{align}
		Combining \eqref{eq:bc1} with \eqref{eq:bc2} and using the Borel--Cantelli lemma, 
		it follows that there exists a random $h_1(a,\delta)$ such that 
		\begin{align*}
			\# \tilde{\mathcal{D}}(h-a,h] \geq e^{(1-\delta)h} \qquad 
			\text{for all $h \geq h_1(a,\delta)$ such that $ h  = na $ for some $n \in \mathbb{N}$}.  
		\end{align*} 
		To transfer this result from $h$ sufficiently large taking values in the lattice 
		$\{ n a : n \in \mathbb{N} \}$ to all sufficiently large $h$, 
		note that any interval of the form $(h-a,h]$ necessarily contains a subinterval of the form $((n-1)a/2,na/2]$. In particular,
		\begin{align*}
			\tilde{\mathcal{D}}(h-a,h] \supseteq \tilde{\mathcal{D}}\left(n\frac{a}{2} - \frac{a}{2}, n\frac{a}{2} \right] 
			\qquad \text{for some $n \in \mathbb{N}$}.
		\end{align*}
		It thus follows that for $h \geq h_2(a,\delta) := h_1(a/2,\delta)$ we have
		\begin{align*}
			\# \tilde{\mathcal{D}}(h-a,h] \geq e^{(1-\delta)h}.
		\end{align*}
		Thus, for all real $h$ greater than some random height $h_2(a,\delta)$, 
		the antichain $\mathbb{S} := \tilde{\mathcal{D}}(h-a,h]$ of first ancestors in $(h-a,h]$ 
		has cardinality at least $e^{(1-\delta)h}$, completing the proof.
\end{proof}

%%%%%%%%%%%%%%%%%%%%%%%%%%%%%%%%%%%%%%%%%%
\section{Spines, L\'evy processes, and tails of hitting times} \label{sec:levy}
%%%%%%%%%%%%%%%%%%%%%%%%%%%%%%%%%%%%%%%%%%

	To describe the behaviour of the largest fragment in a self-similar fragmentation process 
	we will relate the log-sizes of the fragments to a time-changed branching L\'evy process. 
	The large deviation estimates for L\'evy processes are therefore a natural tool to analyse 
	the extremes, i.e., the size of the largest fragment. 
 %    n this section we will describe the 
	% connection between self-similar fragmentations and L\'evy processes.
    The deviation estimates 
	for the time-changed process that we present at the end of this section will be utilized in 
	the proof of our main results.

%%%%%%%%%%%%%%%%%%%%%%%%%%%%%%%%%%%%%%%%%%
\subsection{A tagged fragment as a L\'evy process}\label{sec:one}
%%%%%%%%%%%%%%%%%%%%%%%%%%%%%%%%%%%%%%%%%%

	We now describe the spine construction appearing in~\cite[Chapters 3.2.2 and 3.3.2]{bertoinbook}. 
	Since this process is described in detail there, we will give only a brief verbal description.

For the sake of simplicity, we will begin by describing the construction in the finite-activity case. Assume that the fragmentation starts at time zero with a single object of mass one. We begin by tagging this particle, calling the tag the `spine'.
In the finite-activity case, this object has an exponential lifetime with some rate $\lambda$ and the spine `follows' this particle for the duration of its lifetime, before it dies and breaks into a (possibly infinite) number of pieces whose sizes sum to the size of their parent. At the end of the original particle's lifetime, the spine then makes a size-biased choice among the child particles to determine which one to follow next. Ignoring all other particles in the system, the spine then follows this child particle until that child dies, at which point the spine again chooses a size-biased grandchild to follow next. 

It is possible to generalise this construction to the infinite-activity case, under the integral condition \eqref{eq:1:crumbling}, which guarantees that the size of the followed particle remains positive at all times.

	The size-biased nature of our selection at each time entails that at \emph{any} time $t \geq 0$, 
	given only the state of the underlying fragmentation process at time $t$, 
	but not which particle the spine is following, the spine has a size-biased distribution 
	among all particles in the system at time $t$, regardless of the genealogy of these particles. 
	Consider letting $(Z_t)_{t \geq 0}$ denote the stochastic process where $Z_t$ is the size of the 
	particle followed by the spine at time $t$. As noted in Section \ref{sec:further}, since the spine particle is a size-biased 
	pick from the population at time $t$, we have the many-to-one formula:
	\begin{equation}\label{eq:2:Many2one}
		\mathbb{E} \left[ \sum_{i \geq 1} f(X_i(t)) \right] 
		= \mathbb{E} \left[ \frac{f(Z_t)}{Z_t} \right],
	\end{equation}
	valid for all measurable $f$ supported away from zero.
    The formula \eqref{eq:2:Many2one} reduces the calculation of expectations over the entire system to that of a one-dimensional process.
	It is therefore beneficial to analyse the stochastic process $(Z_t)_{t \geq 0}$ in further detail. 

    Here we begin by noting that the self-similarity of the underlying fragmentation process together with the 
    Markov property entails that we have the distributional identity 
	\begin{equation*}
		\left(Z_{t+t'}/Z_t\right)_{t'} \stackrel{d}{=} \left(Z'_{\, Z_t^{\alpha} t'}\right)_{t'} \qquad   \text{for any $t \geq 0$},
	\end{equation*}
	conditionally on $Z_t$, where $Z'$ is an independent copy of $Z$.  

	Therefore, $Z$ is a positive self-similar Markov process and thus admits a Lamperti representation; that is, it can be expressed as the exponential of a time-changed L\'evy process. Indeed, if we consider
	\begin{equation*}
		\rho_Z(t) = 
		\inf\left\{ s\geq 0 \, : \, \int_0^s Z_r^\alpha {\rm d} r >t \right\}, \qquad t\geq 0,
	\end{equation*}
	then $\xi(t) = -\log Z_{\rho_Z(t)}$ is a
	L\'evy process. The associated L\'evy measure $\Lambda$ satisfies
	\begin{equation}\label{eq:measure}
		\int_{(0, +\infty)} f(x) \:  \Lambda ({\rm d} x)  
		= \int_{\MassP}\sum_{i\geq 1} s_i f(-\log(s_i)) \: \nu ( \mathrm{ d}  \mathbf{s}), 
		\qquad \mathbf{s} = (s_1,s_2, \ldots),
	\end{equation}
   for any measurable $f \colon (0,+\infty) \to \mathbb{R}$.
	More precisely, $\xi$ has  Laplace exponent given by
	\begin{equation}\label{eq:Phidef}
		\Phi(q) = - \log \mathbb{E}\left[e^{-q \xi_1}\right] 
		= \int_{(0, +\infty)} \left(1-e^{-qx}\right) \Lambda(\mathrm{d} x)
		= \int_{\MassP} \left(1 - \sum_{i \geq 1} s_i^{q+1} \right) \: \nu({\rm d} \mathbf{s})
	\end{equation}
	for $q\geq 0$.
	Note that if $\nu$ is a finite measure, then $\xi$ is a pure jump L\'evy process. 
	Naturally, one can recover $Z$ from the L\'evy process $\xi$ by using the time change
	\begin{equation} \label{eq:rhodef}
		\rho(t) := \inf \left\{ u \geq 0 \,:\, \int_0^u e^{\alpha \xi_r} \mathrm{d}r > t \right\}, 
		\qquad t \geq 0,
	\end{equation}
	via $Z_t = \exp \{ -\xi_{\rho(t)} \}$. We can thus recast~\eqref{eq:2:Many2one} as
	\begin{equation} \label{eq:m21f}
	\mathbb{E} \left[ \sum_{i \geq 1} f(X_i(t)) \right] = 
		\mathbb{E}\left[ e^{\xi_{\rho(t)}}f\left(e^{-\xi_{\rho(t)}}\right)\right].
	\end{equation}
    
    For further details on the size-based spine process, we refer to~\cite[Theorem 3.2 and 3.3]{bertoinbook}.

\subsection{Tools from the theory of regularly varying functions}

    In this subsection, we summarise a couple of tools from the theory of regular variation~\cite{bingham1989regular}, 
    a notion going back to Karamata \cite{Karamata1933}.
    Recall that we say a function $L:(0,\infty) \to (0,\infty)$ is slowly varying at infinity if and only if for every $s > 0$, we have 
    \begin{align*}
    \lim_{x \to \infty} \frac{ L(sx)}{L(x)} = 1.
    \end{align*}
    More generally, we say that a function $R:(0,\infty) \to (0,\infty)$ is \emph{regularly varying at infinity} (or a Karamata function) if and only there exists a function $g:(0,\infty) \to (0,\infty)$ such that for all $s > 0$, we have 
    \begin{align*}
    \lim_{x \to \infty} \frac{R(sx)}{R(x)} = g(s).
    \end{align*}
    It transpires that $g$ must take the form $g(s) = s^\rho$ for some $\rho \in \mathbb{R}$. 
    In fact, a theorem of Karamata states that if $R$ is regularly varying, 
    then there exists a slowly varying function $L$ such that $R(x) = x^\rho L(x)$ \cite{bingham1989regular}.
    
    If $A(s)$ is positive, continuous, regularly varying, strictly monotone increasing to infinity, and takes the form 
    \begin{align} \label{eq:ro1}
    A(s) = s^\beta I(s),
    \end{align}
    with $I$ slowly varying, then its inverse function $B(s) = A^{-1}(s)$ has these same properties (though with a different index) and takes the form 
    \begin{align}  \label{eq:ro2}
    B(s) = s^{\frac{1}{\beta}} J(s), 
    \end{align}
    with $J$ slowly varying. By studying the relation $A(B(s))= s$, we see that $I$ and $J$ satisfy the relationship
    \begin{align}  \label{eq:ro3}
    1 = J(s)^\beta I(s^{1/\beta}J(s)).
    \end{align}
    Conversely, studying the relation $B(A(s)) = s$, we have
    \begin{align}  \label{eq:ro4}
    1 = I(s)^{1/\beta} J( s^\beta I(s)).
    \end{align}

\subsection{Tail bounds for L\'evy processes} \label{sec:tail}
	In order to study the behaviour of the largest fragments at certain times,  
	it is natural to consider the probability that a size-biased fragment does not decrease  
	too much in size over a certain period of time.  
	Putting the time change aside for a moment, our goal in this section is to study the asymptotics of the quantity  
	\begin{align*}
		\mathbb{P}[\xi_t \leq tx],
	\end{align*}
	for large $t$ and small $x$. 

	The large-$t$ asymptotics of the probabilities $\mathbb{P}[\xi_t \leq 1]$  
	associated with a L\'evy process $(\xi_t)_{t \geq 0}$  
	are closely related to the upper tails of exponential functionals  
	\[
    		I := \int_0^\infty e^{-\xi_t} \, \mathrm{d}t,
	\]
	of L\'evy processes, for which there exists a rich literature~\cite{haas2021precise, rivero2003law}.  
	However, the exact asymptotics we require do not appear in this body of work,  
	and thus we now state a bespoke result on the asymptotics of these probabilities.  

	Recall that we work under the assumption that our self-similar fragmentation process 
	has a dislocation measure $\nu$ satisfying 
	$\nu( \{\mathbf{s}\in \MassP: 1 - s_1 > \delta \} ) = \delta^{-\theta} \ell(1/\delta)$ 
	for some $\theta \in [0,1)$ and $\ell$ a slowly varying function at infinity.
    Recall that we have three cases: 
    \begin{itemize}
    \item $\theta = 0$, finite activity: where $\theta = 0$ and $\ell(x) \uparrow \lambda$ as $x \to \infty$.
    \item $\theta = 0$, infinite activity: in which $\theta = 0$ and $\ell(x) \uparrow \infty$ as $x \to \infty$. In this case, as in Assumption \ref{assu:1} we suppose further that $\ell(x) \sim \int_1^x \ell_0(u)u^{-1}\mathrm{d}u$ as $x \to \infty$. 
    \item $\theta \in (0,1)$, which automatically means infinite activity.
    \end{itemize}
    
In the infinite-activity case, where $\theta \in (0,1)$ or $\theta = 0$ and $\ell(x) \uparrow \infty$ as $x \uparrow \infty$, define a slowly varying function by 
 \begin{align} \label{eq:Ldef}
L_1(y) = 
\begin{cases}
 \ell_0(y) \qquad &:\text{$\theta = 0$,  where $\ell_0$ is as in \eqref{eq:1:fortheta0},}\\
 \theta \Gamma(1-\theta)\ell(y) \qquad &:\text{$\theta \in (0,1)$}.\\
\end{cases} 
\end{align}
In Lemma \ref{lem:qx} we will use the inverse of the derivative $\Phi'(q)$ of the Laplace exponent $\Phi(q)$ to define a slowly varying function $J$ satisfying 
    \begin{align} \label{eq:dbc0again}
\lim_{s \to \infty} J(s)^{\theta - 1} L_1 \left( s^{\frac{1}{1-\theta}} J(s) \right) = 1,
\end{align}
and 
    \begin{align} \label{eq:dbc0again2}
\lim_{s \to \infty} L_1(s)^{ - \frac{1}{1-\theta}}J(s^{1-\theta}/L_1(s)) = 1. 
\end{align}
From $J$, we define a slowly varying function $L(\cdot)$ by setting 
 \begin{align} \label{eq:Lpuredef}
L(s) = 
\begin{cases}
\lambda \qquad &:\text{$\theta = 0$, finite activity,}\\
\ell(sJ(s)) \qquad &:\text{$\theta = 0$, infinite activity,}\\
\frac{1-\theta}{\theta}J(s) \qquad &:\text{$\theta \in (0,1)$}.
\end{cases} 
\end{align}
In the $\theta \in (0,1)$ case, we can use \eqref{eq:dbc0again} and \eqref{eq:dbc0again2} to characterise $L(\cdot)$ directly in terms of $\ell$ (without the intermediate definition for $J(\cdot)$) as follows. Combining \eqref{eq:Ldef}  with \eqref{eq:dbc0again} and \eqref{eq:dbc0again2} we obtain the relations
\begin{align*}
C_1 L(s)^{\theta-1} \ell \left( s^{\frac{1}{1-\theta}} L(s) \right) \sim 1, \qquad \theta \in (0,1),
\end{align*}
and
\begin{align} \label{eq:beak}
\frac{1}{C_1^{\frac{1}{1-\theta}} \ell(s)^{\frac{1}{1-\theta}}} L \left( s^{1-\theta}/\ell(s) \right) \sim 1, \qquad \theta \in (0,1),
\end{align}
where $C_1 = \theta^\theta (1-\theta)^{1-\theta}\Gamma(1-\theta)$. 

As for the $\theta = 0$ infinite-activity case, it appears there is no simple way of describing $L(\cdot)$ in terms of $\ell$ and $\ell_0$, short of the two-step procedure of first defining $J(\cdot)$ as in \eqref{eq:dbc0again} and $L(\cdot) = \ell(sJ(s))$. 
 
With $L$ now defined in every case, we are now ready to state the main result of this section:

\begin{thm}[Asymptotics for spine processes] \label{thm:levybound}
 For any $\delta>0$ there exist a real number $x_0 > 0$, a slowly varying function $\hat{L}:(0,\infty) \to (0,\infty)$, and  a function $\Delta \colon [0,+\infty) \to [0,+\infty)$ with $\Delta(x) \to 0$ as $x \to 0$, such that 
\begin{equation*} 
	t x^{ - \frac{\theta}{1-\theta}} L(1/x) (1+\Delta(x))  \leq - \log\mathbb{P}[ \xi_t \leq t x] 
\end{equation*}
and
\begin{equation*}
	- \log\mathbb{P}\left[  xt  \leq \xi_t \leq xt + \sqrt{t x^{ \frac{2-\theta}{1-\theta} - \delta}}\right] 
	\leq t x^{ - \frac{\theta}{1-\theta}} L(1/x) (1 + \Delta(x)),
\end{equation*}
whenever $t \geq \hat{L}(1/x)$ and $0 < x \leq x_0$.
	\end{thm}

To begin the proof of Theorem~\ref{thm:levybound}, we employ an exponential change of measure commonly used in large deviation theory. In the context of L\'evy processes, it appeared in Jain and Pruitt~\cite[Lemma~5.2]{JP}.

\begin{lemma} \label{lem:tailsnew}
	Suppose that $(\xi_t)_{t \geq 0}$, under $\mathbb{P}$, is a nondecreasing L\'evy process with Laplace 
		exponent $\Phi(q)$ starting from $\xi_0 = 0$. Let $x\in\left(0, \mathbb{E}[\xi_1]\right)$, and let $q_x$ be determined 
		by $\Phi'(q_x) = x$, where $\Phi'$ is the derivative of $\Phi$. Define $R(q) := \Phi(q) - q \Phi'(q)$. Then under the condition
        \begin{align} \label{eq:tailscondition} 
 		t \geq r(x), \qquad \text{where}\qquad  
		r(x) = C  \frac{ \mathbb{E}^{(q_x)}[ (\xi_1 - x)^4 ]^2 }{\mathbb{E}^{(q_x)}[ (\xi_1 - x)^2 ]^4 },
        \end{align} 
         we have 
         \begin{equation} \label{eq:tailsnew}
  \log \mathbb{P}[ \xi_t \leq t x] \leq  - tR(q_x)
   \end{equation} 
   and
        \begin{equation} \label{eq:tailsnew2}
         - t R(q_x) - \sqrt{t} q_x \sqrt{\Phi''(q_x)}  - c \leq  
	 \log \mathbb{P}\left[xt\leq \xi_t \leq  xt + \sqrt{t\Phi''(q_x)}\right].
        \end{equation}
   
        Here $c, C > 0$ are universal constants. 
\end{lemma}
	Note that the definition of $r(x)$ does not represent a restriction. 
	The reason is that  under $\mathbb{P}^{(q)}$, $(\xi_t)_{t\geq 0}$ is a L\'evy process with 
	Laplace exponent given by $\Phi_q(\theta) = \Phi(q+\theta) - \Phi(q)$ and $\Phi$ 
	is an analytic function on $(0,\infty)$. Then, the fourth and second moments are guaranteed to exist at $q_x$ for all sufficiently small $x > 0$. 
    
We provide a proof of the upper bound \eqref{eq:tailsnew} in Lemma \ref{lem:tailsnew}, relegating the lower bound proof to the appendix. 

\begin{proof}[Proof of the upper bound in Lemma \ref{lem:tailsnew}]
		Recall that for any parameter $q > 0$, we may define a 
		martingale change of measure $\mathbb{P}^{(q)}$ by setting
	\begin{align*}
		\frac{ \mathrm{d}\mathbb{P}^{(q)}|_{\mathcal{F}_{t}} }{ \mathrm{d} \mathbb{P}|_{\mathcal{F}_{t}}} = e^{ t \Phi(q) - q \xi_t },
	\end{align*}
	where $(\mathcal{F}_t)_{t \geq 0}$ is the underlying filtration for the L\'evy process. 
	Then, under $\mathbb{P}^{(q)}$, $(\xi_t)_{t \geq 0}$ is a L\'evy process with 
	$\mathbb{P}^{(q)} [ e^{ - \theta \xi_t} ] = e^{ - t \Phi_q(\theta)}$, where $\Phi_q(\theta) = \Phi(q+\theta) - \Phi(q)$. 

	We will be interested in setting $ q = q_x$. 
	Note in particular that
	\begin{align*}
		\mathbb{P}[ \xi_t \leq tx] = e^{ - t \Phi(q_x) } \mathbb{E}^{(q_x)} \left[ e^{q_x \xi_t } \mathrm{1}_{\{\xi_t \leq tx\}}\right].
	\end{align*}
	We may now prove the upper bound in \eqref{eq:tailsnew} by using a simple Chernoff-type argument. 
	Namely,  we have for all $t\geq 0$ and all $x\in (0, \mathbb{E}[\xi_1])$, that 
	\begin{align*}
		\mathbb{P}[ \xi_t \leq tx] \leq e^{ - t \Phi(q_x) + q_x tx } = e^{ - t R(q_x)},
	\end{align*}
	proving the upper bound \eqref{eq:tailsnew}. 

    The proof of the corresponding lower bound uses a Berry-Esseen type bound for L\'evy processes; see Appendix \ref{sec:tailsnewproof}.
\end{proof}
    
    Before proceeding with our proof of Theorem \ref{thm:levybound}, we record a 
\color{black}
simple consequence  
	of Lemma~\ref{lem:tailsnew}.

	\begin{lemma}\label{lem:bounded}
		Suppose that $(\xi_t)_{t \geq 0}$ is a L\'evy process with 
		Laplace exponent $\Phi(q)$ given in \eqref{eq:Phidef}. 
		For $h > 0$, let $\tau^{(h)}:=\inf\{t>0:\xi_t\geq h\}$. Then, there exists $h_0> 0$ such that for all $h \geq h_0$ we have 
		\begin{equation*}
			\mathbb{P}\left[\tau^{(h)}\geq h^2\right]\leq e^{-2h}.
		\end{equation*}
	\end{lemma}
	\begin{proof}
    		Let $h \geq 1$. Using the upper bound \eqref{eq:tailsnew} with $t = h^2$ and $x = 1/h$ we have
		\begin{equation*}
			\mathbb{P}\left[\tau^{(h)}\geq h^2\right]\leq \mathbb{P}[ \xi_{h^2} \leq h ] \leq e^{-h^2R(q_{1/h})}.
		\end{equation*}
Observe that $R(q_x)$ is positive and decreasing in $x$.  Indeed,  $\Phi(q)$ is increasing in $q$ and $\Phi'(q)$ is decreasing in $q$ implying that $R(q) = \Phi(q) - q \Phi'(q)$ is increasing in $q$. Since  $q_x$ solves $\Phi'(q_x) = x$, we have that as $x$ gets larger, $q_x$ gets smaller.  Thus, $R(q_x)$ is decreasing in $x$.
 In particular, there exist constants $c, h_0$ such that $R(q_{1/h}) \geq c$ whenever $h \geq h_0$. The result follows.
  %          	Now just note that in the infinite-activity case, by~\eqref{eq:Rdas2}, $R(q_{1/h})\geq 1$ 
		% for sufficiently small $h$. In the finite-activity case, $R(q_{1/h})\geq \lambda/2$ for sufficiently small $h$. 
	\end{proof}

The upper bound in the Lemma \ref{lem:tailsnew} is already sufficient to complete the proof of Theorem \ref{thm:levybound} in the finite-activity case:

\begin{proof}[Proof the upper bound in Theorem \ref{thm:levybound} in the finite-activity case]
%On the one hand we have
%\begin{align*}
%	\mathbb{P}[ \xi_t \leq tx ] \geq \mathbb{P}[ \xi_t = 0 ] = e^{ - \lambda t}, 
%\end{align*}
%where $\lambda = \nu(\MassP)$. Conversely, 
By \eqref{eq:tailsnew}, we have
\begin{align*}
	\mathbb{P}[ \xi_t \leq tx ] \leq e^{ - tR(q_x)}.
\end{align*}
The proof may be completed if we can establish that $R(q_x) \uparrow \lambda$ as $x \downarrow 0$. 
In this direction, first note that differentiating through $\Phi(q)$ in \eqref{eq:Phidef} (see, e.g., \eqref{eq:der} in the appendix), we see that $\Phi'(q) \to 0$ as $q \to \infty$. 
In particular, the solution $q_x$ to $\Phi'(q_x) = x$ tends to infinity as $x \to 0$. 
Finally, it is easy to verify that $R(q) \sim \Phi(q) \sim \lambda$ as $q\to\infty$. 
In particular, $R(q_x) \sim \lambda$ as $x \downarrow 0$, as required.
\end{proof}

%With Theorem \ref{thm:levybound} now proven in the finite-activity case, for the remainder of Section \ref{sec:levy} we will work in the infinite-activity case. 

To prepare for the application of Lemma \ref{lem:tailsnew} in our setting with small $x$, we begin with a lemma on the large-$q$ asymptotics of $\Phi(q)$ and its derivatives.

\begin{lemma} \label{lem:Phiasnew}
    Let $\Phi(q)$ be the Laplace exponent given in \eqref{eq:Phidef} associated with a dislocation measure $\nu$ satisfying \eqref{eq:1:crumbling}, and, if $\theta = 0$, satisfying \eqref{eq:1:fortheta0} as well. 
    Then as $q \to \infty$
    \begin{align}
			\Phi(q) & \sim \Gamma(1-\theta) q^\theta \ell(q), \label{eq:Phiasnew}
     \end{align}
    Let $L_1$ be as in \eqref{eq:Ldef}. In the infinite activity case
     \begin{align}
            \Phi'(q) &\sim q^{\theta-1}L_1(q),\label{eq:Phiprimeasnew}\\
            \Phi^{(j)}(q) &\sim (-1)^{j-1} q^{\theta -j } \prod_{k=1}^{j-1}(k-\theta) L_1(q) \quad j \geq 1 \label{eq:Phihighernew}
		\end{align}
        In the finite activity case there are slowly varying functions $o_j(q)$
        vanishing at infinity such that 
        \begin{align}
            \Phi^{(j)}(q) &\sim (-1)^{j-1} q^{-j } o_j(q), \quad j \geq 1. \label{eq:Phihighernew1}
		\end{align}
        \end{lemma}

We will be content to give a sketch proof here, again relegating the fully rigorous proof to Appendix \ref{sec:Phiasnewproof}.

\begin{proof}[Sketch proof of Lemma \ref{lem:Phiasnew}]
By \eqref{eq:Phidef} we have 
		\begin{equation*}
			\Phi(q) = \int_{\MassP} \left( 1 - \sum_{i \geq 1} s_i^{q+1} \right) \nu(\mathrm{d}  \mathbf{s}) \sim \int_{\MassP} \left( 1 - s_1^{q+1} \right)\nu (\rm d \mathbf{s})
		\end{equation*}
        as $q \to \infty$. Integrating by parts, we obtain 
		\begin{equation*}
			\int_{\MassP}\left(1-s_1^{q+1}\right) \nu(\mathrm{d} \mathbf{s}) = 
			\int_0^1 (q+1) u^{q} \nu(s_1 < u) \mathrm{d} u =
			\int_0^1 (q+1) (1-t)^q t^{-\theta}\ell(1/t) \mathrm{d} t,
		\end{equation*}
        where we used the behaviour \eqref{eq:1:crumbling}. 
		We can write the last integral as
		\begin{equation*}
			q^{\theta} \frac{q+1}{q} 
			\int_0^q \left(1-\frac sq \right)^q s^{-\theta}  \ell\left(\frac{q}{s}\right)\mathrm{d} s \sim q^\theta \ell(q) \int_0^\infty e^{ - s} s^{ - \theta} \mathrm{d}s = \Gamma(1-\theta)q^\theta \ell(q) \quad \text{as $q \to \infty$}.
		\end{equation*}
        That completes the sketch proof of the equation \eqref{eq:Phiasnew}.

        To obtain the next equation, \eqref{eq:Phiprimeasnew}, in the case $\theta > 0$, naively differentiate through \eqref{eq:Phiasnew} and use the definition of $\ell_0(x)$. In the case $\theta = 0$ and $\ell(x) \to \infty$, simply use the definition of $\ell_0(x)$ to argue that $\ell'(x) \sim \ell_0(x)/x$. Finally, in the case $\theta=0$, $\ell(x) \to \lambda$ use that fact that $\ell'(x) \to 0$.

        To obtain the final equation \eqref{eq:Phihighernew}, again naively differentiate through \eqref{eq:Phiprimeasnew}.
\end{proof}

\color{black}
Now that we have an asymptotic estimate for $\Phi'(q)$ as $q \to \infty$, we are able to estimate the asymptotics of $q_x$ as $x \to 0$. 

\begin{lemma} \label{lem:qx}
Let $\Phi(q)$ be the Laplace exponent given in \eqref{eq:Phidef} associated with a dislocation measure $\nu$ satisfying \eqref{eq:1:crumbling}, and, if $\theta = 0$, satisfying \eqref{eq:1:fortheta0} as well. Let $q_x$ be the solution to $\Phi'(q_x) = x$. 
Then there is a slowly varying function $J(\cdot)$ such that 
\begin{align} \label{eq:qas}
    		q_x \sim x^{-\frac{1}{1-\theta}} J(1/x).
\end{align}
In the infinite activity case, the slowly varying function $J$ satisfies the relations
\begin{align} \label{eq:qas2}
1 \sim J(s)^{\theta-1}L_1(s^{ \frac{1}{1-\theta}}J(s))
\end{align} 
and 
\begin{align} \label{eq:qas3}
1 \sim L_1(s)^{ - \frac{1}{1-\theta}} J(s^{1-\theta}/L_1(s)).
\end{align}
\end{lemma}
\begin{proof}
Differentiating through \eqref{eq:Phidef}, we see that $\Phi'(q)$ is strictly monotone decreasing in $q$, and decreases to zero as $q \to \infty$. In particular, the function $A(q) := 1/\Phi'(q)$ is strictly monotone increasing in $q$ and smooth. By \eqref{eq:Phiprimeasnew}, $A(q)$ is also regularly varying and takes the form $A(q) = q^{1-\theta}I(q)$ where $I(q) = 1/L_1(q)$. Write $s = 1/x$. Then solving $\Phi'(q) = x$ is equivalent to solving $A(q) = s$. Write $B(s)$ to be the solution, so that $A(B(s)) = s$. We are now in the setting of \eqref{eq:ro1}-\eqref{eq:ro4}, so that $B(s)$ is also regularly varying and takes the form $B(s) = s^{\frac{1}{1-\theta}}J(s)$, where $I$ and $J$ together satisfy \eqref{eq:ro3} and \eqref{eq:ro4} with $\beta=1-\theta$. Unravelling, we see that $A(B(s)) = s$ reads
\begin{align*}
\Phi'(B(1/x)) = x,
\end{align*}
that is, $B(1/x) = q_x$. In particular, now using \eqref{eq:ro2} we see that
\begin{align*}
q_x = x^{ - \frac{1}{1-\theta}} J(1/x),
\end{align*}
where $J(s)$ is slowly varying. In the infinite activity case, using \eqref{eq:ro3} and \eqref{eq:ro4} with $I(s) \sim 1/L_1(s)$, and $\beta = 1-\theta$, we see that $J(s)$ satisfies \eqref{eq:qas2} and \eqref{eq:qas3}. 

\end{proof}

\color{black}

\begin{lemma} \label{lem:condition}
There exists $x_0>0$ such that for all $0 < x \leq x_0$ we have 
\begin{align*}
r(x) :=  C  \frac{ \mathbb{E}^{(q_x)}\left[ (\xi_1 - x)^4 \right]^2 }{\mathbb{E}^{(q_x)}\left[ (\xi_1 - x)^2 \right]^4   } = \hat{L}(1/x),
\end{align*}
where $\hat{L}$ is slowly varying at infinity.
\end{lemma}
The proof of Lemma \ref{lem:condition} is found in the appendix.

Recall that Lemma \ref{lem:tailsnew} expresses the tail probabilities $\mathbb{P}[\xi_t \leq tx]$ in terms of $q\sqrt{\Phi''(q)}$ and $R(q)= \Phi(q) - q \Phi'(q)$. Our next result estimates the large-$q$ asymptotics of these quantities.

\begin{cor}
If $\theta \in [0,1)$, we have
\begin{align} \label{eq:Rish}
    R(q) \sim \Gamma(2-\theta)q^\theta \ell(q).
\end{align}
%and 
%\begin{align} \label{eq:finzi}
%q\sqrt{\Phi''(q)} \sim q^{\theta/2} L_3(q),
%\end{align}
%where \textcolor{Mahogany}{$L_3(q) = \sqrt{|(\theta-1)L_1(q)|}$} is slowly varying at infinity. 
In particular,
\begin{align} \label{eq:finzi2}
q\sqrt{\Phi''(q)} = o\left(  R(q) \right), 
\end{align}
as $q \to \infty$. 
\end{cor}

We note that \eqref{eq:Rish} reads $R(q) \sim \ell(q)$ in the case $\theta = 0$. 
\begin{proof}
In the infinite activity case, combining \eqref{eq:Phiasnew} and \eqref{eq:Phiprimeasnew}, we have
\begin{align} \label{eq:Rish2}
R(q) \sim q^\theta ( \Gamma(1-\theta)\ell(q) - L_1(q)).
\end{align} 
Now to prove \eqref{eq:Rish} in the case $\theta \in (0,1)$, we see that Equation \eqref{eq:Rish} now follows from \eqref{eq:Ldef} and the fact that $(1-\theta)\Gamma(1-\theta) = \Gamma(2-\theta)$.

To prove \eqref{eq:Rish} in the case $\theta = 0$, we first note that using \eqref{eq:Ldef}, \eqref{eq:Rish2} reads 
\begin{align} \label{eq:Rish0}
R(q) \sim \ell(q) - \ell_0(q).
\end{align}
To prove \eqref{eq:Rish}, we need to show that $\lim_{q \to \infty} \ell_0(q)/\ell(q) = 0$. To see this, note that $\ell(q) \sim \int_0^q \frac{\ell_0(x)}{x} \mathrm{d}x$ implies
\begin{align} \label{eq:rela}
\ell(q) \sim \int_{q/2}^q \frac{\ell_0(x)}{x} \mathrm{d}x + \ell(q/2) \sim \ell_0(q) \log (2) + \ell(q),
\end{align}
where to obtain the final relation above we have used the fact that the ratio $\ell_0(x)/\ell_0(q)$ converges uniformly to $1$ for $x \in [q/2,q]$ (see, e.g., \cite[Section 1.2]{bingham1989regular}). Dividing \eqref{eq:rela} through by $\ell(q)$ tells us that $\lim_{q \to \infty} \ell_0(q)/\ell(q) = 0$, and \eqref{eq:Rish0} now follows, thereby completing the proof of \eqref{eq:Rish} in the case $\theta = 0$. 
\medskip

%The equation \eqref{eq:finzi} is immediate from setting $j = 2$ in \eqref{eq:Phihighernew}. 
%\medskip

The final equation, \eqref{eq:finzi2}, follows immediately in the case $\theta  > 0$. In the case $\theta = 0$, it follows from the fact that $\ell_0(q) = o(\ell(q))$ as $q \to \infty$, which we established just a moment ago. 

In the finite activity case use~\eqref{eq:Phihighernew1} to write
\begin{align*}
R(q) \sim \ell(q) \sim \lambda.
\end{align*}
The final equation, \eqref{eq:finzi2} follows by using~~\eqref{eq:Phihighernew1} with $j=2$ and using the fact that
$o_2(q) \to 0$.

\end{proof}

We are now equipped to complete the proof of Theorem \ref{thm:levybound}.

\begin{proof}[Proof of Theorem \ref{thm:levybound}]

Using \eqref{eq:Rish} to obtain the first relation below and \eqref{eq:qas2} to obtain the second, we have 
\begin{align} \label{eq:hu1}
R(q_x) \sim \Gamma(2-\theta) q_x^\theta \ell(q_x) \sim  \Gamma(2-\theta) ( x^{-\frac{1}{1-\theta}} J(1/x))^\theta \ell\left( x^{-\frac{1}{1-\theta}} J(1/x) \right) = x^{ - \frac{\theta}{1-\theta}} \tilde{L}(1/x),
\end{align}
where 
\begin{align} \label{eq:Lnew}
\tilde{L}(s) = \Gamma(2-\theta)J(s)^\theta \ell \left( s^{\frac{1}{1-\theta}} J(s)\right).
\end{align}
We now claim that $\tilde{L}(s) \sim L(s)$, where $L(s)$ is defined in \eqref{eq:Lpuredef}. 

In the case $\theta = 0$, we can see immediately by setting $\theta = 0$ in \eqref{eq:Lnew} that $\tilde{L}(s) = \ell(sJ(s))$. 
In the infinite activity case we get $\tilde{L} = L$ as in~\eqref{eq:Lpuredef}. In the finite activity case 
$\tilde{L}(s) \sim \lambda =L(s)$ as $s J(s) \to \infty$.

As for the case $\theta \in (0,1)$, using the definition \eqref{eq:Ldef}, and then the relation \eqref{eq:dbc0again} to obtain the second, we have

\begin{align*}
\tilde{L}(s) = \frac{1-\theta}{\theta}J(s)^\theta L_1 \left( s^{\frac{1}{1-\theta}}J(s) \right) \sim \frac{1-\theta}{\theta}J(s)^\theta J(s)^{1-\theta} = \frac{1-\theta}{\theta}J(s) =:L(s),
\end{align*}
where the last definition is simply \eqref{eq:Lpuredef} in the case $\theta \in (0,1)$. 

Lemma \ref{lem:condition} states that provided $t \geq \hat{L}(1/x)$, the condition in Lemma \ref{lem:tailsnew} is satisfied. Combining \eqref{eq:tailsnew}, \eqref{eq:tailsnew2} and \eqref{eq:hu1} we obtain the statement of Theorem \ref{thm:levybound}.
\end{proof}

%%%%%%%%%%%%%%%%%%%%%%%%%%%%%%%%%%%%%%%%%%%%%%%%
%%%%%%%%%%%%%%%%%%%%%%%%%%%%%%%%%%%%%%%%%%%%%%%%
%%%%%%%%%%%%%%%%%%%%%%%%%%%%%%%%%%%%%%%%%%%%%%%%
%%%%%%%%%%%%%%%%%%%%%%%%%%%%%%%%%%%%%%%%%%%%%%%%
%%%%%%%%%%%%%%%%%%%%%%%%%%%%%%%%%%%%%%%%%%%%%%%%
%%%%%%%%%%%%%%%%%%%%%%%%%%%%%%%%%%%%%%%%%%%%%%%%
\color{black}

\section{Upper bounds for expected numbers of particles} \label{sec:upper}

\subsection{The critical choice $t = F_\theta(h)$}

    Before marching forward towards the proof of our main result, we take a moment to recall our proof outline given in Section \ref{sec:outline}, where we suggested that with $L(\cdot)$ as in the setting of Theorem \ref{thm:levybound}, the function
    \begin{align} \label{eq:Hdef}
    H(t,h) := h - C_2(\alpha,\theta)(te^{-\alpha h})^{\frac{1}{1-\theta}} L(t e^{ - \alpha h} ) 
    \end{align}
    captures the fundamental nature of the largest particles in the process, in the sense that for large times $t$, there are particles of size greater than $e^{-h}$ at time $t$ if and only if $H(t,h) \gg 0$. We recall that $C_2=C_2(\alpha,\theta) = (\alpha/\theta)^{\theta/(1-\theta)}$. 
    
    We look to find a function $F_\theta(h)$ such that $H(F_\theta(h),h) = o(h)$ as $h \to \infty$. In this direction, first let 
    \begin{align} \label{eq:Gdef3}
    G(h) = 
    \begin{cases}
    1/\lambda \qquad &:\text{$\theta = 0$, finite activity,}\\
    \text{a function satisfying \eqref{eq:complicated} below} \qquad &:\text{$\theta = 0$, infinite activity,}\\
    C_3\ell(h)^{-1},\ \ C_3 = \alpha^{-\theta} (1-\theta)^{-(1-\theta)} \Gamma(1-\theta)^{-1} \qquad &:\text{$\theta \in (0,1)$},
    \end{cases} 
    \end{align}
    where in the $\theta = 0$ infinite-activity case, with $\ell(y) \sim \int_1^y \ell_0(x) x^{-1}\mathrm{d}x$ we have
    \begin{align} \label{eq:complicated}
    1 \sim G(h)L(hG(h)), \text{ where } L(s) := \ell(sJ(s)), \text{ and where } J(s)^{-1}\ell_0(sJ(s)) \sim 1.
    \end{align}
    We observe that in fact the two definitions for $G(h)$ in \eqref{eq:Gdef3} are consistent in the $\theta = 0$ finite-activity case and $\theta \in (0,1)$ infinite-activity case: indeed, $C_3\ell(h)^{-1} \sim 1/\lambda$ in this case, since $C_3=  1$ when $\theta = 0$, and $\ell(h) \sim \lambda$ for large $h$. Conversely, in the $\theta = 0$ infinite-activity case, without imposing further structural conditions on the relationship between $\ell$ and $\ell_0$, the function $G(\cdot)$ may not be reduced any further. 

    Now define
    \begin{align} \label{eq:Fdef}
    F_\theta(h) = h^{1-\theta}e^{\alpha h} G(h).
    \end{align}
    
    We prove the following lemma:

    \begin{lemma} \label{lem:Gcheck}
    Let $G$ be defined as in \eqref{eq:Gdef3} and let $\delta \in \mathbb{R}\setminus\{0\}$. Then with this choice of $G$ in $F_\theta(h)$ as in \eqref{eq:Fdef}, we have 
    \begin{align} \label{eq:eitheror}
    H((1+\delta)F_\theta(h), h) \to \begin{cases}
        +\infty, \qquad \text{$\delta < 0$},\\
        -\infty, \qquad \text{$\delta > 0$},
    \end{cases}
    \end{align}
    as $h\rightarrow\infty$.
    \end{lemma}
    \begin{proof}
    By plugging $t=F_\theta(h)$ into \eqref{eq:Hdef}, one can see that the statement of the lemma is equivalent to verifying that $G(h)$ satisfies 
    \begin{align} \label{eq:Gdef}
    1 \sim C_2 G(h)^{\frac{1}{1-\theta}} L(h^{1-\theta}G(h)),
    \end{align}
    where $C_2 = C_2(\alpha,\theta)$. Thus we need to verify that \eqref{eq:Gdef} is satisfied in each of the three cases.
    
    In the first case, $\theta = 0$ with finite activity, we note that $L(s) \sim \lambda$, and $C_2 = 1$, so that \eqref{eq:Gdef} is immediate with $G(h) = 1/\lambda$. 
    
    In the second case with $\theta = 0$ and infinite activity, we note that again $C_2 = 1$, so that \eqref{eq:Gdef} reads $1 \sim G(h)L(hG(h))$, so that \eqref{eq:Gdef} is satisfied by definition in \eqref{eq:Gdef3}. 
    
    In the final case where $\theta \in (0,1)$, recall that \eqref{eq:beak} states that $(C_1\ell(s))^{-\frac{1}{1-\theta}} L \left( s^{1-\theta}/\ell(s) \right) \sim 1$. Plainly if we choose $G(h)$ so that $C_2G(h)^{\frac{1}{1-\theta}} = (C_1\ell(h))^{- \frac{1}{1-\theta}}$, then \eqref{eq:Gdef} is satisfied. Rearranging, we obtain $G(h) := \left( C_1C_2^{1-\theta}\ell(h) \right)^{-1}$. Now use the definitions $C_1 = \theta^\theta(1-\theta)^{1-\theta}\Gamma(1-\theta)$ and $C_2 = (\alpha/\theta)^{\theta/(1-\theta)}$ to match the definition in \eqref{eq:Gdef3}. 
    \end{proof}

\subsection{Remaining proof overview}

    In the previous section, we outlined heuristics suggesting how taking $t = (1+o(1))F_\theta(h)$ is critical for the existence of particles of size $e^{-h}$ at time $t$. With this quantity in mind, we now overview the remaining parts of the proof of our main result:
	\begin{itemize}
    		\item In the remainder of Section~\ref{sec:upper}, 
			we prove that if $t \geq (1+\delta)F_\theta(h)$,  
    			then the probability that there exists a fragment of size greater than $e^{-h}$ 
			at time $t$ decays exponentially in $h$.  
    		\item In the following section, Section~\ref{sec:lower}, 
			we prove that if $t \leq (1-\delta)F_\theta(h)$,  
    			then the probability that there is \emph{no} fragment of size greater than 
			$e^{-h}$ at time $t$ also decays exponentially in $h$.  
	\end{itemize}

	We establish the first claim — that with high probability there is no fragment of size $\geq e^{-h}$  
	at time $t \geq (1+\delta)F_\theta(h)$ — by means of a first-moment argument.  
	Define
	\begin{align*}
    		E(t,h) := \mathbb{E}\!\left[\, \#\left\{\text{fragments at time $t$ of size at least $e^{-h}$}\right\} \,\right].
	\end{align*}
	Our main task in this section is to prove the following result.

	\begin{proposition} \label{prop:infiniteupper}
		Let $ F_\theta(h)$ be defined as in \eqref{eq:Fdef} with $G$ as in \eqref{eq:Gdef3}. Then, for every $\delta > 0$ there exists 
		$c_\delta > 0$ and $h_\delta > 0$ such that whenever $h \geq h_\delta$ we have
		\begin{align*}
			E((1+\delta)F_\theta(h),h) \leq 2e^{ - c_\delta h}.
		\end{align*}
	\end{proposition}

\subsection{A general tool}

	We now develop a general upper bound for $E(t,h)$ that we will exploit in the 
	proof of Proposition \ref{prop:infiniteupper}.
	Setting $f(x) = \mathrm{1}_{\{ x \geq e^{- h} \}}$ in the 
	many-to-one formula \eqref{eq:m21f}, we have 
	\begin{align*} 
		E(t,h)
		= \mathbb{E}\left[ e^{ \xi_{\rho(t)} } \mathrm{1}_{\{ \xi_{\rho(t)} \leq h \}} \right] 
		\leq e^h \mathbb{P}\left[ \xi_{\rho(t)} \leq h \right].
	\end{align*}

We begin working on the proof of Proposition \ref{prop:infiniteupper} by looking for an 
upper bound on $\mathbb{P}\left[ \xi_{\rho(t)} \leq h \right]$. In this direction, we seek control of the amount of time that the process $(\xi_{\rho(t)})_{t \geq 0}$ spends in space windows of size $\varepsilon$. Namely, for $\varepsilon > 0$ and $i \in \{0,1,\ldots\}$, define the random variable 
		 \begin{align*}
			X_i = X_i(h,\varepsilon) 
			:= \int_0^\infty \mathrm{1}_{\{ \xi_{\rho(t)} \in (h-(i+1)\varepsilon, h - i \varepsilon ]  \}}~ \mathrm{d}t
		\end{align*}
		to be the amount of time spent by the time-changed L\'evy
		process $(\xi_{\rho(t)})_{t \geq 0}$ in the interval $(h-(i+1)\varepsilon,h-i\varepsilon]$. 
		With a large integer $N=N(h) \geq 1$ to be determined later, we can write
		\begin{align*}
			\int_0^\infty \mathrm{1}_{\{ \xi_{\rho(t)} \leq h  \}}~ \mathrm{d}t
			= \sum_{ i =0}^{N-1} X_i + Y_N,
		\end{align*}
		where 
		\begin{align*}
			Y_N := \int_0^\infty \mathrm{1}_{\{ \xi_{\rho(t)} \in (0, h - N  \varepsilon ]  \}}~ \mathrm{d}t.
		\end{align*}
		Now, since $\xi$ is a subordinator,
		\begin{align*}
			\mathbb{P}\left[ \xi_{\rho(t)} \leq h \right]
			= \mathbb{P}\left[ \sum_{i = 0}^{N-1} X_i(h,\varepsilon) + Y_N  \geq  t \right].
		\end{align*}
		Recall that $\tau^{(\varepsilon)} = \inf \{ t > 0 \,:\, \xi_t \geq \varepsilon \}$ is the hitting time of $\varepsilon$ of the (un-time-changed) L\'evy process.
		While $\xi_{\rho(t)}$ is contained inside the interval 
		$[h-(i+1)\varepsilon,h-i\varepsilon)$, it behaves like the L\'evy process 
		$(\xi_t)_{t \geq 0}$ but with its clock slowed down so that a unit of time can take up to $e^{\alpha(h-i\varepsilon)}$ times longer. In particular, we have the stochastic domination
		\begin{align*}
			X_i \leq^{(d)} e^{ \alpha(h-i\varepsilon) } \tau^{(\varepsilon)}.
		\end{align*}
		In other words, for all $w \geq 0$ we have
		\begin{align*}
			\mathbb{P}[ X_i  \geq w  ] \leq 
			\mathbb{P} \left[ e^{ \alpha(h-i\varepsilon) } \tau^{(\varepsilon)}\geq w \right].
		\end{align*}
		Also, we have the stochastic domination 
		\begin{align*}
			Y_N \leq^{(d)} e^{\alpha(h-N\varepsilon)} \tau^{(h-N\varepsilon)} 
			\leq^{(d)} e^{\alpha(h-N\varepsilon)} \tau^{(h)}.
		\end{align*}
		More generally, while the $(X_i)_{i \geq 0}$ and $Y_N$ are not independent random variables, 
		if $(\tau_i)_{i \geq 1}$ are independent random variables 
		identically distributed like $\tau^{(\varepsilon)}$; and $\tau^{(h)}$ is a further 
		independent random variable distributed like the hitting time of $h$, 
		we have the stochastic domination of sums
		\begin{align*}
			\sum_{i = 0}^{N-1} X_i(h,\varepsilon) + Y_N \leq^{(d)} 
			\sum_{i = 0}^{N-1 } e^{\alpha(h-i\varepsilon)} \tau_i + e^{\alpha(h-N\varepsilon)} \tau^{(h)}.
		\end{align*} 
		In summary, it follows that with $(\tau_i)_{i \geq 0}$ 
		i.i.d.\ with $\tau_i \sim \tau^{(\varepsilon)}$, we have 
		\begin{align*} 
			E(t,h) \leq e^h 
			\mathbb{P} 
			\left[\sum_{i = 0}^{N-1 } e^{\alpha(h-i\varepsilon)} \tau_i + e^{\alpha(h-N\varepsilon)} \tau^{(h)} \geq t \right].  
	\end{align*}
	Writing $t = e^{\alpha h} s$ we have 
	\begin{align*} 
		E(t,h)  \leq e^h \mathbb{P} 
		\left[\sum_{i = 0}^{N-1 } e^{ -\alpha i\varepsilon} \tau_i + e^{- \alpha N\varepsilon} \tau^{(h)} \geq s \right].  
	\end{align*}

    We would like to now simplify the problem, arguing that when $N$ is large, the contribution of the term $ e^{- \alpha N\varepsilon} \tau^{(h)} $ is negligible. 
	In particular, since 
	\begin{align*} 
	\mathbb{P} \left[\sum_{i = 0}^{N-1 } e^{ -\alpha i\varepsilon} \tau_i + e^{- \alpha N\varepsilon} \tau^{(h)} \geq s \right]
	\leq  
	\mathbb{P}\left[ \tau^{(h)} \geq h^2 \right] + 
	\mathbb{P} \left[\sum_{i = 0}^{N-1 } e^{ -\alpha i\varepsilon} \tau_i \geq s - e^{- \alpha N\varepsilon}  h^2\right],
	\end{align*}
	by writing $s_1 := s - e^{- \alpha N\varepsilon}  h^2$ and using Lemma \ref{lem:bounded} provided $h$ is sufficiently large, we have 
	\begin{align*} 
		E(t,h)  \leq 
		e^{-2h} + 
		e^h \mathbb{P} \left[\sum_{i = 0}^{N-1 } e^{ -\alpha i\varepsilon} \tau_i \geq s_1 \right].
	\end{align*}
    Let $N = \lceil 2\log h/\alpha \varepsilon \rceil$, so that $e^{ - \alpha N \varepsilon} h^2 \leq 1$. Accordingly,
	\begin{align} \label{eq:naranja}
		E(t,h)  \leq 
		e^{-2h} + 
		e^h \mathbb{P} \left[\sum_{i = 0}^{N-1 } e^{ -\alpha i\varepsilon} \tau_i \geq s-1 \right].
	\end{align}
    
We now develop \eqref{eq:naranja} further to prove the following proposition.

	\begin{proposition} \label{prop:genupper}
		Let $\delta > 0$ and $\varepsilon > 0$. 
		Let $N = \lceil 2\log h/\alpha \varepsilon \rceil$ as above. Suppose that $t > 0$ is such that if we set $s = t e^{ - \alpha h}$ we have 
        \begin{align}\label{eq:condition_s}
        s^N \leq e^{ \delta h} \qquad \text{and} \qquad \delta s \geq N+1.
        \end{align}
		Then 
		\begin{align*}
			E(t,h) \leq e^{-2h} + e^{(1+\delta)h} J((1-\delta)s,\varepsilon),
		\end{align*}
		where
		\begin{align*}
			J(u,\varepsilon) := 
			\sup \left\{ \prod_{i \geq 0} \mathbb{P} \left[ \tau^{(\varepsilon)} >\varepsilon y_i \right] 
			\,:\, (y_0,y_1,\ldots)\in[0,\infty)^{\mathbb{N}} \mbox{ such that } \varepsilon \sum_{i \geq 1} y_i e^{ - \alpha i \varepsilon } \geq u \right\} 
\end{align*}
		and $\tau^{(\varepsilon)} \overset{(d)}{=} \inf\{ t > 0 \,:\, \xi_t \geq \varepsilon \}$. 
	\end{proposition}

Condition \eqref{eq:condition_s} holds if  
\begin{align*}
t\in \left(e^{\alpha h}\frac{2}{\delta}\left[1+\frac{\log(h)}{\varepsilon \alpha  }\right], e^{\alpha h(1+\tfrac{\varepsilon \delta}{2\log(h)})}  \right).
\end{align*}
    
	In the course of proving Proposition \ref{prop:genupper} we will require the following lemma.

	\begin{lemma} \label{lem:summa}
		Let $A_0,\ldots,A_{N-1}$ be independent nonnegative random variables. Then, for $s > N$
		\begin{align*} 
			&\mathbb{P}[A_0+\ldots+A_{N-1} \geq s] \nonumber \\
			&\leq s^N \sup \left\{ \prod_{i=0}^{N-1} \mathbb{P}[A_i \geq y_i ] \,:\, (y_0,\ldots,y_{N-1})\in[0,\infty)^N \mbox{such that }\sum_{i = 0}^{N-1}y_i \geq s- N \right\}.
			\end{align*}
	\end{lemma}
	
\begin{proof}
Let $X_i:=\lfloor A_i\rfloor$, $i=0,\dots,N-1$. Then $X_i\in\mathbb{Z}_{\ge0}$ and
\[
\left\{\sum_{i=0}^{N-1} A_i \ge s\right\}\subseteq
\left\{\sum_{i=0}^{N-1} X_i \ge s-N\right\},
\]
since $X_i\leq A_i<X_i+1$.

For $t >0 $, we have the (generous) inclusion
\[
\left\{\sum_{i=0}^{N-1} X_i \ge t\right\}
\subseteq
\bigcup_{\substack{a\in\mathbb{Z}_{\ge0}^N:\\ a_i\le \lfloor t\rfloor,\ \sum_i a_i \ge t}}
\bigcap_{i=0}^{N-1}\{X_i\ge a_i\}.
\]
Applying the union bound, the fact that 
$\{X_i\ge a_i\}=\{A_i\ge a_i\}$ and independence, we obtain
\[
	\mathbb{P}\!\left[\sum_{i=0}^{N-1} A_i \ge s\right]
\le
\sum_{\substack{a\in\mathbb{Z}_{\ge0}^N:\\ a_i\le \lfloor s-N\rfloor,\ \sum_i a_i \ge s-N}}
\prod_{i=0}^{N-1}\mathbb{P}[A_i\ge a_i].
\]
Bounding the sum by its supremum times the number of terms, and noting that the number of
$N$-tuples with $0\le a_i\le \lfloor s-N\rfloor$ is at most $(\lfloor s-N\rfloor+1)^N\le s^N$, 
for $s>N$, we get
\[
	\mathbb{P}\!\left[\sum_{i=0}^{N-1} A_i \ge s\right]
\le
s^{\,N}
\sup_{\substack{a\in\mathbb{Z}_{\ge0}^N:\\ \sum_i a_i \ge s-N}}
\prod_{i=0}^{N-1}\mathbb{P}[A_i\ge a_i].
\]
Finally, extend the supremum from integer $a_i$ to real $y_i\in[0,\infty)$ (which can only increase
the supremum), preserving the constraint $\sum_i y_i\ge s-N$. This yields the claimed bound.
\end{proof}

	\begin{proof}[Proof of Proposition \ref{prop:genupper}]

    We use the previous lemma to control the sum in \eqref{eq:naranja}. Setting $A_i = \tau_ie^{ - \alpha i \varepsilon }$ in Lemma \ref{lem:summa} and using this inequality in \eqref{eq:naranja}, we obtain
    \begin{align*}
    E(t,h) \leq e^{ -2h} + s^N e^h J_N,
    \end{align*}
where
\begin{align*}
	J_N = \sup \left\{  \prod_{i=0}^{N-1} \mathbb{P}\left[ \tau_i e^{ - \alpha i \varepsilon} \geq y_i \right] \,:\, (y_0,\ldots,y_{N-1}) \in [0,\infty)^N : \sum_{i=0}^{N-1} y_i \geq s-N-1 \right\}.
\end{align*}

Recall that $N = \lceil 2\log(h)/\alpha \varepsilon \rceil$ and $s = t e^{ - \alpha h}$. (We will ultimately take $s$ of the order $h^{1-\theta}$ times slowly varying terms.) Let $\delta > 0$. Provided $s\in\mathbb R$ is a number sufficiently large compared to $N$ so that $(N+1) \leq \delta s$ but $s$ is not too large compared to $h$ so that $s^N \leq e^{\delta h}$, then 
  \begin{align} \label{eq:juk1}
     E(t,h) \leq e^{ -2h} + e^{(1+\delta)h} J_N',
    \end{align}
where 
\begin{align*}
	J_N' = \sup \left\{  \prod_{i=0}^{N-1} \mathbb{P} \left[ \tau_i e^{ - \alpha i \varepsilon} \geq y_i \right] : (y_0,\ldots,y_{N-1}) \in [0,\infty)^N : \sum_{i=0}^{N-1} y_i \geq (1-\delta)s \right\}.
\end{align*}
Now note that by replacing $y_i $ with $\varepsilon y_i e^{ - \alpha i \varepsilon}$, and then extending the range of the supremum from $(y_0,\ldots,y_{N-1}) \in [0,\infty)^N$ to $(y_0,y_1,\ldots) \in [0,\infty)^{\mathbb{N}}$ (one can always set $y_i = 0$ for $i \geq N$), we have
\begin{align*}
J_N' \leq J( (1-\delta)s, \varepsilon),
\end{align*}
where $J(u,\varepsilon)$ is as in the statement of the proposition. Using the previous inequality in \eqref{eq:juk1} we obtain the result.
\end{proof}

\color{black}
\subsection{Proof of Proposition \ref{prop:infiniteupper}}

	In this section we work towards proving Proposition \ref{prop:infiniteupper}, 
	which is an upper bound for the expected number of large particles in 
	infinite-activity fragmentation processes.
    
	Fix $\varepsilon > 0$.
	We now use Theorem \ref{thm:levybound} to control the term $J(u,\varepsilon)$ occurring in 
	the setting of Proposition \ref{prop:genupper}. Here we have $J(u,\varepsilon) \leq e^{ - K(u,\varepsilon)}$, where
	\begin{align} \label{eq:Kdef}
		K(u,\varepsilon) := 
		\inf \left\{  \varepsilon 
		\sum_{ i \geq 0}  
   L\left(y_i\right)
        y_i^{\frac{1}{1-\theta}}  
		\,:\, (y_0,y_1,\ldots)\in[0,\infty)^{\mathbb{N}}, \:
		\varepsilon \sum_{ i \geq 0} y_i e^{ - \alpha i \varepsilon } 
		\geq u \right\},
	\end{align}
	and we recall from the statement of Theorem \ref{thm:levybound} that $L$ is 
	a slowly varying function at infinity. 
	Before giving our main result on the asymptotics of $K(u,\varepsilon)$,  we present the following 
	lemma which describes a minimisation problem that provides a proxy for the asymptotics 
	of $K(u,\varepsilon)$ after correct normalisation for the terms depending on $u$. 

	\begin{lemma} \label{lem:simple}
	Let $\alpha,\varepsilon > 0$ and $\theta \in (0,1)$. We have 
		\begin{align*}
			C(\theta,\alpha,\varepsilon)  :=& 
			\inf \left\{  \varepsilon \sum_{ i \geq 0}  z_i^{\frac{1}{1-\theta}}  
				: (z_0,z_1,\ldots) \in [0, +\infty)^\mathbb{N}, \:
			\varepsilon \sum_{ i \geq 0} z_i e^{ - \alpha i \varepsilon } 
			\geq 1 \right\}\\ 
			=&  \left( \frac{1-e^{ - \alpha \varepsilon/\theta} }{ \varepsilon} \right)^{\frac{\theta}{1-\theta}}.
		\end{align*}
		Moreover, the coordinates of the sequence 
		$(z_0,z_1,\ldots)\in[0,\infty)^{\mathbb{N}} $ attaining the minimum take the form
		\begin{align*}
			z_i = \frac{1-e^{- \alpha  \varepsilon/\theta}}{\varepsilon} 
			e^{ - \alpha i \varepsilon(1/\theta-1)}.
		\end{align*}
	\end{lemma}

	\begin{proof}
		This follows from a simple Lagrange multipliers calculation, 
		solving 
		\begin{equation*}
			0 = \frac{\partial}{ \partial z_i} 
			\left(\sum_{ i \geq 0}  z_i^{\frac{1}{1-\theta}}  
			- \mu \sum_{i \geq 0} z_i e^{ - \alpha i \varepsilon } \right)
		\end{equation*}
		for a given parameter $\mu$, and using the fact
		that the minimiser $(z_0,z_1,\ldots)$ satisfies 
		$\varepsilon \sum_{ i \geq 0} z_i e^{ - \alpha i \varepsilon } \geq 1$.
	\end{proof}

	In order to relate the asymptotics of $K(u,\varepsilon)$ to the variational 
	problem in Lemma \ref{lem:simple}, we will require the following technical lemma, whose proof we postpone to the appendix.

	\begin{lemma} \label{lem:tech}
		Let $\alpha,\varepsilon > 0$ and $\theta \in (0,1)$. Let $f_n:[0,\infty) \to [0,\infty)$ be a sequence of functions 
		such that there exist a constant $C>0$ so that for each $n\in\mathbb N$, 
		$f_n$ is bounded on $[0,1]$, and 
		$f_n(x) \geq C x^{ - \delta_n}$ for $x \geq 1$, for some 
		sequence $\delta_n \to 0$. 
		Suppose further that $f_n(x) \to 1$, uniformly for $x$ in 
		compact subsets $[a,A] \subseteq (0,\infty)$, with $a,A\in(0,\infty)$. Then
		\begin{align*}
			\liminf_{n \to \infty} 
			\inf_{ \varepsilon \sum_{i=0}^{\infty} z_ie^{ - \alpha i \varepsilon} \geq 1  } 
			\varepsilon  \sum_{i=0}^{\infty} z_i^{\frac{1}{1-\theta}}f_n(z_i) 
			\geq C(\theta,\alpha,\varepsilon),
		\end{align*}
		where $C(\theta,\alpha,\varepsilon)$ is given in Lemma \ref{lem:simple}.
	\end{lemma}

	We now use the previous lemma to describe the asymptotics of $K(u,\varepsilon)$.

	\begin{lemma} \label{lem:Klower}
		For all $\delta > 0$, there exists $u_\delta >0$
		such that whenever $u \geq u_\delta$, we have 
		\begin{align*}
			K(u,\varepsilon) 
			\geq (1-\delta) L(u) u^{\frac{1}{1-\theta}} C(\theta,\alpha,\varepsilon),
		\end{align*}
        where $C(\theta,\alpha,\varepsilon)$ is given in Lemma \ref{lem:simple}.
	\end{lemma}

	\begin{proof}
		Note that $K(u,\varepsilon)$ may be written
		\begin{align*}
			K(u,\varepsilon) = L(u)u^{\frac{1}{1-\theta}} 
			\inf \left\{  \varepsilon \sum_{ i \geq 0}  f_u(z_i) z_i^{\frac{1}{1-\theta}}  
			\,:\, (z_0,z_1,\ldots)\in[0,\infty)^{\mathbb{N}}, \: 
			\varepsilon \sum_{ i \geq 0}z_i e^{ - \alpha i \varepsilon } \geq 1 \right\},
		\end{align*}
		where
		\begin{align*}
			f_u(z) = L(uz)/L(u).
		\end{align*}
		Since $L$ is a slowly varying function, 
		$f_u$ converges pointwise to $1$. 
		Moreover, according to the uniform convergence theorem 
		(\cite[Section 1.2]{bingham1989regular}), this convergence is 
		uniform on compact subsets of $(0,\infty)$. It is also possible to show 
		using Karamata's representation theorem \cite[Section 1.3]{bingham1989regular} 
		that $f_u(z) \geq C z^{ o_u(1)}$ for $z \geq 1$. 
		Proposition \ref{prop:infiniteupper} now follows 
		from applying the previous lemma to $(f_u)$. 
	\end{proof}

	\begin{proof}[Proof of Proposition \ref{prop:infiniteupper}]
		Let $\delta > 0$. 
		By Proposition \ref{prop:genupper}, for all sufficiently large 
		$s = te^{-\alpha h}$ and $h$ we have 
		\begin{align*}
			E(t,h) \leq e^{-2h} + e^{(1+\delta)h} J((1-\delta)s , \varepsilon) .
		\end{align*}
		Now as noted above, by Theorem \ref{thm:levybound} we 
		have $J(u,\varepsilon) \leq e^{ - K(u,\varepsilon)}$, where 
		$K(u,\varepsilon)$ is as in \eqref{eq:Kdef}. 
		Now by Lemma \ref{lem:Klower} for all $s$ sufficiently large, we have
		\begin{align*}
			E(t,h) \leq 
			e^{-2h} + \exp \left\{ (1 + \delta)h - (1-\delta)^{1+\frac{1}{1-\theta}} 
			L((1-\delta)s) s^{\frac{1}{1-\theta}} C(\theta,\alpha,\varepsilon) \right\}. 
		\end{align*}
		Note that 
		\begin{align*}
			C(\theta,\alpha,\varepsilon) = 
			\left( \frac{1-e^{ - \alpha  \varepsilon/\theta} }{ \varepsilon} \right)^{\frac{\theta}{1-\theta}} \to (\alpha/\theta)^{\theta/(1-\theta)}, \qquad \mbox{ as } \varepsilon\rightarrow 0,
		\end{align*}
		in an increasing way. Thus, we may choose $\varepsilon>0$ sufficiently 
		small so that $C(\alpha,\theta,\varepsilon) \geq (1-\delta)(\alpha/\theta)^{\theta/(1-\theta)}$. Then
		\begin{align*}
			E(t,h) \leq e^{-2h} + \exp \left\{ (1 + \delta)h - (1-\delta)^{2+\frac{1}{1-\theta}}L((1-\delta)s)
            s^{\frac{1}{1-\theta}}( \alpha/\theta)^{\theta/(1-\theta)} \right\}. 
		\end{align*}
		After a brief calculation using Lemma \ref{lem:Gcheck}, one may check that after 
		choosing $c = c_{\alpha,\theta}$ that does not depend on $\delta$, sufficiently large,  we have 
		$E((1+c\delta)F_\theta(h),h) \leq e^{ - 2 h}+e^{ - \delta h}$. 
		Replacing $c \delta$ with $\delta$, we obtain Proposition \ref{prop:infiniteupper} as written.

        %actually, we can select c=4 and the result h
\end{proof}

\section{Lower bounds for existence of large particles} \label{sec:lower}

	In this section, it will be useful to introduce the following notation. Let
	\begin{align*}
		I(t,h) := \left\{ \text{At time $t$ there are no fragments of size greater than or equal to $e^{-h}$} \right\}
	\end{align*}
    	be the event that every fragment at time $t$ has size $< e^{-h}$. 
    	During the course of this section, 
	we prove the following result that complements Proposition~\ref{prop:infiniteupper}.

	% \begin{proposition} \label{prop:finitelower}
	% 	Let $F_0(h) = he^{\alpha h}/\lambda$ be associated with a finite 
	% 	activity fragmentation process with total mass $\lambda$. 
	% 	Then for every $\delta > 0$ there exists $c_\delta > 0$ and $h_\delta$ 
	% 	such that whenever $h \geq h_\delta$ we have
	% 	\begin{align*}
	% 		D((1-\delta)F_0(h),h) \leq e^{ - c_\delta h}.
	% 	\end{align*}
	% \end{proposition}

	\begin{proposition} \label{prop:infinitelower}
		Let $ F_\theta(h)$ defined as in \eqref{eq:Fdef} with $G(h)$ as in \eqref{eq:Gdef3}. 
        	For every $\delta > 0$ there exists a random $h_1(\delta) > 0$ 
		such that for all $h \geq h_1(\delta)$, $I((1-\delta)F_\theta(h),h)$ does not occur. 
		In other words, for all sufficiently large $h$ greater than some random $h$, 
		there exists a particle of size $\geq e^{-h}$ at time $t = (1-\delta)F_\theta(h)$.
	\end{proposition}

	In the proof of Proposition \ref{prop:infinitelower}, we will use the tools developed in Section \ref{sec:CMJ}.
	% In the proofs of both Proposition \ref{prop:finitelower} and 
	% Proposition \ref{prop:infinitelower}, we will use tools 
	% developed in Section \ref{sec:CMJ}.
	Recall that in the context of Section \ref{sec:independent} that 
	we say that a collection of fragments existing at various 
	points of time is an antichain if no fragment is ancestor 
	to any other fragment in the collection. We prove in that 
	Section \ref{sec:independent} that when $h$ is large, one can choose a 
	large antichain of fragments of sizes between $e^{ - h}$ and 
	$e^{ -(h-\varepsilon)}$. The future behaviour of distinct fragments in an antichain is independent.

\subsection{Lower bounds for probabilities of paths for L\'evy processes}

	The main idea in the proof of Proposition~\ref{prop:infinitelower}
	is that we construct an event under which the time-changed 
	L\'evy process $(\xi_{\rho(t)})_{t \geq 0}$ takes a path which leads to a large particle at time $t$. 

	For $N \in \mathbb{N}$ let $t_0,\ldots,t_{N-1} > 0$ be a collection of times. 
	Define $q_0 := 0$ and $q_j := \sum_{i=1}^j t_{N-j}$, so that $q_j - q_{j-1} = t_{N-j}$ 
	for $1 \leq j \leq N$. We begin by observing that since $(\xi_t)_{t \geq 0}$ is a 
	Markov process with stationary increments, for any starting position $-x$ with $x \in [0,\varepsilon]$ we have
	\begin{align} \label{eq:Markov}
		\mathbb{P}_{-x} \left[ \xi_{q_j } \in [(j-1)\varepsilon, j \varepsilon ] ~ 
		\text{for all $j = 1,\ldots,N$} \right] \geq \prod_{ i = 0}^{N-1} f_\varepsilon(t_i),
	\end{align}
	where
	\begin{align*}
		f_\varepsilon(t) := 
		\inf_{ u \in [0,1] } \mathbb{P}_{-u\varepsilon} [ \xi_t \in [0,\varepsilon] ] = 
		\inf_{u \in [0,1]} \left( \mathbb{P}_0[ \xi_t \leq (1+u)\varepsilon ] - 
		\mathbb{P}_0 [ \xi_t \leq u\varepsilon] \right).
	\end{align*}

	Our next result provides a lower bound for $f_\varepsilon(t)$ using Theorem \ref{thm:levybound}.

	\begin{lemma} \label{lem:pushlower}
		For every $\varepsilon, \delta \in (0,1/2)$ there exists a 
		$t_{\delta,\varepsilon} >0 $ such that whenever $t \geq t_{\delta,\varepsilon}$ we have 
		\begin{align*}
			f_\varepsilon(t) \geq 
			\exp \left( - (1 + \delta) L(t) t^{\frac{1}{1-\theta}} \varepsilon^{-\theta/(1-\theta)} \right),
		\end{align*}
        where $L(t)$ and $\theta \in [0,1)$ are as in the statement of Theorem \ref{thm:levybound}.
	\end{lemma}

	\begin{proof}
We would like to apply the lower bound in Theorem \ref{thm:levybound}, which states that for any $\delta>0$ there exists $x_0 > 0$, a slowly varying $\hat{L}:(0,\infty) \to (0,\infty)$, and  a function $\Delta \colon [0,+\infty) \to [0,+\infty)$ with $\Delta(x) \to 0$ as $x \to 0$, such that 
\begin{equation} \label{eq:repeat}
	\mathbb{P}\left[ xt \leq \xi_t \leq xt +\sqrt{t x^{ \frac{2-\theta}{1-\theta} - \delta}} \right] 
	\geq \exp \left\{ -  t x^{ - \frac{\theta}{1-\theta}} L(1/x) (1 + \Delta(x)) \right\},
\end{equation}
whenever $t \geq \hat{L}(1/x)$ and $0 < x \leq x_0$.

Choose $\rho = \rho_\eta > 0$ sufficiently small so that $(1 - \rho)^{ - \frac{\theta}{1-\theta}} \leq 1 + \eta$. 
Consider using Theorem \ref{thm:levybound} with $\delta=1/2$ and $x = x_{t,\rho,u} = \varepsilon(1-\rho+u)/t$ in \eqref{eq:repeat}. Note that
\begin{align*}
tx^{\frac{2-\theta}{1-\theta}-\delta} = \varepsilon^{\frac{2-\theta}{1-\theta} - \delta} t^{ 1 + \delta - \frac{2-\theta}{1-\theta}}(1-\rho+ u )^{\frac{2-\theta}{1-\theta} - \delta} \leq C^2 t^{-1/2},
\end{align*}
where $C^2 = 2^{\frac{2-\theta}{1-\theta} - \delta} $ does not depend on $u$. With this choice of $x$, now note that provided $t \geq t_{\rho,\varepsilon}$ is sufficiently large so that $Ct^{-1/4} \leq \rho \varepsilon$ we have the inclusion of intervals
\begin{align*}
\left[  xt , xt +\sqrt{t x^{ \frac{2-\theta}{1-\theta} - \delta}} \right] \subseteq \left[ \varepsilon(1-\rho+u) ,\varepsilon(1-\rho+u)  + C t^{-1/4} \right] \subseteq [ u\varepsilon, (1+u)\varepsilon].
\end{align*}
In particular, by \eqref{eq:repeat} it follows that for all sufficiently large $t \geq t_{\eta, \varepsilon}$, for any $u \in [0,1]$ we have 
\begin{align}  \label{eq:repeat2}
\mathbb{P}\left[ \xi_t \in [u\varepsilon,(1+u)\varepsilon] \right] \geq  \exp \left\{ - (1-\rho+u)^{- \frac{\theta}{1-\theta}} t^{\frac{1}{1-\theta}} \varepsilon^{ - \frac{\theta}{1-\theta}}  L ( 1/x_{t,\rho,u}) \left(1 + \Delta ( x_{t,\rho,u} \right)   \right\}.
\end{align}

We can now control all of the terms on the right-hand side of \eqref{eq:repeat2}. Let $\eta > 0$ be arbitrary. First of all, it is a basic property of slowly varying functions that for all fixed $c, \delta > 0$, there exists $t_{c,\delta}$ such that whenever $t \geq t_{c,\delta}$ guarantees that $L(ct) \leq (1+\delta) L(t)$. In particular 
\begin{align} \label{eq:y1}
L(1/x_{t,\rho,u}) \leq (1+\eta) L(t) \qquad \text{for all $t \geq t_\varepsilon$},
\end{align}
where $t_\varepsilon$ does not depend on $u$ or $\rho$. 

Second, by the fact that $\Delta(x)$ converges to zero as $x \downarrow 0$, we have
\begin{align} \label{eq:y2}
1+\Delta(x_{t,\rho}) \leq 1 + \eta \qquad \text{for all $t \geq t_\varepsilon'$}.
\end{align}
Finally, by the definition of $\rho = \rho_\eta > 0$, we have
\begin{align} \label{eq:y3}
(1- \rho + u)^{ - \frac{\theta}{1-\theta}} \leq (1- \rho)^{ - \frac{\theta}{1-\theta}}  \leq 1 + \eta.
\end{align} 

Using \eqref{eq:y1}, \eqref{eq:y2} and \eqref{eq:y3} in \eqref{eq:repeat2} we see that for all $t \geq t_{\eta,\varepsilon}$ we have 
\begin{align}  \label{eq:repeat3}
\mathbb{P}\left[ \xi_t \in [u\varepsilon,(1+u)\varepsilon] \right] \geq  \exp \left\{ - (1+\eta)^3 L(t)t^{ \frac{1}{1-\theta}} \varepsilon^{ - \frac{\theta}{1-\theta}}  \right\}.
\end{align}
Now given any $\delta > 0$, we may choose $\eta > 0$ sufficiently small so that $(1+\eta)^3 \leq 1 + \delta$, and obtain the result for all $t \geq t_{\delta,\varepsilon}$. 

	\end{proof} 

	With $t_{\delta,\varepsilon}$ as in the statement of Lemma \ref{lem:pushlower}, 
	let $t_0,\ldots,t_{N-1} \geq t_{\delta,\varepsilon}/\varepsilon$, and define
	\begin{align} \label{eq:Tdef}
		T(t_0,\ldots,t_{N-1}) 
		:= e^{ \alpha ( N - 2 )\varepsilon } \varepsilon 
		\sum_{ j = 0 }^{N-1} t_j e^{ - \alpha \varepsilon j } .
	\end{align}

	\begin{lemma} \label{lem:carefulbound}
		Under $\mathbb{P}_x$, let $(\xi_t)_{t \geq 0}$ be the L\'evy process 
		of Section \ref{sec:one}, and let $\rho(t)$ be the time change 
		defined by $t = \int_0^{\rho(t)} e^{\alpha \xi_s} \mathrm{d}s$. 
		Let $\delta,\varepsilon> 0$, $N\in \mathbb{N}$ be given. 
		Suppose $t_0,\ldots,t_{N-1} \geq t_{\delta,\varepsilon}/\varepsilon$.
		Then, with $T = T(t_0,\ldots,t_{N-1})$ as in \eqref{eq:Tdef}, we have
		\begin{align*}
			\mathbb{P} \left[ \xi_{\rho(T)} \leq N \varepsilon \right]
			\geq \exp \left\{ - (1+\delta) \varepsilon 
			\sum_{ j = 0}^{N-1}L( t_i) t_i^{ \frac{1}{1-\theta}} \right\}.
		\end{align*}
	\end{lemma}

	\begin{proof}
		Let $t_0,\ldots,t_{N-1} \geq t_{\delta,\varepsilon}/\varepsilon$ 
		be a set of large times, where $t_{\delta,\varepsilon}$ is as in the statement of Lemma \ref{lem:pushlower}. 
		Define $q_j := \varepsilon \sum_{i=1}^j t_{N-i}$ for $0 \leq j \leq N$. 
		Now combining Lemma \ref{lem:pushlower} with \eqref{eq:Markov} 
		for any $x \in [0,\varepsilon]$, 
		\begin{align*} 
			\mathbb{P}_{-x} \left[ \xi_{q_j } \in [(j-1)\varepsilon, j \varepsilon ] ~ 
			\text{for all $j = 1,\ldots,N$} \right]
			\geq \exp \left( -  (1 + \delta) \varepsilon 
			\sum_{ j =0}^{ N-1} L(\varepsilon t_i) t_i^{\frac{1}{1-\theta}}  \right). 
		\end{align*}
		Now on the event 
		$\Gamma := \{\xi_{q_j } \in [(j-1)\varepsilon, j \varepsilon ] ~ \text{for all $j = 1,\ldots,N$}\}$, 
		since $(\xi_t)_{t \geq 0}$ is a nondecreasing stochastic process, 
		we plainly have $\xi_t \geq (j-1)\varepsilon$ for all $t \in [q_j,q_{j+1}]$. 
		In particular, if we set $r := \varepsilon \sum_{ i = 0}^{N-1} t_i = q_N$ we have
		\begin{align*}
			I(r) = \int_0^r e^{\alpha \xi_s} \mathrm{d} s  
			= \sum_{ j = 0}^{N-1} \int_{q_j}^{q_{j+1}} e^{ \alpha \xi_s} \mathrm{d}s 
			\geq \varepsilon\sum_{j =0}^{N-1} t_{N-(j+1)} e^{ \alpha (j-1)\varepsilon } 
			= T(t_0,\ldots,t_{N-1}),
		\end{align*}
		where the final equality above follows from reindexing the sum and using the 
		definition \eqref{eq:Tdef} of $T(t_0,\ldots,t_{N-1})$.

		Since $I$ is the inverse function of $\rho$, and both functions are strictly increasing, 
		on the event $\Gamma$ we have $r \geq \rho(T)$. It follows that on the event $\Gamma$, 
		$\xi_{\rho(T)} \leq \xi_{r} \leq N\varepsilon$. In particular, 
		$\mathbb{P} [ \xi_{\rho(T)} \leq N\varepsilon ] \geq \mathbb{P}[ \Gamma]$. 
		To obtain the result as written, note that given $\delta > 0$, for each $t_i$ 
		is sufficiently large compared to $\varepsilon$, we have 
		$L(\varepsilon t_i) \leq (1+\delta) L(t_i)$. Now replace $\delta$ as necessary. 
	\end{proof}

	Given a fixed $T_0$, it is natural to try to optimise the bound occurring 
	in Lemma \ref{lem:carefulbound} to find values $(t_0,\ldots,t_{N-1})$ minimising 
	\[
   	 \sum_{j = 0 }^{N-1} L(t_i) t_i^{\frac{1}{1-\theta}}
    	\] 
	subject to condition $T(t_0,\ldots,t_{N-1}) = T_0$. In light of Lemma \ref{lem:simple} 
	a natural candidate to consider is taking $t_i := Q e^{ - \alpha i\varepsilon(1-\theta)/\theta}$
	for some sufficiently large $Q$. We use this idea to prove the following result.

	\begin{cor} \label{cor:lowprob}
		Let $p > 0$. For each $\delta > 0$ there exists $T_\delta>0$ 
		such that whenever $T \geq T_\delta$ we have
		\begin{align*} 
			\mathbb{P}\left[ \xi_{\rho(T)} \leq p \right] 
			\geq \exp \left\{ -  (1+\delta) L(T) (T e^{ - \alpha p} )^{\frac{1}{1-\theta}} 
			(1 - e^{ - \alpha p /\theta} )^{-\frac{1}{1-\theta}} 
			(\alpha/\theta)^{\frac{\theta}{1-\theta}} \right\}.
		\end{align*}
	\end{cor}

	\begin{proof}
		Fix $\delta>0$ and consider $\varepsilon = \varepsilon(\delta)>0$ and $N = N(p,\varepsilon) \in \mathbb{N}$
		to be determined later.
		In the setting of Lemma \ref{lem:carefulbound}, take 
		$t_i = Q e^{ - \alpha \varepsilon i(1-\theta)/\theta}$. Then
		\begin{align} \label{eq:pintor}
			T_Q := T(t_0,\ldots,t_{N-1}) =Q  e^{ \alpha(N-2)\varepsilon } 
			\frac{\varepsilon}{1-e^{-\alpha \varepsilon/\theta}}( 1 - e^{ - N \alpha \varepsilon/\theta} ) .
		\end{align}
		Now provided $Q e^{ - \alpha \varepsilon N(1-\theta)/\theta} $ 
		is sufficiently large, by Lemma \ref{lem:carefulbound} we have
		\begin{align*}
		\mathbb{P}\left[ \xi_{\rho(T)} \leq N\varepsilon \right] 
			\geq \exp \left\{ - (1+\delta) \varepsilon Q^{ \frac{1}{1-\theta}} \sum_{j = 0}^{N-1} L(Q e^{ - \alpha \varepsilon i(1-\theta)/\theta} ) e^{ - \alpha i \varepsilon /\theta}    \right\} .
		\end{align*}
		By the uniform convergence theorem \cite[Section 1.2]{bingham1989regular}, 
		for all $Q$ sufficiently large,
		\begin{align*}
			L(Q e^{ - \alpha \varepsilon i(1-\theta)/\theta} )  \leq (1 + \delta ) L(Q),
		\end{align*}
		so that
		\begin{align*}
			\mathbb{P}\left[ \xi_{\rho(T)} \leq N\varepsilon \right]
			\geq \exp \left\{ - (1+\delta)^2 L(Q) \varepsilon Q^{ \frac{1}{1-\theta}} 
			\sum_{j = 0}^{N-1} e^{ - \alpha i \varepsilon /\theta}    \right\} .
		\end{align*}
		Since 
		$\sum_{j = 0}^{N-1} e^{ - \alpha i \varepsilon /\theta}  
		\leq  \frac{1}{1-e^{-\alpha\varepsilon/\theta}}$, we obtain
		\begin{align*} 
		\mathbb{P}\left[ \xi_{\rho(T)} \leq N\varepsilon \right]
			\geq \exp \left\{ - (1+\delta)^2 L(Q) \frac{\varepsilon}{1 - e^{ - \alpha \varepsilon/\theta}} 
			Q^{ \frac{1}{1-\theta}}    \right\}.
		\end{align*}
		We now note that $T$ and $Q$ are related to each other through \eqref{eq:pintor}, 
		so that using \eqref{eq:pintor} to express $Q$ in terms of $T$ in the previous inequality, 
		we obtain, for all sufficiently large $T \geq T_0(\delta,\varepsilon,N)$ we have
		\begin{multline*} 
			\mathbb{P}\left[ \xi_{\rho(T)} \leq N\varepsilon \right]
				\\ 
               \geq \exp \left\{ -  (1+\delta)^3 L(T) 
				\left( \frac{\varepsilon}{1-e^{-\alpha \varepsilon/\theta}} \right)^{ - \frac{\theta}{1-\theta}}  (T e^{ - \alpha (N-2)\varepsilon} )^{\frac{1}{1-\theta}} (1 - e^{ - N \alpha \varepsilon/\theta} )^{-\frac{1}{1-\theta}} \right\}.
		\end{multline*}
		Letting $ p = N\varepsilon$, we may write
		\begin{align*} 
		\mathbb{P}\left[ \xi_{\rho(T)} \leq p \right] 
			\geq \exp \left\{ -  (1+\delta)^3 L(T) (T e^{ - \alpha p} )^{\frac{1}{1-\theta}} 
			(1 - e^{ - \alpha p /\theta} )^{-\frac{1}{1-\theta}} 
			(\alpha/\theta)^{\frac{\theta}{1-\theta}} m_{\alpha,\theta}(\varepsilon) \right\},
		\end{align*}
		where
		\begin{align*}
			m_{\alpha,\theta}(\varepsilon) := 
			e^{\frac{2\alpha \varepsilon}{1-\theta}} 
			\left( \frac{\alpha \varepsilon/\theta}{1-e^{-\alpha \varepsilon/\theta}} \right)^{ - \frac{\theta}{1-\theta}}.
		\end{align*}
		which converges to $1$ as $\varepsilon \to 0$. 
		Choose $\varepsilon$ sufficiently small so that $m_{\alpha,\theta}(\varepsilon) \leq 1 + \delta$. 
		Then after replacing $(1+\delta)^4$ with $1+\delta'$ for some suitably chosen $\delta'$, 
		we obtain the result as written.
	\end{proof}

\subsection{Proof of Proposition~\ref{prop:infinitelower}}
	Having the necessary tools at our disposal, in this subsection, we will present the proof of Proposition \ref{prop:infinitelower}. 

    Define the quantity
    \begin{equation}
    d(T,p) := \mathbb{P} [ \xi_{\rho(T)} \leq p ].
    \end{equation}
    We begin with a quick lemma manipulating the antichain property.

    \begin{lemma} \label{lem:previous}
    Let $\Gamma(h-p,\varepsilon,N)$ be the event that the CMJ process underlying our fragmentation process has an antichain $\mathbb{S}$ of particles of heights in $[h-p-\varepsilon,h-p]$ of cardinality satisfying $\# \mathbb{S} \geq N$.     Then the conditional probability, given $\Gamma(h-p,\varepsilon,N)$, of the event $I(t,h)$ that no particle has size $\geq e^{-h}$ at time $t$ satisfies the lower bound
\begin{equation}
\mathbb{P}[ I(t,h) | \Gamma(h-p,\varepsilon,N) ] \leq e^{ - N d(te^{-\alpha(h-p)},p)  }.
\end{equation}
    \end{lemma}

\begin{proof}
Suppose that $\omega$ is a word in the CMJ process with $\mathcal{H}(\omega) \in [h-p-\varepsilon,h-p]$. The probability that this particle has a child of height $\leq h$ at time $t$ is at least the probability that if we apply a spine label to this particle, this subsequent spine particle has height at least $h$ at time $t$. By subtracting a height of $h-p$, and rescaling time by $e^{ - \alpha(h-p)}$, we see that this coincides with the probability that the stochastic process $(\xi_{\rho(s)})_{s \geq 0}$ does not exceed height $p$ by time $te^{ - \alpha(h-p)}$, which has probability $d(Te^{ - \alpha(h-p)}, p)$. 

In short, the probability that a particle of height $\mathcal{H}(\omega) \in [h-p-\varepsilon,h-p]$ existing 
in the process has a child of size at least $e^{-h}$ at time $t$ is at least $d(t e^{ - \alpha(h-p)}, p)$ (\emph{regardless of the moment at which this particle exists}). Note this is even true if the particle is viewed \emph{after} time $t$: if a particle of height in $[h-p-\varepsilon,h-p]$ exists at time $t ' > t$, this immediately implies the existence of a particle of size $\geq e^{ - (h-p)} \geq e^{-h}$ at the (earlier) time $t$. 

Since particles in an antichain give rise to independent future populations, it follows that given that there are at least $N$ particles of heights in $[h-p-\varepsilon,h-p]$ in an antichain, the probability that none of these particles has a descendent of size at least $e^{-h}$ at time $t$ is at most $(1-d(te^{-\alpha(h-p)},p))^N$. Thus,
\begin{equation}
\mathbb{P}[ I(t,h) | \Gamma(h-p,\varepsilon,N) ] \leq (1-d(te^{ - \alpha(h-p)}, p ))^N \leq e^{ - N d(te^{-\alpha(h-p)},p)  },
\end{equation}
 as required.
\end{proof}

	\begin{proof}[Proof of Proposition \ref{prop:infinitelower}]

For $\varepsilon, p = p(\varepsilon)$ to be decided below, let 
    	\begin{align*}
    		A_h &:= \Gamma(h-p,\varepsilon,e^{(1-\varepsilon)h})\\
            &:= \left\{ \exists ~ \text{antichain $\mathbb{S}$} \,:\, 
		\# \mathbb{S} \geq e^{ (1-\varepsilon) h } \text{ and }
	u \in \mathbb{S} \implies \mathcal{H}(u) \in [h-p-\varepsilon,h-p] \right\}
    	\end{align*}
	be the event that there exists an antichain $\mathbb{S}$ in $\mathbb{T}$ 
	of cardinality at least $e^{(1-\varepsilon)h}$ consisting of 
	particles of size between $e^{-(h-p)}$ and $e^{ - (h-p-\varepsilon)}$. 

Now on the one hand, taking $\delta = a = \varepsilon$ in the setting of Theorem~\ref{thm:antichains}, 
	there exists almost surely a random %$\mathcal{F}_{\mathrm{CMJ}}$-measurable 
	height $h_0(\varepsilon)\geq 0$ such that $A_{h}$ occurs for all $h \geq h_0(\varepsilon)$.

On the other hand, by Lemma \ref{lem:previous}, we have
\begin{align} \label{eq:completa}
\mathbb{P}[ I(t,h) | A_h ] \leq e^{ - F(t,h,p,\varepsilon)},
\end{align}
where
\begin{align*}
F(t,h,p,\varepsilon) = e^{(1-\varepsilon)h} d(te^{ - \alpha(h-p)},p).
\end{align*}
Replacing $\delta$ with $\varepsilon$ in the statement of Corollary \ref{cor:lowprob}, provided $te^{ - \alpha(h-p)} \geq t_\varepsilon$ is sufficiently large, we have
\begin{align} \label{eq:completa2}
F(t,h,p,\varepsilon) \geq e^{ G(t,h,p,\varepsilon)}, 
\end{align}
where
\begin{align*}
G(t,h,p,\varepsilon) := (1-\varepsilon)h - (1+\varepsilon)L(te^{ - \alpha (h-p)}) (t e^{ - \alpha h} )^{ \frac{1}{1-\theta} } (1 - e^{ - \alpha p/\theta})^{ - \frac{1}{1-\theta} } (\alpha/\theta)^{ \frac{\theta}{1-\theta}}.  
\end{align*}

We note that if we choose $p = p(\varepsilon)>0$ to be a constant taken sufficiently large so that $(1 - e^{ - \alpha p /\theta} )^{- \frac{1}{1-\theta}} \leq (1+\varepsilon)$. Then 
\begin{align*}
G(t,h,p,\varepsilon) \geq (1-\varepsilon)h - (1+\varepsilon)^2 L(te^{ - \alpha (h-p)}) (t e^{ - \alpha h} )^{ \frac{1}{1-\theta} } (\alpha/\theta)^{ \frac{\theta}{1-\theta}}.  
\end{align*}

Using \eqref{eq:Gdef}, it now follows that if $t = (1-\delta)F_\theta(h)$ 
we have
\begin{equation*}
    G(t,h,p,\varepsilon) \geq h\left(1-\varepsilon-(1+\varepsilon)^2(1-\delta)^{\frac{1}{1-\theta}}\right) + o(h).
\end{equation*}
Therefore if we set $t = (1-\delta)F_\theta(h)$, then provided $\varepsilon$ is sufficiently small compared to $\delta$ we have $G(t,h,p,\varepsilon) \geq c_{\delta,\varepsilon} h $  for $c_{\delta, \varepsilon}>0$ as $h \to \infty$. 

Next, we note that it follows from \eqref{eq:completa2} that $F(t,h,p,\varepsilon) \geq e^{ c_{\delta,\varepsilon}h}$, and subsequently from \eqref{eq:completa} that
\begin{align*}
\mathbb{P}[ I((1-\delta)F_\theta(h),h) | A_h ] \leq e^{ - e^{ c_{\delta,\varepsilon}h}} \leq \exp\{ - c_{\delta,\varepsilon}' h\}
\end{align*}
for all sufficiently large $h$. 

\color{black}
For $n \in \mathbb{N}, 0 \leq i \leq n-1$ write $h_{n,i} := n+i/n$. 
Consider now  
\begin{align*}
\sum_{n \in \mathbb{N}} \sum_{i=0}^{n-i} e^{ - c_{\alpha,\theta}'\varepsilon h_{n,i}} \leq \sum_{n \in \mathbb{N}} n e^{ - c_{\alpha,\theta}'\varepsilon n} < \infty.
\end{align*}
 It follows from the Borel--Cantelli lemma that, conditionally on $A_h$ occurring for some $h$, only finitely many of the events
 \begin{align*}
 \{ I(t,h_{n,i}) : n \in \mathbb{N}, 0 \leq i \leq n-1 \} 
 \end{align*}
 may occur.
 As noted above, there exists almost surely a random height $h_0(\varepsilon) > 0$ such that $A_h$ occurs for all $h \geq h_0(\varepsilon)$.
 It follows that (unconditionally), only finitely many of the events $ \{ I(t,h_{n,i}) : n \in \mathbb{N}, 0 \leq i \leq n-1 \} $ may occur.

At this point we have proved our claim for all $\varepsilon$ along sequences $\{h_{n,i}\}_{n,i}$. To see that this implies our claim,
fix $\varepsilon>0$ and $h>0$. 
Take $n = \lfloor h \rfloor$ and note that for some $i\leq n$, 
\[
	[h_{n,i} -p -\varepsilon/2, h_{n,i}-p] \subseteq [h-p-\varepsilon, h-p].
\]
The claim follows by invoking the first part of the proof with $\varepsilon/2$ in place of $\varepsilon$.

\end{proof}

\color{black}
\section{Proof of Theorem \ref{thm:ia}} \label{sec:proof}

	To complete the proof of our main results, we first make some general 
	observations about certain functions and their inverses. 
	We say that a function $f(h)$ is \textbf{nice} if and only if it takes the form
	\begin{align*}
		f(h) = e^{\alpha h} h^\beta G(h)
	\end{align*}
	for some $\beta > 0$ and a function $G$ slowly varying at infinity.

	A brief calculation tells us that the inverse function of a nice function takes the form
	\begin{align} \label{eq:ad}
		f^{-1}(t) = \frac{1}{\alpha} 
		\big( \log t - \beta \log \log t + \beta \log \alpha - \log G( \log t ) \big) 
		+ O(\log \log t / \log t).
	\end{align}
	We note that for small real $\delta$, if $f_\delta(h) := (1+\delta)f(h)$,
	we have
	\begin{align} \label{eq:ad2}
		f^{-1}_\delta(t) = f^{-1}(t) - \frac{1}{\alpha} \log(1+\delta) + O(\log \log t/\log t). 
	\end{align}
	Using this, it is easy to see that we have the following lemma.

	\begin{lemma} \label{lem:bothways}
		Let $(M_t)_{t \geq 0}$ be an increasing stochastic 
		process satisfying $M_0 = 0$. For $h \geq 0$, define
		\begin{align*}
			\tau_h := \sup \{ t \geq 0 \,:\, M_t \leq h \}.
		\end{align*}
		Suppose there is a nice function $f(h)$ such that 
		\begin{equation*}
			\lim_{h \to \infty}\tau_h/f(h) =1.
		\end{equation*}
		Then, as $t\to\infty$, $M_t - f^{-1}(t)$ converges to zero almost surely.
	\end{lemma}

We note from using the definition \eqref{eq:Fdef} of $F_\theta(h)$ in \eqref{eq:ad}, we have
\begin{align} \label{eq:gnew}
g(t) := F_\theta^{-1}(t) = \frac{1}{\alpha}( \log t - (1-\theta) \log \log t + h(t)),
\end{align}
where, by virtue of \eqref{eq:Gdef}, $h(t) = (1-\theta)\log \alpha - \log G ( \log t)$ is given by 
\begin{align*}
h(t) = 
\begin{cases}
\log \ell (\log t ) + C(\alpha,\theta) \qquad &:\text{$\theta \in (0,1)$, or $\theta = 0$ finite activity}\\
(1-\theta) \log \alpha - \log G( \log t) \qquad &:\text{$\theta = 0$ with infinite activity},
\end{cases}
\end{align*}
where in the first case above, $C(\alpha,\theta)$ is given by $C(\alpha,\theta) := (1-\theta)\log \alpha - C_3$, with $C_3$ as in \eqref{eq:Gdef}, so that 
\begin{align*}
C(\alpha,\theta) = \log \alpha + (1-\theta)\log(1-\theta) + \log \Gamma(1-\theta),
    \end{align*}
 and in the second, $G(\cdot)$ is as in \eqref{eq:complicated}. This agrees with the statement of Theorem \ref{thm:ia}.

We now prove our main result, Theorem \ref{thm:ia}.
% \begin{proof}[Proof of Theorem \ref{thm:fa}]
% 	Let $\tau_h := \sup \{ t \geq 0: m_t \leq h \}$
% 	and $F_0(h) := \frac{1}{\lambda}he^{\alpha h}$. Combining the statements of 
% 	Proposition \ref{prop:finiteupper} and Proposition \ref{prop:finitelower}, 
% 	for all $\delta > 0$ there exist $c_\delta > 0$ and $h_\delta > 0$ such that whenever $h \geq h_\delta$ we have
% 	\begin{align*}
% 		\mathbb{P} \left( \tau_h \notin ((1-\delta)F_0(h),(1+\delta)F_0(h) \right) \leq e^{ - c_\delta h}.
% 	\end{align*}
% 	In particular, we are in the setting of Lemma \ref{lem:bothways} 
% 	for the process $(m_t,t\geq 0)$ with $F_0(h) = \frac{1}{\lambda}he^{\alpha h}$. 
% 	By \eqref{eq:ad} (setting $\beta = 1$ and $G(h) = 1/\lambda$) the associated inverse function takes the form
% 	\begin{align*}
% 		F_0^{-1}(t) = \frac{1}{\alpha} \left( \log t - \log \log t 
% 		+ \log \alpha - \log \frac{1}{\lambda} \right) + O(\log \log t/\log t).
% 	\end{align*}
% 	The proof now follows from Lemma \ref{lem:bothways}.
% \end{proof}

\begin{proof}[Proof of Theorem \ref{thm:ia}]
\color{black}
	Let $\tau_h := \sup \{ t \geq 0: m_t \leq h \}$ and recall that  $F_\theta(h)$ is defined in \eqref{eq:Fdef}. 
	Proposition~\ref{prop:infinitelower} gives the bound
	\begin{equation}\label{eq:banana1}
		\liminf_{h \to \infty} \frac{\tau_h}{F_\theta(h)} \geq 1.
	\end{equation}
	On the other hand, 
	Proposition \ref{prop:infiniteupper} yields that for all $\delta > 0$ 
	there exists $c_\delta > 0$ and $h_\delta > 0$ such that whenever $h \geq h_\delta$ we have
	\begin{align*}
		\mathbb{P} \left[ \tau_h  \geq (1+\delta)F_\theta(h) \right] \leq 2e^{ - c_\delta h}.
	\end{align*}
	Define
		\begin{align*}
			h_i := i\delta \qquad \text{and} \qquad 
			t_i := F_\theta(h_i).
		\end{align*}
		By the Borel--Cantelli lemma, with probability one, 
		$\tau_{h_i} \leq (1+\delta)t_i$ for all sufficiently large $i\in\mathbb N$. 
		Since $h \mapsto \tau_h$ is a nondecreasing stochastic process, 
		\begin{align} \label{eq:apple1}
			h \in (h_k,h_{k+1}] 
			\implies 
			\tau_h \leq \tau_{h_{k+1}} \leq (1+\delta)F_\theta(h_{k+1})
		\end{align}
		for all sufficiently large $k$. 
		Note that for a sufficiently large $c > 0$, we have 
		\begin{align} \label{eq:apple2}
	(1+\delta)F_\theta(h_{k+1}) \leq (1+c\delta) F_\theta(h), 
			\qquad  \qquad  ~\forall h \in (h_k,h_{k+1}].
		\end{align}
		In particular, combining \eqref{eq:apple1} and \eqref{eq:apple2}, 
		it follows that $\tau_h \leq (1 + c \delta) F_\theta(h)$ for all 
		sufficiently large $h$. Since $\delta > 0$ is arbitrary, it follows that
		\begin{equation}\label{eq:banana2}
			\limsup_{h \to \infty} \frac{\tau_h}{F_\theta(h)} \leq 1.
		\end{equation}

		By~\eqref{eq:banana1} and \eqref{eq:banana2}, we are in the setting of Lemma \ref{lem:bothways}
	with $f=F_\theta$, so that $\lim_{t \to \infty} (m_t - g(t)) = 0 $ almost surely, where $g$ is given in \eqref{eq:gnew}. 
That completes the proof.

\end{proof}

\begin{appendix}
\section{Deviation estimates for L\'evy processes}

	Here we present all of the details for the proofs of Theorem~\ref{thm:levybound} and its various constituent lemmas.

\subsection{Completion of proof of Lemma \ref{lem:tailsnew}}
We begin by completing the proof of Lemma \ref{lem:tailsnew}.

	\begin{proof}[Proof of lower bound \eqref{eq:tailsnew2} in Lemma~\ref{lem:tailsnew}]
We will continue to use the notation in the proof of the upper bound in Lemma \ref{lem:tailsnew} given in the main body.\label{sec:tailsnewproof}

        By way of preparation for a use of Berry--Esseen bound below, we note here that \color{black} it is a brief calculation to show that the expectation and variance of $\xi_t$ under $\mathbb{P}^{(q)}$ are given by 
	\begin{align} \label{eq:center}
		\mathbb{E}^{(q)}[\xi_t] = t \Phi'(q) \qquad \text{and} \qquad \mathbb{E}^{(q)}\left[ (\xi_t - t \Phi'(q))^2 \right] = t \Phi''(q).
	\end{align}

We have 
\begin{align} \label{eq:gto}
\mathbb{P}[ xt \leq \xi_t \leq xt + \sqrt{t \Phi''(q_x)} ] &= e^{ - t \Phi(q_x)} \mathbb{E}^{(q_x)}\left[ e^{ q_x \xi_t } \mathrm{1}\{ xt \leq \xi_t \leq xt + \sqrt{t \Phi''(q_x)}  \} \right] \nonumber \\
&\geq e^{ - t \Phi(q_x) + t q_x x } \mathbb{P}^{(q_x)}[ xt \leq \xi_t \leq xt + \sqrt{t \Phi''(q_x)} ].
\end{align}
\color{black}
	% For arbitrary $\varepsilon > 0$ we have
	% \begin{align*} 
	% 	\mathbb{P}[ \xi_t \leq tx] &\geq  e^{ - t \Phi(q_x) } 
	% 	\mathbb{E}^{(q_x)} \left[ e^{q_x \xi_t } \mathrm{1}_{\{t(x-\varepsilon) \leq \xi_t \leq tx\}} \right]\nonumber \\
	% 		&\geq e^{ - t \Phi(q_x) } e^{ t q_x(x-\varepsilon) } 
	% 		\mathbb{P}^{(q_x)} \left[ \frac{-\varepsilon \sqrt{t} }{\sqrt{\Phi''(q_x)}}\leq \frac{\xi_t - tx }{ \sqrt{\Phi''(q_x)t } } 
	% 		\leq 0 \right].
	% \end{align*}
	By \eqref{eq:center}, the random variable $Z_t := (\xi_t - tx)/ \sqrt{\Phi''(q_x)t }$ has zero expectation and unit variance. 
	In fact, under $\mathbb{P}^{(q_x)}$, $Z_t$ converges in distribution to a standard Gaussian random variable as $t \to \infty$. 
	Better yet, using the Berry--Esseen Theorem to obtain the first inequality below, 
	and then Jensen's inequality to obtain the second, for any $\alpha < \beta$ in $\mathbb{R}$, we have 
	\begin{align*}
		\left| \mathbb{P}^{(q_x)} \left[\frac{\xi_t - tx }{ \sqrt{\Phi''(q_x)t } } \in [\alpha,\beta] \right] 
			- \frac{1}{\sqrt{2\pi}} \int_\alpha^\beta e^{-u^2/2} \mathrm{d} u \right| 
		&\leq \frac{C'}{\sqrt{t}} 
		\frac{ \mathbb{E}^{(q_x)}\left[ |\xi_1 - x|^3 \right] }{\mathbb{E}^{(q_x)}\left[ (\xi_1 - x)^2 \right]^{3/2} }\\
		&\leq \frac{C'}{\sqrt{t}} 
		\frac{ \mathbb{E}^{(q_x)}\left[ (\xi_1 - x)^4 \right]^{3/4} }{\mathbb{E}^{(q_x)}\left[ (\xi_1 - x)^2 \right]^{3/2} },
	\end{align*}
	where $C'>0$ is a universal constant. We use Jensen's inequality again, to observe that 
	$$
		\mathbb{E}^{(q_x)}\left[ (\xi_1 - x)^4 \right] \geq \mathbb{E}^{(q_x)}\left[ (\xi_1 - x)^2\right].
	$$
	Since $y^{3/4}\leq y$ for $y\geq 1$, we have  
	\begin{align*}
		\left| \mathbb{P}^{(q_x)} \left[\frac{\xi_t - tx }{ \sqrt{\Phi''(q_x)t } } \in [\alpha,\beta] \right] 
		- \frac{1}{\sqrt{2\pi}} \int_\alpha^\beta e^{-u^2/2} \mathrm{d} u \right| 
		&\leq \frac{C'}{\sqrt{t}} \frac{ \mathbb{E}^{(q_x)}[ (\xi_1 - x)^4 ] }{\mathbb{E}^{(q_x)}[ (\xi_1 - x)^2 ]^{2} }.
	\end{align*}
    Using \eqref{eq:gto} in conjunction with the previous display we have 
	\begin{align} \label{eq:chernfirst3}
		\mathbb{P}[xt\leq \xi_t \leq tx+ \sqrt{t\Phi''(q_x)}] &\geq e^{ - tR(q_x) - \sqrt{t}  q_x \sqrt{\Phi''(q_x) }}  
		\left\{ C'' - \frac{C'}{\sqrt{t}} \frac{ \mathbb{E}^{(q_x)}[ (\xi_1 - x)^4 ] }{\mathbb{E}^{(q_x)}[ (\xi_1 - x)^2 ]^2   } \right\},
	\end{align}
	where $C'' := \int_{0}^1 e^{ - u^2/2}\mathrm{d}u/\sqrt{2\pi}$. 
	Note that the constant $C$ in \eqref{eq:tailscondition} can be taken in a way that Equation \eqref{eq:tailscondition} implies
	\begin{align*} 
		\frac{C'}{\sqrt{t}} \frac{ \mathbb{E}^{(q_x)}[ (\xi_1 - x)^4 ] }{\mathbb{E}^{(q_x)}[ (\xi_1 - x)^2 ]^2   } \leq \frac{C''}{2}. 
	\end{align*}
	Plugging the previous inequality into \eqref{eq:chernfirst3} we have for all $x$ such that $t\geq r(x)$
	\begin{align*} 
		\mathbb{P}[xt\leq \xi_t \leq tx+ \sqrt{t\Phi''(q_x)}] &\geq e^{ - tR(q_x) - \sqrt{t} q_x \sqrt{\Phi''(q_x)} - c},
	\end{align*}
	where $c := -\log(C''/2)$, completing the proof. 
\end{proof}

\subsection{Complete proof for Lemma \ref{lem:Phiasnew}}
\label{sec:Phiasnewproof}
\begin{proof}[Proof of Lemma \ref{lem:Phiasnew}]
	We begin by proving \eqref{eq:Phiasnew}. \color{black} First we prove an upper bound for $\Phi(q)$. By \eqref{eq:Phidef} we have 
		\begin{equation*}
			\Phi(q) \leq \int_{\MassP} \left( 1 - s_1^{q+1} \right)\nu (\rm d \mathbf{s}).
		\end{equation*}
		To analyze the integral on the right-hand side just integrate by parts to get
		\begin{equation*}
			\int_{\MassP}\left(1-s_1^{q+1}\right) \nu(\mathrm{d} \mathbf{s}) = 
			\int_0^1 (q+1) t^{q} \nu(s_1 < t) \mathrm{d} t =
			\int_0^1 (q+1) (1-t)^q t^{-\theta}\ell(1/t) \mathrm{d} t,
		\end{equation*}
        where we used the behaviour \eqref{eq:1:crumbling}. 
		We can write the last integral as
		\begin{equation*}
			q^{\theta} \frac{q+1}{q} \ell(q)
			\int_0^q \left(1-\frac sq \right)^q s^{-\theta} \frac{\ell(q/s)}{\ell(q)} \mathrm{d} s.
		\end{equation*}
		By an appeal to Potter bounds~\cite[Theorem 1.5.6 (i)]{bingham1989regular} 
		for sufficiently large $q$,
		\begin{equation*}
			\ell(q/s) / \ell(q) \leq 2 s^{ -(1-\theta)/2}
		\end{equation*}
		for all $s \in (0,1)$. Using the simple estimate $1-x\leq e^{-x}$ for $x \in (0,1)$, we can appeal to 
		the Dominated Convergence Theorem and slow variation of $\ell$ 
		to show that the integral in question tends to $\Gamma(1-\theta)$ as $q \to \infty$.
		This concludes a proof of the upper bound for $\Phi(q)$.

		As for the lower bound for $\Phi(q)$, we observe that
		\begin{equation*}
			\int_{\MassP} \left(1 - \sum_{i \geq 1} s_i^{q+1} \right) \nu(\mathrm{d} \mathbf{s}) \geq  
			\int_{\MassP} \left(1 - s_1^{q+1} - s_1^{q/2}(1-s_1)^{q/2+1}\right) \nu(\mathrm{d} \mathbf{s}),
		\end{equation*}
		using the fact that $s_2^{q+1} + s_3^{q+1} + \ldots \leq s_1^{q/2}(1 - s_1)^{q/2+1}$ whenever 
		$\sum_{ i \geq 1} s_i \leq 1$.
		We have already checked that the integral of $1-s_1^{q+1}$ has the desired asymptotic. 
		To conclude the proof of \eqref{eq:Phiasnew} it suffices to check that 
		the integral 
		\begin{equation*}
			\int_{\MassP} s_1^{q/2}(1-s_1)^{q/2+1} \nu(\mathrm{ d} \mathbf{s}) =
			\frac q2\int_0^1 (1-t)^{q/2-1}t^{q/2-\theta}(2t-1)\ell(1/t) \mathrm{d} t 
		\end{equation*}
		converges to zero as $q \to \infty$. This is indeed true, since by yet another appeal to 
		Potter bounds~\cite[Theorem 1.5.6 (ii)]{bingham1989regular},  
		$\ell(t) \leq \const /t$ for some constant $\const>0$ and all $t \in (0,1)$. 
		Thus, the last integral is bounded above via
		$Cq \mathrm{B}(q/2+\theta +1, q/2)$, where $\mathrm{B}(\cdot, \cdot)$ is the beta function. Using Stirling's approximation one 
        gets the bound
        \begin{equation}\label{eq:fancy}
        \int_{\MassP} s_1^{q/2}(1-s_1)^{q/2+1} \nu(\mathrm{ d} \mathbf{s}) \leq C_1 q^{-1/2} 2^{-q+\theta}
        \end{equation}
        for some $C_1 \geq C$.
        Now \eqref{eq:Phiasnew} follows.\medskip

We turn to the proof of \eqref{eq:Phihighernew}. 
First we observe that by differentiating \eqref{eq:Phidef} with respect to $q$, we see that for all $j \geq 1$, 
\begin{align} \label{eq:der}
\Phi^{(j)}(q) = - \int_{\MassP} \sum_{i \geq 1} s_i^{q+1} \log(s_i)^{j} \nu(\mathrm{d}\mathbf{s}).
\end{align}
Accordingly, $(-1)^{j-1}\Phi^{(j)}(q)$ is a nonnegative and monotone decreasing function of $q$. 

Karamata's monotone density theorem \cite[Theorem 1.7.2]{bingham1989regular} states that if $U$ is a monotone increasing or decreasing regularly varying function with index $\rho \in \mathbb{R} - \{0\}$ so that $U(x) = x^\rho L(x)$ for some slowly varying function $L:(0,\infty) \to (0,\infty)$, then its derivative takes the form $U'(x) \sim \rho x^{\rho-1}L(x)$. 
In the case where $\theta \neq 0$, Equation \eqref{eq:Phihighernew} now immediately follows from \eqref{eq:Phiasnew} and inductively applying Karamata's monotone density theorem \cite[Theorem 1.7.2]{bingham1989regular}.

As for the proof of \eqref{eq:Phihighernew} in the case $\theta = 0$, we are not at liberty to use Karamata's monotone density theorem, and must examine $\Phi'(q)$ directly. 
%In this direction, setting $j=1$ in \eqref{eq:der} we have
%		\begin{equation*}
%			\Phi'(q) = \int_{\MassP} \sum_{i \geq 1} s_i^{q+1}\log(1/s_i)  %\nu(\mathrm{d} \mathbf{s}).
%		\end{equation*}
		We will use a simple estimate:
		for all $x \in (0,1)$,
		$\log(1/x) \leq 1/(ex)$. 
		Thus, since $\sum_{i\geq 2}s_i = 1-s_1$ and $\max_{i \geq 2}s_i \leq 1/2$ we can write
		\begin{equation*}
			\sum_{i \geq 2} s_i^{q+1}\log(1/s_i)^j
			\leq \const \sum_{i \geq 2} s_i^{q-j}
			\leq \const 2^{j} s_1^{q/2}(1-s_1)^{q/2}.
		\end{equation*}
		Therefore,
		\begin{equation*}
			\left| (-1)^{j-1}\Phi^{(j)}(q) - \int_{\MassP} s_1^{q+1}\log(1/s_1)^j  \nu(\mathrm{d} \mathbf{s}) \right|
			\leq \const 2^{j}\int_{\MassP}
			s_1^{q/2}(1-s_1)^{q/2} \nu(\mathrm{d} \mathbf{s}).
		\end{equation*}
		% We have seen in the proof of Lemma~\ref{lem:Phias}, that
        By an appeal to~\eqref{eq:fancy} we thus conclude that
		\begin{equation*}
			(-1)^{j-1}\Phi^{(j)}(q) - \int_{\MassP} s_1^{q+1}\log(1/s_1)^j  \nu(\mathrm{d} \mathbf{s}) 
			= O \left(q^{-1/2}2^{-q/2} \right).
		\end{equation*}
		Take $g(x) = x^q \log(1/x)^{j-1} ((q+1) \log(1/x)-j)$, the derivative of the function $x \mapsto x^{q+1}\log(1/x)^j$ and write
		\begin{equation*}
			\int_{\MassP} s_1^{q+1} \log(1/s_1) \nu(\mathrm{d} \mathbf{s}) =
			\int_0^1 g(t) \nu(s_1 \leq t) \mathrm{d} t =
			\int_0^1 g(1-t) \ell(1/t) \mathrm{d} t.
		\end{equation*}
		As previously, write the last integral as
		\begin{equation*}
			\frac{1}{q}
			\int_0^q \left(1-\frac sq \right)^q \log\left(\left(1-\frac sq \right)^{-1}\right)^{j-1}
			\left( (q+1)\log\left(\left(1-\frac sq \right)^{-1}\right) -j \right) \ell(q/s) \mathrm{d} s.
		\end{equation*}
		Using the expansions 
		\[
			\left(1-\frac{s}{q}\right)^q = e^{-s}+\frac{e^{-s}s^2}{2q(1+\epsilon(s,q))}	
        \]
		with $\epsilon(s,q)\leq \const/ q^2$, and
		\begin{equation*}
			(q+1)\log\left(\left(1-\frac sq \right)^{-1}\right) = s + \epsilon_0(s,q)
		\end{equation*}
		with $\epsilon_0(s,q) \leq s/q + (q+1) s^2/(q^2-sq)$, 
		together with an appeal to Potter bounds 
		we infer that
		\begin{equation*}
			\int_{\MassP} s_1^{q+1} \log(1/s_1)^j \nu(\mathrm{d} \mathbf{s}) 
			= q^{-j}\int_0^\infty e^{-s}s^{j-1}(s-j) \ell(q/s)\mathrm{d}s + o\left( q^{-j-1/2} \right).
		\end{equation*}
		As a result the right-hand side is an approximation of $(-1)^{j-1}\Phi^j(q)$ up to order $o\left( q^{-j-1/2} \right)$.
        In the finite activity case, $\ell(q) \to \lambda$ as $q \to \infty$ and
        it follows that the function 
        \begin{equation*}
            o_j(q)=\int_0^\infty e^{-s}s^{j-1}(s-j) \ell(q/s)\mathrm{d}s
        \end{equation*}
        is slowly varying an vanishing at infinity.
		In the infinite activity case, 
        using the regularity condition imposed on $\ell$, we can write
		\begin{align*}
			\int_0^\infty e^{-s}s^{j-1}(s-j) \ell(q/s)\mathrm{d}s & \sim
			\int_0^\infty e^{-s}s^{j-1}(s-j) \int_0^{q/s}\frac{1}{y} \ell_0(y) \mathrm{d}y \: \mathrm{d}s \\ 
			& = \int_0^\infty e^{-s}s^{j-1} \ell_0(q/s) \mathrm{d}s \sim \Gamma(j) \ell_0(q).
		\end{align*}
		It follows that in the case $\theta = 0$ we have $(-1)^{j-1}\Phi^{(j)}(q) \sim q^{-j}\Gamma(j)\ell_0(q)$. 
        %In other words, the case $j=1$ of \eqref{eq:Phihighernew} holds when $\theta = 0$. To obtain the general $j \geq 1$ case of this formula, we may again use Karamata's monotone density theorem \cite[Theorem 1.7.2]{bingham1989regular}, inductively. That completes the proof.

\end{proof}

\color{black}
\subsection{Proof of Lemma \ref{lem:condition}}

\begin{proof}[Proof of Lemma \ref{lem:condition}]
It is a brief calculation differentiating the Laplace exponent of the L\'evy process to verify that
\begin{align*}
\frac{ \mathbb{E}^{(q_x)}[ (\xi_1 - x)^4 ] }{\mathbb{E}^{(q_x)}[ (\xi_1 - x)^2 ]^2   } = 3 + \frac{\Phi^{(4)}(q_x)}{ \Phi^{(2)}(q_x)^2}.
\end{align*} 
By virtue of \eqref{eq:Phihighernew}, as $q \to \infty$, we have 
\begin{align*}
\frac{\Phi^{(4)}(q)}{ \Phi^{(2)}(q)^2} \sim \tilde{C}(\theta) q^{ - \theta} L_1(q)^{-1},
\end{align*}
where $\Phi^{(2)}$ and $\Phi^{(4)}$ are the second and fourth derivative of $\Phi$ respectively
and $\tilde{C}(\theta) = C_4(\theta)/C_2(\theta)^2 > 0$. 

Since $q_x \to \infty$, as $x \downarrow 0$, we have  
\begin{align} \label{eq:ado}
r(x) \sim C \left(3 + \tilde{C}(\theta)q_x^{-\theta} L_1(q_x)^{-1} \right)^2. 
%\qquad \text{as $x \downarrow 0$}.
\end{align}

Now whenever $\theta > 0$, we see from \eqref{eq:ado} that $r(x)$ is bounded for all $x \leq x_0$, so that the result clearly holds with $\hat{L}$ taken to equal some sufficiently large constant.

When $\theta = 0$, however, using \eqref{eq:ado} and the definition \eqref{eq:Ldef}, we have
\begin{align*}
r(x) \sim C \left(3 + C'(0) \ell_0(q_x)^{-1} \right)^2,
\end{align*}
where $\hat{L}$ is a slowly varying function that may tend to infinity as $x \to 0$ (this happens in the case where $\ell_0(q) \to 0$ as $q \to \infty$). 

Either way, there exists some $x_0>0$ such that for all $0 < x \leq x_0$, $r(x) = \hat{L}(1/x)$ for a slowly varying function of $x$.
\end{proof}

\color{Mahogany}
\subsection{Proof of Lemma \ref{lem:tech}}
\color{black}

	\begin{proof}[Proof of Lemma \ref{lem:tech}]
		Let 
		\begin{equation*}
			H_n :=  \inf\left\{  \varepsilon  
			\sum_{i=0}^{\infty} z_i^{\frac{1}{1-\theta}}f_n(z_i) \,:\, 
			(z_0,z_1,\ldots)\in[0,\infty)^{\mathbb{N}}, \: 
			\varepsilon \sum_{i=0}^{\infty} z_ie^{ - \alpha i \varepsilon} \geq 1 \right\}.
		\end{equation*}
		First we claim that there exists $A > 1$ and $n_0\in\mathbb N$ such that for all integers $ n \geq n_0$, 
		we may restrict our attention
		to the case where the infimum is attained by 
		sequences with coordinates $z_i$ taking values in a compact interval 
		$[0,A]$. To see this, note that since, for $x \geq 1$, 
		$f_n(x) \geq C x^{- \delta_n}$ for some sequence 
		$(\delta_n)_{n \geq 1}$ such that $\delta_n \to 0$, it follows that 
		for all sufficiently large $n$, 
		$z^{\frac{1}{1-\theta}}f_n(z) \geq z^{ \frac{1}{2} \frac{1}{1-\theta} }$. 
		In particular, if some $z_i > A$, we have 
		$\varepsilon \sum_{i=0}^\infty z_i^{\frac{1}{1-\theta}} f_n(z_i) > \
		\varepsilon A^{\frac{1}{2}\frac{1}{1-\theta}}$. Thus,
		\begin{align*}
			H_n = I_n(A) \wedge \left( \varepsilon A^{\frac{1}{2}\frac{1}{1-\theta}} \right),
		\end{align*}
		where 
		\begin{align} \label{eq:pickap}
			 I_n(A) := 
			\inf\left\{  \varepsilon  
			\sum_{i=0}^{\infty} z_i^{\frac{1}{1-\theta}}f_n(z_i) 
			\,:\,(z_0,z_1,\ldots)\in[0,A]^{\mathbb{N}}, \: 
			\varepsilon \sum_{i=0}^{\infty} z_ie^{ - \alpha i \varepsilon} 
			\geq 1\right\} .
		\end{align}
		Choose $A>0$ sufficiently large so that 
		$\varepsilon A^{\frac{1}{2}\frac{1}{1-\theta}} > C(\theta,\alpha,\varepsilon)$. 
		It remains to examine $I_n:=I_n(A)$ for such $A$. 
		Now let $a \in (0,1)$ be small, and consider the contribution to the sum 
		$\varepsilon \sum_{i=0}^\infty z_i e^{ - \alpha i \varepsilon }$ 
		from terms with $z_i < a$ and from terms with $z_i \geq a$. 
		Plainly we have 
		\begin{equation*}
			\varepsilon \sum_{i=0}^\infty 
			z_i \mathrm{1}_{\{z_i < a\}}e^{ - \alpha i \varepsilon } 
			\leq \varepsilon a \sum_{i=0}^\infty e^{ - \alpha i \varepsilon } 
			= \tilde{C}a,
		\end{equation*}
		where $\tilde{C}:=\varepsilon  (1- e^{ - \alpha \varepsilon })^{-1}$. 
		In particular, it follows that for small $a$,
		\begin{align*}
			\varepsilon \sum_{i \geq 0} \mathrm{1}_{\{z_i \geq a\}} 
			z_i e^{-\alpha i \varepsilon} 
			\geq 1 - \tilde{C}a.
		\end{align*}
		Thus, picking up from \eqref{eq:pickap}, we have
		\begin{multline} \label{eq:pickap2}
			I_n \geq  
			\inf\left\{  \varepsilon  \sum_{i=0}^{\infty} z_i^{\frac{1}{1-\theta}}f_n(z_i) 
			: (z_0,z_1,\ldots)\in[0,A]^{\mathbb{N}}, \:
			\varepsilon 
			\sum_{i \geq 0} \mathrm{1}_{\{z_i \geq a\}} \frac{z_i}{e^{  \alpha i \varepsilon}}
			\geq 1 - \tilde{C}a  \right\}  \\
			\geq  \inf\left\{  \varepsilon  
			\sum_{i=0}^{\infty}\mathrm{1}_{\{z_i \geq a\}}z_i^{\frac{1}{1-\theta}}f_n(z_i) 
			:(z_0,z_1,\ldots)\in[0,A]^{\mathbb{N}}, \:
			\varepsilon 
			\sum_{i \geq 0} \mathrm{1}_{\{z_i \geq a\}} \frac{z_i}{e^{ \alpha i \varepsilon} }
			\geq 1 - \tilde{C}a  \right\}.
		\end{multline}
		Moreover, since the sequence of functions $(f_n)_{n \geq 1}$ converges 
		uniformly to $1$ on the compact set $[a,A]$, for all $a\in (0,1)$ there exists 
		$n_0(a)\in\mathbb N$ such that whenever $n \geq n_0(a)$, 
		$f_n(x) \geq 1 - a$ for all $x \in [a,A]$. 
		Using this fact in \eqref{eq:pickap2} to obtain the first inequality below, 
		and then rescaling the $z_i$  with $z_i/(1 - \tilde{C}a)$, 
		to obtain the following equality, for all $n \geq n_0(a)$ we have 
		\begin{equation} \label{eq:pickap3}
        \begin{split}
		&\frac{I_n}{1-a}\geq 
			\inf\left\{  \varepsilon  
				\sum_{i=0}^{\infty}\mathrm{1}_{\{z_i \geq a\}}
			z_i^{\frac{1}{1-\theta}}
			\,:\, (z_0,z_1,\ldots)\in[0,A]^{\mathbb{N}}, \: 
			\varepsilon \sum_{i \geq 0} 
			\mathrm{1}_{\{z_i \geq a\}} \frac{z_i}{e^{ \alpha i \varepsilon} }
			\geq 1 - \tilde{C}a  \right\} \\
			&\geq (1 - \tilde{C}a)^{\frac{1}{1-\theta}} 
			\inf\left\{  \varepsilon  
				\sum_{i=0}^{\infty}\mathrm{1}_{\left\{z_i \geq a\right\}}z_i^{\frac{1}{1-\theta}}
			\,:\, (z_0,z_1,\ldots)\in[0,A]^{\mathbb{N}}, \:
			   \varepsilon 
			   \sum_{i \geq 0} \mathrm{1}_{\left\{z_i \geq a \right\}} 
		   \frac{z_i}{e^{  \alpha i \varepsilon}} \geq 1 \right\}.
		\end{split}
        \end{equation}
		Finally, we note that
		\begin{align*} 
		\inf&\left\{  \varepsilon  \sum_{i=0}^{\infty}\mathrm{1}_{\{z_i \geq a\}} z_i^{\frac{1}{1-\theta}}\,:\, 
		(z_0,z_1,\ldots)\in[0,A]^{\mathbb{N}}  \mbox{ with }   
	\varepsilon \sum_{i \geq 0} \mathrm{1}_{\{z_i \geq a\}} \frac{z_i}{e^{ \alpha i \varepsilon}} \geq 1 \right\}  \\
		=&  \inf\left\{  \varepsilon  \sum_{i=0}^{\infty}z_i^{\frac{1}{1-\theta}}\,:\, 
		(z_0,z_1,\ldots)\in\left(\{0\}\cup[a,A]\right)^{\mathbb{N}}  \mbox{ with }   \varepsilon 
		\sum_{i \geq 0} \frac{z_i}{e^{ \alpha i \varepsilon}} \geq 1   \right\}  \\
		\geq& \inf\left\{  \varepsilon  \sum_{i=0}^{\infty}z_i^{\frac{1}{1-\theta}}\,:\,  
(z_0,z_1,\ldots)\in[0,\infty)^{\mathbb{N}}  \mbox{ with }    
	\varepsilon \sum_{i \geq 0} \frac{z_i}{e^{  \alpha i \varepsilon}} \geq 1  \right\}  = C(\alpha,\theta,\varepsilon).
		\end{align*}
		Thus, by \eqref{eq:pickap3} and the previous inequality, 
		given $a \in (0,1)$, for all $n \geq n_0(a)$, we have
		\begin{align*}
			I_n \geq (1-a)(1 - \tilde{C}a)^{\frac{1}{1-\theta}} C(\theta,\alpha,\varepsilon).
		\end{align*}
		Since $a\in (0,1)$ was arbitrary, that completes the proof.
	\end{proof}
    
\end{appendix}

\section*{Acknowledgments}
The authors are extremely grateful to Jean Bertoin for several comments, in particular for pointing out an error in an earlier version regarding our use of the limit theory of Crump--Mode--Jagers branching processes. We also thank Victor Rivero for pointing out the link between tails of L\'evy processes and their exponential functionals.

The authors also express their gratitude to an anonymous referee for highlighting an error in our proof of Proposition \ref{prop:heightlower}, as well as for pointing out an oversight of a technical condition in the statement of \eqref{eq:bertoin}.

Part of this work was carried out while Joscha Prochno was on a research semester at Graz University of Technology. He is grateful for the excellent working conditions.

%%%%%%%%%%%%%%%%%%%%%%%%%%%%%%%%%%%%%%%%%%%%%%
% Funding information, if any,             %%
%% should be provided in the                %%
%% funding section.                         %%
%%%%%%%%%%%%%%%%%%%%%%%%%%%%%%%%%%%%%%%%%%%%%%
\section*{Funding}
Piotr Dyszewski was partially supported by the National Science Centre, Poland (Sonata, grant
number 2020/39/D/ST1/00258). 

Sandra Palau acknowledges support from the Royal Society, UK as a Newton International Fellow Alumnus (grant AL231019). 

Joscha Prochno is supported by the German Research Foundation (DFG) under project 516672205.

\bibliography{selfsimilar} 

\begin{thebibliography}{10}

\bibitem{Aldous1991}
D.~Aldous.
\newblock The continuum random tree. i.
\newblock {\em The Annals of Probability}, 19(1):1--28, 1991.

\bibitem{aldous1998standard}
D.~Aldous and J.~Pitman.
\newblock The standard additive coalescent.
\newblock {\em Annals of Probability}, pages 1703--1726, 1998.

\bibitem{babler2008modelling}
M.~U. B{\"a}bler, M.~Morbidelli, and J.~Ba{\l}dyga.
\newblock Modelling the breakup of solid aggregates in turbulent flows.
\newblock {\em Journal of Fluid Mechanics}, 612:261--289, 2008.

\bibitem{BEE}
A.~M. Basedow, K.~H. Ebert, and H.~J. Ederer.
\newblock Kinetic studies on the acid hydrolysis of dextran.
\newblock {\em Macromolecules}, 11(4):774--781, 1978.

\bibitem{Berestycki}
J.~Berestycki.
\newblock Ranked fragmentations.
\newblock {\em ESAIM, Probab. Stat.}, 6:157--175, 2002.

\bibitem{bertoin1}
J.~Bertoin.
\newblock Homogeneous fragmentation processes.
\newblock {\em Probab. Theory Relat. Fields}, 121(3):301--318, 2001.

\bibitem{bertoin2}
J.~Bertoin.
\newblock Self-similar fragmentations.
\newblock {\em Ann. Inst. Henri Poincar{\'e}, Probab. Stat.}, 38(3):319--340,
  2002.

\bibitem{bertoin3}
J.~Bertoin.
\newblock The asymptotic behavior of fragmentation processes.
\newblock {\em J. Eur. Math. Soc. (JEMS)}, 5(4):395--416, 2003.

\bibitem{bertoinbook}
J.~Bertoin.
\newblock {\em Random fragmentation and coagulation processes}, volume 102 of
  {\em Camb. Stud. Adv. Math.}
\newblock Cambridge: Cambridge University Press, 2006.

\bibitem{bertoin2018martingales}
J.~Bertoin, T.~Budd, N.~Curien, and I.~Kortchemski.
\newblock Martingales in self-similar growth-fragmentations and their
  connections with random planar maps.
\newblock {\em Probability Theory and Related Fields}, 172(3):663--724, 2018.

\bibitem{bertoin2018random}
J.~Bertoin, N.~Curien, and I.~Kortchemski.
\newblock Random planar maps and growth-fragmentations.
\newblock {\em The Annals of Probability}, 46(1):207--260, 2018.

\bibitem{BM}
J.~Bertoin and S.~Mart{\'\i}nez.
\newblock Fragmentation energy.
\newblock {\em Advances in applied probability}, 37(2):553--570, 2005.

\bibitem{bingham1989regular}
N.~H. Bingham, C.~M. Goldie, and J.~L. Teugels.
\newblock {\em Regular variation.}, volume~27 of {\em Encycl. Math. Appl.}
\newblock Cambridge etc.: Cambridge University Press, paperback ed. edition,
  1989.

\bibitem{BD1}
M.~D. Brennan and R.~Durrett.
\newblock Splitting intervals.
\newblock {\em Ann. Probab.}, 14:1024--1036, 1986.

\bibitem{BD2}
M.~D. Brennan and R.~Durrett.
\newblock Splitting intervals. {II}: {Limit} laws for lengths.
\newblock {\em Probab. Theory Relat. Fields}, 75:109--127, 1987.

\bibitem{CR}
M.~Caballero and V.~Rivero.
\newblock {On the asymptotic behaviour of increasing self-similar Markov
  processes}.
\newblock {\em Electronic Journal of Probability}, 14(none):865 -- 894, 2009.

\bibitem{doney2}
R.~Doney.
\newblock On single-and multi-type general age-dependent branching processes.
\newblock {\em Journal of Applied Probability}, 13(2):239--246, 1976.

\bibitem{doney1}
R.~A. Doney.
\newblock A limit theorem for a class of supercritical branching processes.
\newblock {\em Journal of Applied Probability}, 9(4):707--724, 1972.

\bibitem{DGJPS}
P.~Dyszewski, N.~Gantert, S.~G.~G. Johnston, J.~Prochno, and D.~Schmid.
\newblock Sharp concentration for the largest and smallest fragment in a
  {{\(k\)}}-regular self-similar fragmentation.
\newblock {\em Ann. Probab.}, 50(3):1173--1203, 2022.

\bibitem{EscobedoMischlerPerthame2003}
M.~Escobedo, S.~Mischler, and B.~Perthame.
\newblock Gelation in coagulation and fragmentation models.
\newblock {\em Communications in Mathematical Physics}, 231(1):157--188, 2002.

\bibitem{evans1998stationary}
S.~N. Evans and J.~Pitman.
\newblock Stationary markov processes related to stable ornstein--uhlenbeck
  processes and the additive coalescent.
\newblock {\em Stochastic processes and their applications}, 77(2):175--185,
  1998.

\bibitem{Filippov1961}
A.~Filippov.
\newblock On the distribution of the sizes of particles which undergo
  splitting.
\newblock {\em Theory of Probability \& Its Applications}, 6(3):275--294, 1961.

\bibitem{fu}
A.~F. Filippov.
\newblock {\:U}ber das verteilungsgesetz der gr\:ossen der teilchen bei
  zerst\:uckelung.
\newblock {\em Teoriya Veroyatnostei i ee Primeneniya}, 6(3):299--318, 1961.

\bibitem{ghabache2016frozen}
E.~Ghabache, C.~Josserand, and T.~S{\'e}on.
\newblock Frozen impacted drop: from fragmentation to hierarchical crack
  patterns.
\newblock {\em Physical review letters}, 117(7):074501, 2016.

\bibitem{GH1}
C.~Goldschmidt and B.~Haas.
\newblock Behavior near the extinction time in self-similar fragmentations.
  {I}: {The} stable case.
\newblock {\em Ann. Inst. Henri Poincar{\'e}, Probab. Stat.}, 46(2):338--368,
  2010.

\bibitem{GH2}
C.~Goldschmidt and B.~Haas.
\newblock Behavior near the extinction time in self-similar fragmentations.
  {II}: {Finite} dislocation measures.
\newblock {\em Ann. Probab.}, 44(1):739--805, 2016.

\bibitem{grady1985geometric}
D.~Grady and M.~Kipp.
\newblock Geometric statistics and dynamic fragmentation.
\newblock {\em Journal of Applied Physics}, 58(3):1210--1222, 1985.

\bibitem{haas2021precise}
B.~Haas.
\newblock Precise asymptotics for the density and the upper tail of exponential
  functionals of subordinators.
\newblock {\em arXiv preprint arXiv:2106.08691}, 2021.

\bibitem{haas2023tail}
B.~Haas.
\newblock Tail asymptotics for extinction times of self-similar fragmentations.
\newblock In {\em Annales de l'Institut Henri Poincare (B) Probabilites et
  statistiques}, volume~59, pages 1722--1743. Institut Henri Poincar{\'e},
  2023.

\bibitem{haas2025ancestral}
B.~Haas and G.~Miermont.
\newblock Ancestral diversity in fragmentation trees.
\newblock {\em arXiv preprint arXiv:2512.15500}, 2025.

\bibitem{haas2008continuum}
B.~Haas, G.~Miermont, J.~Pitman, and M.~Winkel.
\newblock Continuum tree asymptotics of discrete fragmentations and
  applications to phylogenetic models.
\newblock {\em Annals of Probability}, 36(5):1790--1837, 2008.

\bibitem{Iksanov:2024:asymptotic}
A.~Iksanov, K.~Kolesko, and M.~Meiners.
\newblock Asymptotic fluctuations in supercritical {Crump}-{Mode}-{Jagers}
  processes.
\newblock {\em Ann. Probab.}, 52(4):1538--1606, 2024.

\bibitem{JP}
N.~C. Jain and W.~E. Pruitt.

\bibitem{Karamata1933}
J.~Karamata.
\newblock Sur un mode de croissance r{\'e}guli{\`e}re. th{\'e}or{\`e}mes
  fondamentaux.
\newblock {\em Bulletin de la Soci{\'e}t{\'e} Math{\'e}matique de France},
  61:55--62, 1933.

\bibitem{Kingman1978}
J.~F. Kingman.
\newblock The representation of partition structures.
\newblock {\em Journal of the London Mathematical Society}, 2(2):374--380,
  1978.

\bibitem{kolmogorov}
A.~N. Kolmogoroff.
\newblock {\"U}ber das logarithmisch normale {Verteilungsgesetz} der
  {Dimensionen} der {Teilchen} bei {Zerst{\"u}ckelung}.
\newblock {\em C. R. (Dokl.) Acad. Sci. URSS, n. Ser.}, 31:99--101, 1941.

\bibitem{kooij2021explosive}
S.~Kooij, G.~van Dalen, J.-F. Molinari, and D.~Bonn.
\newblock Explosive fragmentation of prince rupert’s drops leads to
  well-defined fragment sizes.
\newblock {\em Nature communications}, 12(1):2521, 2021.

\bibitem{lienau1936random}
C.~Lienau.
\newblock Random fracture of a brittle solid.
\newblock {\em Journal of the Franklin Institute}, 221(6):769--787, 1936.

\bibitem{mallard2016subduction}
C.~Mallard, N.~Coltice, M.~Seton, R.~D. M{\"u}ller, and P.~J. Tackley.
\newblock Subduction controls the distribution and fragmentation of earth’s
  tectonic plates.
\newblock {\em Nature}, 535(7610):140--143, 2016.

\bibitem{miermont2001ordered}
G.~Miermont.
\newblock Ordered additive coalescent and fragmentations associated to l{\'e}vy
  processes with no positive jumps.
\newblock {\em Electronic Journal of Probability}, 6(none):1 -- 33, 2001.

\bibitem{Mott2006}
N.~Mott and E.~Linfoot.
\newblock {\em A Theory of Fragmentation}, pages 207--225.
\newblock Springer Berlin Heidelberg, Berlin, Heidelberg, 2006.

\bibitem{moulinet2015popping}
S.~Moulinet and M.~Adda-Bedia.
\newblock Popping balloons: a case study of dynamical fragmentation.
\newblock {\em Physical review letters}, 115(18):184301, 2015.

\bibitem{nerman}
O.~Nerman.
\newblock On the convergence of supercritical general ({C}-{M}-{J}) branching
  processes.
\newblock {\em Z. Wahrscheinlichkeitstheor. Verw. Geb.}, 57:365--395, 1981.

\bibitem{ramesh2015review}
K.~Ramesh, J.~D. Hogan, J.~Kimberley, and A.~Stickle.
\newblock A review of mechanisms and models for dynamic failure, strength, and
  fragmentation.
\newblock {\em Planetary and Space Science}, 107:10--23, 2015.

\bibitem{rivero2003law}
V.~Rivero*.
\newblock A law of iterated logarithm for increasing self-similar markov
  processes.
\newblock {\em Stochastics and Stochastic Reports}, 75(6):443--472, 2003.

\bibitem{Smoluchowski1916}
M.~v. Smoluchowski.
\newblock Drei vortrage uber diffusion, brownsche bewegung und koagulation von
  kolloidteilchen.
\newblock {\em Zeitschrift fur Physik}, 17:557--585, 1916.

\bibitem{Stewart1989}
I.~Stewart and E.~Meister.
\newblock A global existence theorem for the general coagulation--fragmentation
  equation with unbounded kernels.
\newblock {\em Mathematical Methods in the Applied Sciences}, 11(5):627--648,
  1989.

\end{thebibliography}
\bibliographystyle{abbrv}

\end{document}